\documentclass{amsart}%

\makeatletter
\@namedef{subjclassname@2020}{%
  \textup{2020} Mathematics Subject Classification}
\makeatother

\usepackage{amsfonts}
\usepackage{amsmath}
\usepackage{amssymb}
\usepackage{graphicx}%

\usepackage{mathrsfs}
\renewcommand{\mathrm}{\mathscr}

\usepackage[colorlinks, linkcolor=blue,  citecolor=blue, urlcolor=blue, bookmarks=false]{hyperref}%
\hypersetup{pdfstartview=FitH}
\usepackage{url}
\usepackage{xspace}

\usepackage[numeric]{amsrefs}

\usepackage{dsfont}

\usepackage{verbatim}
\usepackage{alltt}
\usepackage{graphicx}
\graphicspath{{figures/}{../figuresOld/}}
\usepackage{caption}
\usepackage{subcaption}

\usepackage{amsthm, amsxtra}

\usepackage{enumitem}
\setlist[enumerate]{leftmargin=*, labelindent=0pt}

\setlist[itemize]{}

\usepackage{amsthm, amsxtra}

\numberwithin{equation}{section}

\newtheorem{theorem}{Theorem}[section]
\newtheorem{theorem*}{Theorem}
\newtheorem{remark}[theorem]{Remark}
\newtheorem{lemma}[theorem]{Lemma}

\newtheorem{corollary}[theorem]{Corollary}

\newtheorem*{question*}{Question}

\newtheorem*{conjecture*}{Conjecture}

\newtheorem{example}[theorem]{Example}

\newtheorem{definition}[theorem]{Definition}

\def\XXint#1#2#3{{\setbox0=\hbox{$#1{#2#3}{\int}$}
     \vcenter{\hbox{$#2#3$}}\kern-.5\wd0}}

\let\originalleft\left
\let\originalright\right
\renewcommand{\left}{\mathopen{}\mathclose\bgroup\originalleft}
\renewcommand{\right}{\aftergroup\egroup\originalright}

\author{Young-Heon Kim}
\address{University of British Columbia, Vancouver, Canada}
\email{yhkim@math.ubc.ca}

\author{Yuan Long Ruan}
\address{Beihang University, Beijing, China}
\email{ruanyl@buaa.edu.cn}

\subjclass[2020]{Primary 49, 60; secondary 52.}

\keywords{Optimal transport, Wasserstein projection, Stochastic order, Contraction, Expansion}

\thanks{YHK is partially supported by the
Natural Sciences and Engineering Research Council of Canada (NSERC) as well as  Exploration Grant  from the New Frontiers in Research Fund (NFRF). YLR is partially supported by the National Natural Science Foundation of China (NSFC).\\
\copyright 2021 by the author.
}

\begin{document}
\title[Wasserstein Projection]{Backward and Forward Wasserstein Projections in Stochastic Order}
\date{\today}

\begin{abstract}
We study metric projections onto cones in the Wasserstein space of probability measures, defined by stochastic orders. Dualities for backward and forward projections are established under general conditions. Dual optimal solutions and their characterizations require study on a case-by-case basis. Particular attention is given to convex order and subharmonic order. While backward and forward cones possess distinct geometric properties, strong connections between backward and forward projections can be obtained in the convex order case. Compared with convex order, the study of subharmonic order is subtler. In all cases, Brenier–Strassen type polar factorization theorems are proved, thus providing a full picture of the decomposition of optimal couplings between probability measures given by deterministic contractions (resp. expansions) and stochastic couplings. Our results extend to the forward convex order case the decomposition obtained by Gozlan and Juillet, which builds a connection with Caffarelli's contraction theorem. A further noteworthy addition to the early results is the decomposition in the subharmonic order case where the optimal mappings are characterized by volume distortion properties. To our knowledge, this is the first time in this occasion such results are available in the literature.
\end{abstract}
\maketitle

\setcounter{tocdepth}{1}
\tableofcontents

\section{Introduction}

Stochastic ordering of distributions is ubiquitous in probability and
statistics. Depending on its context of application, the order of
distributions are determined according to their behavior under a given group
$\mathcal{A}$ of test functions. Two probability measures $\mu,$ $\nu$ are
called increasing in a stochastic order defined by $\mathcal{A}$ if they
satisfy%
\[
\int\varphi d\mu\leqslant\int\varphi d\nu\text{ for all }\varphi\in
\mathcal{A}\text{.}%
\]
Convex order and subharmonic order are two frequently used stochastic orders
which corresponds respectively to $\mathcal{A}$\ being the set of convex
functions and subharmonic functions. The larger the test set $\mathcal{A}$ is,
the stronger the stochastic order becomes. So subharmonic order is more
restrictive than convex order. There are also many other widely used
stochastic orders, e.g. in one dimension, an increasing concave order is
defined by increasing concave functions \cite{johnson2018stochastic}.
Stochastic ordering of high dimensional distributions can also be defined
according to the orderings of their one dimensional projections
\cite{davidov2013linear}. Stochastic ordering of multiple distributions is
defined similarly \cite{el2005inferences}. We refer to
\cite{shaked2007stochastic} for a full account of various stochastic orders
and their applications in operations research and economics etc.

Stochastic order is a functional way to characterize the properties of
couplings between a pair of probability measures. Strassen theorem shows that
convex order relationship is equivalent to the existence of a martingale
coupling \cite{strassen1965existence}. This is generalized to subharmonic
order which is proved to be necessary and sufficient for the existence of a
Brownian martingale \cite{ghoussoub2020optimal}. Generalizations to other
stochastic orders are also considered \cite{bowles2019mather}. Strassen
theorem is the starting point of many recent studies on optimal martingale
transport and its applications in mathematical finance
\cite{beiglbock2013model, galichon2014stochastic}, Skorokhod embedding and
related topics \cite{beiglbock2017optimal, ghoussoub2019pde,
ghoussoub2021solution}.

Despite its wide applications, stochastic ordering is usually hard to
implement in practice. For one thing, sampling probability measures via naive
simulation is costly. For another, discretizing probabilities in a given
stochastic order is tricky, since stochastic order relation is usually
unstable, i.e. discretized measures could easily violate the original
stochastic order. This has been observed on the level of convex order. A large
amount of research has been devoted to the stability issues of optimal
martingale transport. Under some conditions in one dimension, it is stable
\cite{juillet2016stability, guo2019computational, backhoff2019stability}, but
not in high dimensions \cite{bruckerhoff2021instability}. These issues have
become a great hindrance to the numerical pursuits of stochastic orders.

As a general tool of sampling probability measures in stochastic order, we
propose to study Wasserstein projections onto the cones defined by a given
stochastic order. One such projection for convex order was employed by Gozlan
and Juillet \cite{gozlan2018characterization} to obtain a martingale version
of Brenier's polar factorization \cite{brenier1991polar}. Note that Brenier's
motivation for the investigation of polar factorization was to address the
instability issues in the numerical study of perfect incompressible fluids. We
intend to report numerical benefits offered by Wasserstein projections in a
separate article. In the current article, we instead build the necessary
mathematical framework required of downstream applications, and demonstrate
its uses in exploring the properties of optimal mappings between probability measures.

For any probability measures $\mu$, $\nu$ and cost function $c,$ we define the
Wasserstein transport cost as%

\[
\mathcal{T}_{c}\left(  \mu,\nu\right)  =\inf_{\pi\in\Pi\left(  \mu,\nu\right)
}\int_{X\times Y}c\left(  x,y\right)  d\pi\left(  x,y\right)  .
\]
We also write
\[
\mathcal{T}_{k}\left(  \mu,\nu\right)  =\mathcal{T}_{c}\left(  \mu,\nu\right)
\text{ if }c\left(  x,y\right)  =\left\vert x-y\right\vert ^{k},\text{
}k\geqslant1.
\]
and denote%
\[
\mathcal{W}_{k}\left(  \mu,\nu\right)  =\left(  \mathcal{T}_{k}\left(  \mu
,\nu\right)  \right)  ^{1/k}.
\]

Given a pair of probability measures $\mu$, $\nu$ and a stochastic order
determined by a function class $\mathcal{A}$, we study projections onto
backward cone with vertex $\nu$ and forward cone with vertex $\mu$.
Specifically, the backward cone $\mathrm{P}_{\leqslant\nu}^{\mathcal{A}}$ is
the set of probability measures less than $\nu$ w.r.t. the stochastic order
defined by $\mathcal{A}$. The forward cone $\mathrm{P}_{\mu\leqslant
}^{\mathcal{A}}$ is the set of probability measures greater than $\mu$ w.r.t.
the stochastic order defined by $\mathcal{A}$. The projection problems we have
in mind is defined w.r.t. the Wasserstein transport cost $\mathcal{T}_{c}$ for
a given cost function $c(x,y)$:%
\begin{equation}
\text{(backward projection) }\inf_{\bar{\mu}\in\mathrm{P}_{\leqslant\nu
}^{\mathcal{A}}}\mathcal{T}_{c}\left(  \mu,\bar{\mu}\right)  ,
\label{img_backproj}%
\end{equation}%
\begin{equation}
\text{(forward projection) }\inf_{\bar{\nu}\in\mathrm{P}_{\mu\leqslant
}^{\mathcal{A}}}\mathcal{T}_{c}\left(  \bar{\nu},\nu\right)  .
\label{img_forproj}%
\end{equation}
The word \textit{backward} emphasizes the fact that the projection we look for
locates in the direction "\textit{backward in time axis"} relative to the
vertex $\nu$ of the cone$.$ Similarly, the word \textit{forward} emphasizes
the fact that the projection we look for locates in the direction
"\textit{forward in time axis"} relative to the vertex $\mu$ of the cone$.$
These are illustrated in Figure \ref{fig:cones}.\begin{figure}[h]
\includegraphics[width=0.9\linewidth]{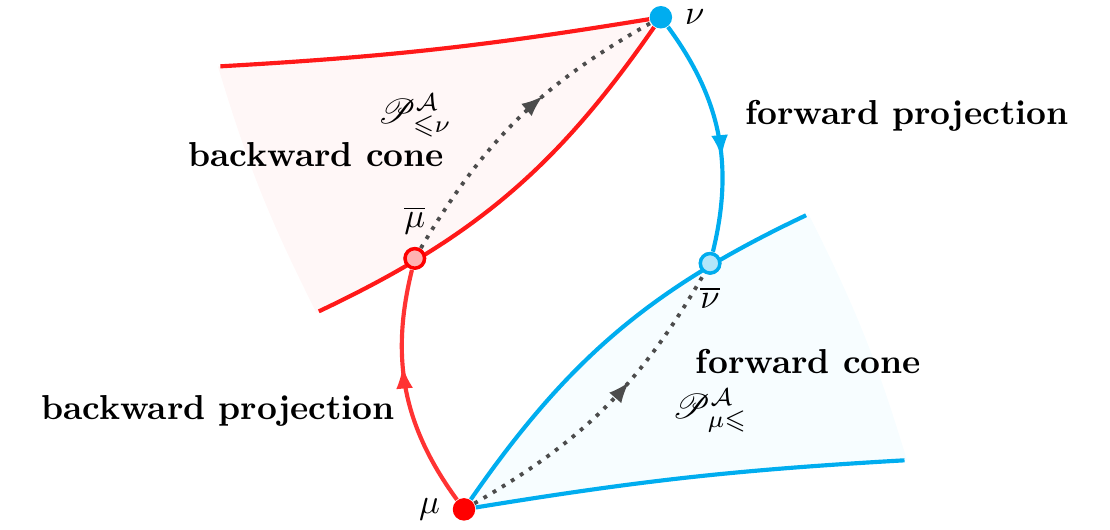}\caption{Illustration of
backward and forward projections onto cones defined by stochastic order
$\mathcal{A}$. Dotted line indicates the direction of increasing stochastic
order.}%
\label{fig:cones}%
\end{figure}

We first present the dual theorems for backward and forward projections.%

\begin{theorem*}%
\label{thm_mainI}Let $X,$ $Y$ be locally compact polish spaces. Given $\mu\in
P\left(  X\right)  ,$ $\nu\in P\left(  Y\right)  ,$ a cost function $c\left(
x,y\right)  $ and a defining function class $\mathcal{A}$ associated with a
stochastic order, the following dualities hold under appropriate conditions.

\begin{enumerate}
\item backward duality (Theorem \ref{thm_dual_bwg}):%
\begin{equation}
\inf_{\bar{\mu}\in\mathrm{P}_{\leqslant\nu}^{\mathcal{A}}}\mathcal{T}%
_{c}\left(  \mu,\bar{\mu}\right)  =\sup_{\left(  u,\varphi\right)  }\left\{
\int_{X}ud\mu-\int_{Y}\varphi d\nu\right\}  , \label{intro_dual_bk}%
\end{equation}
where $u\left(  x\right)  -\varphi\left(  y\right)  \leqslant c\left(
x,y\right)  $ with $u$ continuous and $\varphi\in\mathcal{A}.$

\item forward duality (Theorem \ref{thm_dual_fwg}):%
\begin{equation}
\inf_{\bar{\nu}\in\mathrm{P}_{\mu\leqslant}^{\mathcal{A}}}\mathcal{T}%
_{c}\left(  \bar{\nu},\nu\right)  =\sup_{\left(  \varphi,v\right)  }\left\{
\int_{X}\varphi d\mu-\int_{Y}vd\nu\right\}  , \label{intro_dual_fw}%
\end{equation}
where $\varphi\left(  x\right)  -v\left(  y\right)  \leqslant c\left(
x,y\right)  $ with $\varphi\in\mathcal{A}$ and $v$ continuous$.$
\end{enumerate}

%

\end{theorem*}%

The duality theorems for Wasserstein projections given by Theorem
\ref{thm_mainI} include as a special case the classical Kantorovich duality.
This happens when the stochastic order becomes \textit{degenerate}, see Remark
\ref{rmk_trivial}. Another special case is the duality for backward convex
order projection, this is previously proved by first establishing the
equivalence between backward convex order projection and the weak optimal
transport introduced in \cite{gozlan2017kantorovich}, and then using the
duality theorem for the weak optimal transport. In one dimension, the
equivalence is proved in \cite{gozlan2018characterization}, then generalized
to higher dimensions \cite{samson2017concentration} under the condition that
$\mu$ has a density w.r.t. the Lebesgue measure. The general case is proved in
\cite{gozlan2020mixture} and \cite{alfonsi2020sampling}. The duality for weak
optimal transport is proved in \cite{gozlan2017kantorovich} and
\cite{backhoff2019existence} via different approaches. In the compact case it
is also proved in \cite{alibert2019new} without using measurable selection theorems.

It is worth noting that in proving the backward duality $\left(
\ref{intro_dual_bk}\right)  $ in Theorem \ref{thm_mainI}, we do not rely on
weak optimal transport or other indirect reformulation via the method
described in the last paragraph. The approach taken here is direct and general
enough to handle all linear stochastic orders (Definition \ref{def_lin_order})
at the same time. The way we handle stochastic orders is more in line with the
perspective of Strassen theorem and has not been studied in the current context.

In appearance, the backward duality $\left(  \ref{intro_dual_bk}\right)
$\ and the forward duality $\left(  \ref{intro_dual_fw}\right)  $ look so
famililar that people might unconsciously fall under the illussion that
backward and forward projections are conjugate or even equivalent to each
other. But actually their relationship depends critically on the defining
class $\mathcal{A}$. Except in some special situations, there is no immediate
connection between backward and forward projection for general stochastic
orders, this will be explained in detail in section \ref{sec_bf}. In fact, in
the most promising case, i.e. the convex order case, the difference is already
promient in numerical computation \cite{alfonsi2020sampling} and it is
observed that, compared with the forward convex order projection, the backward
convex order projection is easier to manage due to its natural connection with
the weak optimal transport. Such a connection is not available for forward
convex order projection.

The defining class $\mathcal{A}$ determines the admissible functions of the
dual, thus gives properties of the projections specific to $\mathcal{A}$. We
now focus on the convex order and the subharmonic order where we obtain
interesting geometric properties of the optimal mappings such as contraction
and expansion. We summarize the main results about convex order projections as
below, notations are defined later.%

\begin{theorem*}%
\label{thm_main2}Let $\mu,$ $\nu\in P_{2}\left(  \mathbb{R}^{d}\right)  ,$
$c\left(  x,y\right)  =\left\vert x-y\right\vert ^{2}$. Denote by
$\mathcal{T}_{2}\left(  \mu,\mathrm{P}_{\leqslant\nu}^{\text{cx}}\right)  $,
$\mathcal{D}_{2}\left(  \mu,\mathrm{P}_{\leqslant\nu}^{\text{cx}}\right)  $
the optimal primal, dual values of the backward convex order projection, and
respectively $\mathcal{T}_{2}\left(  \mathrm{P}_{\mu\leqslant}^{\text{cx}}%
,\nu\right)  $, $\mathcal{D}_{2}\left(  \mathrm{P}_{\mu\leqslant}^{\text{cx}%
},\nu\right)  $ of the forward convex order projection. Then under appropriate
conditions we have the following.

\begin{enumerate}
\item The duality for backward convex order projection,%
\[
\mathcal{T}_{2}\left(  \mu,\mathrm{P}_{\leqslant\nu}^{\text{cx}}\right)
=\mathcal{D}_{2}\left(  \mu,\mathrm{P}_{\leqslant\nu}^{\text{cx}}\right)
\text{,}%
\]
and the duality for forward convex order projection,%
\[
\mathcal{T}_{2}\left(  \mathrm{P}_{\mu\leqslant}^{\text{cx}},\nu\right)
=\mathcal{D}_{2}\left(  \mathrm{P}_{\mu\leqslant}^{\text{cx}},\nu\right)
\text{.}%
\]

\item The optimal dual value for backward convex order projection is attained
(Theorem \ref{thm_dual_att_cx_bw}), and the optimal mapping from $\mu$ to the
unique projection is characterized by convex contraction (Definition
\ref{def_cx_contr}\ and Theorem \ref{thm_cx_bw_prop}).

\item The optimal dual value for forward convex order projection is attained
(Theorem \ref{thm_dual_att_cx}), and the optimal mapping from $\nu$ to the
unique projection is characterized by convex expansion (Definition
\ref{def_cx_exp}\ and Theorem \ref{thm_cx_prop}).

\item The optimal mappings for backward and forward convex order projections
are inverse to each other (Theorem \ref{thm_cx_bfequ} and Corollary
\ref{cor_cx_bfequ}).
\end{enumerate}

%

\end{theorem*}%

Item (1) is an instance of Theorem \ref{thm_mainI} enriched with desirable
traits for convex order projections (see Theorem \ref{thm_dual_cx}). Dual
attainment in item (2) and (3) are key results. Item (2) is originally given
by Gozlan and Juillet \cite{gozlan2020mixture}. This gives the
\textbf{backward decomposition}: given two probability measures, there is a
transport plan between them given by the gradient map of a convex contraction,
followed by a martingale coupling. This establishes a link with the celebrated
Caffarelli's contraction theorem \cite{caffarelli2000monotonicity} (see also
\cite{kim2012generalization} and \cite{fathi2019proof}) : if $\nu$ is a
log-concave perturbation of the Gaussian measure $\mu$, then the optimal
transport map from $\mu$ to $\nu$ is given by the gradient of a convex
function which is a contraction. In the language of item (2), the optimal map
is a contraction when the projection $\bar{\mu}$ to the backward cone is equal
to $\nu$. Item (3) is novel and it reinforces the link with Caffarelli's
contraction theorem by showing a \textbf{forward decomposition}: given two
probability measures, there is a transport plan between them given by a
martingale coupling followed by an expansion map which is the gradient of a
convex function. Item (3) is a natural companion to item (2), however, it is
not a straightforward result, since the backward and forward convex order cone
have distinct geometric properties, for example, one is geodesically convex in
the Wasserstein space $P_{2}\left(  \mathbb{R}^{d}\right)  $ while the other
is not; see section \ref{sec_bf}. The properties of backward and forward
mapping and their relation in item (4) are remarkable, however it seems to be
unique to the convex order case. In one dimension, these properties are
obtained by \cite{alfonsi2020sampling, backhoff2020weak} via methods specific
to one dimension.

Theorem \ref{thm_main2} immediately prompts one to ask whether similar results
hold for other stochastic orders, whether the projections can be characterized
in some special way, in particular giving a link to Caffarelli's contraction
type results. We are able to obtain a result similar to Theorem
\ref{thm_main2} for subharmonic order. Notice that subharmonic order is
stronger than the convex order, therefore the corresponding cones become much
smaller, for example, in $\mathbb{R}^{d}$ $(d\geqslant2)$, there is no
subharmonic order between discrete measures. This makes us expect weaker
properties for the projection mappings than in the convex order case. These
weaker properties are also natural in view of Theorem \ref{thm_mainI} since
the class $\mathcal{A}$ consists of subharmonic functions that are less
special than convex functions. Indeed, in the following theorem the
characterizations of the projection mappings are given by what we call
Laplacian contraction and expansion, which resembles a linearized version of
the convex contraction and expansion.%

\begin{theorem*}%
\label{thm_main3}Given $\mu,$ $\nu$ supported in a bounded smooth domain$,$
$c\left(  x,y\right)  =\left\vert x-y\right\vert ^{2}$. Denote by
$\mathcal{T}_{2}\left(  \mu,\mathrm{P}_{\leqslant\nu}^{\text{sh}}\right)  $,
$\mathcal{D}_{2}\left(  \mu,\mathrm{P}_{\leqslant\nu}^{\text{sh}}\right)  $
the optimal primal, dual values of the backward subharmonic order projection,
and respectively $\mathcal{T}_{2}\left(  \mathrm{P}_{\mu\leqslant}^{\text{sh}%
},\nu\right)  $, $\mathcal{D}_{2}\left(  \mathrm{P}_{\mu\leqslant}^{\text{sh}%
},\nu\right)  $ of the forward subharmonic order projection. Then under
appropriate conditions we have the following.

\begin{enumerate}
\item The duality for backward sbharmonic order projection$,$%
\[
\mathcal{T}_{2}\left(  \mu,\mathrm{P}_{\leqslant\nu}^{\text{sh}}\right)
=\mathcal{D}_{2}\left(  \mu,\mathrm{P}_{\leqslant\nu}^{\text{sh}}\right)
\text{,}%
\]
and the forward sbharmonic order projection,%
\[
\mathcal{T}_{2}\left(  \mathrm{P}_{\mu\leqslant}^{\text{sh}},\nu\right)
=\mathcal{D}_{2}\left(  \mathrm{P}_{\mu\leqslant}^{\text{sh}},\nu\right)  .
\]

\item The optimal dual value for backward subharmonic order projection is
attained (Theorem \ref{thm_dual_sh_bw}), and the optimal mapping from $\mu$ to
the unique projection is characterized by Laplacian contraction (Definition
\ref{def_Lap_contraction} and Theorem \ref{thm_sh_prop}).

\item The optimal dual value for forward subharmonic order projection is
attained (Theorem \ref{thm_dual_sh_fw}), and the optimal mapping from $\nu$ to
the unique projection is characterized by Laplacian expansion (Definition
\ref{def_Lap_expansion} and Theorem \ref{thm_sh_prop_fw}). Moreover, the
forward projection mapping, say $\nabla\bar{\psi}_{0}^{\ast}$, is a volume
increasing map. In particular, if $\nu$\ is absolutely continuous then the
forward projection $\bar{\nu}=\left(  \nabla\bar{\psi}_{0}^{\ast}\right)
_{\#}\nu$ is also absolutely continuous, and their densities (with the same
notation) satisfy%
\[
\bar{\nu}\left(  \nabla\bar{\psi}_{0}^{\ast}\left(  x\right)  \right)
\leqslant\nu\left(  x\right)  ,\text{ }a.e.\text{ }x.
\]

\end{enumerate}

%

\end{theorem*}%

This theorem seems to be the first occasion in the literature where a
connection is made between contraction type properties of optimal mappings and
the subharmonic order (thus with Brownian martingales). The volume increasing
property of the forward projection mapping in item (3) of Theorem
\ref{thm_main3} is remarkable, and it gives a weaker counterpart to the convex
order case. It seems a similar property is not available for the backward
subharmonic order projection. Item (2) and (3) show that optimal couplings
between two probability measures are composed by Laplacian contraction or
expansion and a Brownian martingale transport. They also raise a natural
question: when would the projection $\bar{\mu}$ of $\mu$ onto $\mathrm{P}%
_{\leqslant\nu}^{\text{sh}}$ be equal to $\nu,$ and respectively the
projection $\bar{\nu}$ of $\nu$ onto $\mathrm{P}_{\mu\leqslant}^{\text{sh}}$
be equal to $\mu$? Answering these questions may give Caffarelli's contraction
type results for a large class of measures beyond the known cases.

The dual attainment is usually nontrivial, there is no one-size-fits-all
approach to it. The situation for subharmonic order is subtler than convex
order. Convex functions enjoy many vital properties under $c$-transforms. But
almost all these properties break in the case of subharmonic order, this has
led to a series of difficulties both in the proof of attainment and
characterization of optimal mappings. Moreover, the crucial double
convexification trick is unfortunately unavailable for subharmonic order
projections. We will discuss more in section \ref{sec_att_sh}.

\noindent\textbf{Acknowledgement.} We would like to express our sincere thanks
to (alphabetically ordered) Luigi Ambrosio, Zhen-Qing Chen, Alessio Figalli,
Mathav Murugan and Edwin A. Perkins for the insightful discussions while the
paper was in preparation.

\noindent\textbf{Notation}:

$X:$ polish space. With slight abuse of notation, metric on this space is
written as $\left\Vert \cdot\right\Vert _{X},$ i.e., the distance from a
reference point which we do not explictly specify. If $X$ a (separable) Banach
space, then $\left\Vert \cdot\right\Vert _{X}$ is the norm on $X.$ When no
confusion arises$,$ we write $\left\Vert \cdot\right\Vert $ for simplicity.

$M\left(  X\right)  ,$ $M_{+}\left(  X\right)  :$ (nonnegative) finite Radon
measures on $X$.

$M_{k}\left(  X\right)  ,$ $M_{k,+}\left(  X\right)  :$ (nonnegative) measures
with finite $k$-th moment. $M_{0}\left(  X\right)  =M\left(  X\right)  ,$
$M_{0,+}\left(  X\right)  =M_{+}\left(  X\right)  .$

$P\left(  X\right)  ,$ $P^{ac}\left(  X\right)  :$ probability measures on $X$
(which are absolutely continuous w.r.t. the reference measure of $X$).

$P_{k}\left(  X\right)  ,$ $P_{k}^{ac}\left(  X\right)  :$ probability
measures in $P\left(  X\right)  ,$ $P^{ac}\left(  X\right)  $\ with finite
$k$-th moment. $P_{0}\left(  X\right)  =P\left(  X\right)  ,$ $P_{0}%
^{ac}\left(  X\right)  =P^{ac}\left(  X\right)  .$

$\Pi\left(  \mu,\nu\right)  :$ the set of \textbf{probability} couplings with
marginals $\mu$ and $\nu$.

$\mathcal{A}$ $:$ the defining class of a stochastic order, $\mathcal{A}%
_{\text{cx}}\mathcal{\ }$stands for the defining class of convex order and
$\mathcal{A}_{\text{sh}}$ for subharmonic order.

$\mu\leqslant_{\mathcal{A}}\nu:$ $\mu$ is smaller than $\nu$ in the stochastic
order defined by $\mathcal{A}$.

$\mathrm{P}_{k,\leqslant\nu}^{\mathcal{A}}:$ the \textit{backward cone},%
\[
\mathrm{P}_{k,\leqslant\nu}^{\mathcal{A}}=\left\{  \eta\in P_{k}\left(
Y\right)  :\eta\leqslant_{\mathcal{A}}\nu\right\}  .
\]
the subscripts $k$ are omitted and written $\mathrm{P}_{\leqslant\nu
}^{\mathcal{A}}$ if the underlying space is bounded.

$\mathrm{P}_{k,\mu\leqslant}^{\mathcal{A}}:$ the \textit{forward cone},%
\[
\mathrm{P}_{k,\mu\leqslant}^{\mathcal{A}}=\left\{  \xi\in P_{k}\left(
X\right)  :\mu\leqslant_{\mathcal{A}}\xi\right\}
\]
the subscripts $k$ are omitted and written $\mathrm{P}_{\mu\leqslant
}^{\mathcal{A}}$ if the underlying space is bounded.

$C_{0}^{\infty}\left(  X\right)  :$ the set of smooth functions with compact
support in $X.$

$C_{0}\left(  X\right)  :$ continuous functions which go to zero at
\textit{infinity}.

$C_{b}\left(  X\right)  :$ bounded continuous functions.

$\mathrm{S}_{b}\left(  X\right)  :$ bounded measurable functions.

$C_{b,k}\left(  X\right)  :$ $k\geqslant0,$ continuous functions with growth
no more than $\left\Vert x\right\Vert ^{k}$.

$C_{0,k}\left(  X\right)  :$ $k\geqslant0,$ continuous functions with
asymptotic order $o\left(  \left\Vert x\right\Vert ^{k}\right)  .$

$\mathrm{S}_{b,k}\left(  X\right)  :$ $k\geqslant0,$ measurable functions with
growth no more than $\left\Vert x\right\Vert ^{k}$.

$g^{\ast}:$ the Legendre-Fenchel dual of a function $g$.

$g_{e}:$ the subharmonic envelope of a function $g.$

\section{\label{sec_lin_order}Linear stochastic order}

The stochastic orders we are interested in are those which are characterized
by a class of admissible functions. We call them linear stochastic orders. The
term \textit{linear} is used to emphasize the sort of problems our method can
be applied to.

\begin{definition}
\label{def_lin_order}Let $\mathcal{A}$ be a nonempty class of measurable
functions which form a convex cone:

\begin{itemize}
\item[(1)] if $f\in\mathcal{A}$, then $af\in\mathcal{A}$, $\forall
a\geqslant0.$

\item[(2)] if $f,$ $g\in\mathcal{A}$, then $af+\left(  1-a\right)
g\in\mathcal{A}$, $\forall0\leqslant a\leqslant1.$
\end{itemize}

\noindent A measure $\mu\in M_{+}\left(  X\right)  $ is smaller than $\nu\in
M_{+}\left(  X\right)  $ in the linear stochastic order defined by
$\mathcal{A}$, denoted by $\mu\leqslant_{\mathcal{A}}\nu$, if $\mu$, $\nu$
have equal mass and
\begin{equation}
\int fd\mu\leqslant\int fd\nu\label{def_lin_order_1}%
\end{equation}
for all $f\in\mathcal{A}$ such that both integrals exist in the extended
sense. The class $\mathcal{A}$ is called the defining class of the associated
stochastic order.
\end{definition}

Note by definition $0\in\mathcal{A}$. Hereafter, with slight abuse of
notation, we will use \textit{a defining class} $\mathcal{A}$ to mean both the
function class itself and the stochastic order associated with it. We simply
call $\mathcal{A}$ a (linear) stochastic order. Here are a few examples of
linear stochastic orders.

\begin{example}
On the real line, the usual stochastic order is defined by the set of all
increasing functions. (Increasing) convex order corresponds to the set of all
(increasing) convex functions. These can also be generalized to consider the
so-called $m$-Convex order with $m\geqslant1$ being an integer$,$ defined by
the set of functions whose $m$-th derivative is nonnegative.
\end{example}

\begin{example}
In arbitrary dimension, two commonly encountered stochastic orders are convex
order and subharmonic order, which are respectively defined by the set of all
lower semicontinuous convex functions and the set of all subharmonic
functions. The two notions coincide in one dimension.
\end{example}

\begin{example}
\label{eg_trivial_order}When the defining class $\mathcal{A}$ is the set of
all bounded continuous functions, the associated stochastic order becomes
degenerate. In this special case, we call it a \textbf{trivial order}.
Measures in trivial order are identical.
\end{example}

In these examples, the set of admissible functions for which $\left(
\ref{def_lin_order_1}\right)  $ is equality contains nontrivial elements. For
$m$-Convex order, $\left(  \ref{def_lin_order_1}\right)  $ is equality for all
polynomials with degree no greater than $m-1$. For convex order, $\left(
\ref{def_lin_order_1}\right)  $ is equality for all for linear functions. For
subhamronic order, $\left(  \ref{def_lin_order_1}\right)  $ becomes equality
for all for harmonic functions. These could be useful for performing
normalizations on the admissible class.

The following representation of linear stochastic order is straightforward.

\begin{lemma}
\label{lm_stoch}Let $\mathcal{A}$ be a defining function class as defined in
Definition \ref{def_lin_order}. Assume that $\left\{  -1,1\right\}
\subset\mathcal{A}$. Let $\mu,$ $\nu\in M_{+}\left(  X\right)  $. Then%
\begin{equation}
\sup_{f\in\mathcal{A}}\left\{  \int_{X}fd\mu-\int_{X}fd\nu\right\}  =\left\{
\begin{array}
[c]{ll}%
0, & \mu\leqslant_{\mathcal{A}}\nu,\\
\infty, & \text{otherwise}.
\end{array}
\right.  \label{lm_stoch_1}%
\end{equation}

\end{lemma}

\begin{remark}
\label{rm_lm_stoch}The assumption $\left\{  -1,1\right\}  \subset\mathcal{A}$
ensures that%
\begin{equation}
\mu\text{ and }\nu\text{ are of equal mass.} \label{rm_lm_stoch_1}%
\end{equation}
In some situations, the constraint $\left(  \ref{rm_lm_stoch_1}\right)  $ can
be implied from other accompanying constraints. If this is the case, then we a
priori know that the measures have identical mass, thus the assumption
$\left\{  -1,1\right\}  \subset\mathcal{A}$ in Lemma \ref{lm_stoch} can be omitted.
\end{remark}

\subsection{Backward and forward projection}

We investigate two types of Wasserstein projections associated with a given
linear stochastic order $\mathcal{A}$.

\noindent\textbf{Backward projection}

Given $\mu\in P\left(  X\right)  $, $\nu\in P\left(  Y\right)  .$ Define the
backward convex cone%
\begin{equation}
\mathrm{P}_{\leqslant\nu}^{\mathcal{A}}=\left\{  \eta\in P\left(  Y\right)
:\eta\leqslant_{\mathcal{A}}\nu\right\}  \label{bw_cone}%
\end{equation}
consisting of measures in $P\left(  Y\right)  $ which are smaller than $\nu$
in the linear stochastic order $\mathcal{A}$. The backward projection of $\mu$
onto the cone $\mathrm{P}_{\leqslant\nu}^{\mathcal{A}}$ is defined as any
$\bar{\mu}\in\mathrm{P}_{\leqslant\nu}^{\mathcal{A}}$ which attains%
\[
\mathcal{T}_{c}\left(  \mu,\mathrm{P}_{\leqslant\nu}^{\mathcal{A}}\right)
\triangleq\inf_{\bar{\mu}\in\mathrm{P}_{\leqslant\nu}^{\mathcal{A}}%
}\mathcal{T}_{c}\left(  \mu,\bar{\mu}\right)  ,
\]
i.e., the transportation cost between $\mu$ and $\mathrm{P}_{\leqslant\nu
}^{\mathcal{A}}$.

\noindent\textbf{Forward projection}

The forward convex cone, denoted by%
\begin{equation}
\mathrm{P}_{\mu\leqslant}^{\mathcal{A}}=\left\{  \xi\in P\left(  X\right)
:\mu\leqslant_{\mathcal{A}}\xi\right\}  , \label{fw_cone}%
\end{equation}
is the set of measures in $P\left(  X\right)  $ which are greater than $\mu$
in the linear stochastic order $\mathcal{A}$. The forward projection of $\nu$
onto $\mathrm{P}_{\mu\leqslant}^{\mathcal{A}}$\ is any $\bar{\nu}\in
\mathrm{P}_{\mu\leqslant}^{\mathcal{A}}$ which attains%
\[
\mathcal{T}_{c}\left(  \mathrm{P}_{\mu\leqslant}^{\mathcal{A}},\nu\right)
\triangleq\inf_{\bar{\nu}\in\mathrm{P}_{\mu\leqslant}^{\mathcal{A}}%
}\mathcal{T}_{c}\left(  \bar{\nu},\nu\right)  ,
\]
i.e., the transportation cost between $\mathrm{P}_{\mu\leqslant}^{\mathcal{A}%
}$ and $\nu$.

\subsection{\label{subsec_compVSnon}Compact vs non-compact case}

When we mention the projection problems, we use the term \textit{compact case}
to mean the underlying spaces $X$ and $Y$ are compact, not just that the
measures $\mu$, $\nu$ have compact supports. Correspondingly, the term
\textit{non-compact case} (\textit{or general case}) means $X$ and $Y$ are not
necessarily compact.

It is important to note that, the set $\mathrm{P}_{\mu\leqslant}^{\mathcal{A}%
}$ given by $\left(  \ref{fw_cone}\right)  $ is a subset of $P\left(
X\right)  .$ All measures in $\mathrm{P}_{\mu\leqslant}^{\mathcal{A}}$ live in
$X.$\ Therefore, forward projection where $X$ is a proper subset of the
underlying space and forward projection where $X$ equal the underlying space
are different. Take $\mathbb{R}^{d},$ if $X$ is only a proper subset of
$\mathbb{R}^{d},$ then $\mathcal{T}_{c}\left(  \mathrm{P}_{\mu\leqslant
}^{\mathcal{A}},\nu\right)  $ optimizes over all admissible measures sitting
in $X.$ Measures not in $X$ are not admissible to the optimization$.$\ If $X$
equals $\mathbb{R}^{d},$ then $\mathcal{T}_{c}\left(  \mathrm{P}_{\mu
\leqslant}^{\mathcal{A}},\nu\right)  $ optimizes over all admissible measures
without restrictions on where they locate$.$ So, it is preferable to be aware
of this distinction when dealing with forward projection. This distinction for
backward projection does not exist however, see e.g. Lemma \ref{lm_cxsupp_inc}.

\section{Duality the compact case}

The rigorous proof of the duality in the general case, where the underlying
space is not necessarily compact, is delicate and requires additional
preparations. Compared with the general case, the compact case is less
restrictive in the assumptions and requires minimal preparations in the proof.
So we first deal with the compact case in this section. The general case will
be the topic of the next section. Starting with the compact case also helps
the reader grasp the main idea more easily.

\subsection{Backward projection}

The following simple lemma is very useful.

\begin{lemma}
\label{lm_marginal}Let $X,$ $Y$ be Polish spaces and $\mu\in P\left(
X\right)  .$ Then, for any $\xi\in M_{+}\left(  Y\right)  ,$ $\pi\in
M_{+}\left(  X\times Y\right)  ,$%
\[
\sup_{\left(  u,v\right)  }\left\{  \int ud\mu-\int vd\xi-\int\left(  u\left(
x\right)  -v\left(  y\right)  \right)  d\pi\left(  x,y\right)  \right\}
=\left\{
\begin{array}
[c]{ll}%
0, & \xi\in P\left(  Y\right)  ,\pi\in\Pi\left(  \mu,\xi\right)  ,\\
\infty, & \text{otherwise},
\end{array}
\right.
\]
where the supremum runs over $\left(  u,v\right)  \in C_{b}\left(  X\right)
\times C_{b}\left(  Y\right)  .$
\end{lemma}

We define the class of bounded measurable functions on $X,$%
\[
\mathrm{S}_{b}\left(  X\right)  =\left\{  u\text{ bounded, measurable}%
\right\}  .
\]

\begin{theorem}
\label{thm_dual_bw}Let $X,$ $Y$ be compact Polish spaces, $\mu\in P\left(
X\right)  $, $\nu\in P\left(  Y\right)  $, $c:X\times Y\mapsto\left[
0,\infty\right]  $ be lower semicontinuous and $\mathcal{A}$ be a defining
function class as defined in Definition \ref{def_lin_order}. Assume that
$\mathcal{A}$ and $\mathcal{A\cap}C_{b}$ define the same stochastic order.

(i) The backward duality holds%
\[
\mathcal{T}_{c}\left(  \mu,\mathrm{P}_{\leqslant\nu}^{\mathcal{A}}\right)
=\sup_{\left(  u,\varphi\right)  \in\mathcal{V}_{c}^{\ast}\cap C_{b}}\left\{
\int_{X}ud\mu-\int_{Y}\varphi d\nu\right\}  ,
\]
where $\mathcal{V}_{c}^{\ast}$ is the set of measurable functions $\left(
u,\varphi\right)  $ such that%
\begin{equation}
\varphi\in\mathcal{A}\text{ and }u\left(  x\right)  -\varphi\left(  y\right)
\leqslant c\left(  x,y\right)  ,\text{ }\forall x,y. \label{eq_Vc*}%
\end{equation}

(ii) The alternative form of backward duality holds%
\[
\mathcal{T}_{c}\left(  \mu,\mathrm{P}_{\leqslant\nu}^{\mathcal{A}}\right)
=\sup_{\varphi\in\mathcal{A\cap}C_{b}}\left\{  \int_{X}Q_{c}\left(
\varphi\right)  d\mu-\int_{Y}\varphi d\nu\right\}  ,
\]
where%
\begin{equation}
Q_{c}\left(  \varphi\right)  \left(  x\right)  =\inf_{y\in Y}\left\{
\varphi\left(  y\right)  +c\left(  x,y\right)  \right\}  . \label{eq_Qc}%
\end{equation}
In both forms (i) and (ii), $\varphi\in\mathcal{A\cap}C_{b}$ can be relaxed to
$\varphi\in\mathcal{A\cap}\mathrm{S}_{b}$.
\end{theorem}

\begin{remark}
Here we loosely write $\mathcal{V}_{c}^{\ast}\cap C_{b}$ to mean couples
$\left(  u,\varphi\right)  \in\mathcal{V}_{c}^{\ast}$ with all functions in
the slots belonging to $C_{b}.$ The same writing applies similarly through out
the article unless it is necessary to use full notation.
\end{remark}

\begin{proof}
The proof proceeds through several steps. In step 1 and step 2, we show that
\begin{equation}
\mathcal{T}_{c}\left(  \mu,\mathrm{P}_{\leqslant\nu}^{\mathcal{A}}\right)
=\sup_{\left(  u,v,\varphi\right)  \in\mathcal{V}_{c}\cap\left(  L^{1}%
\times\mathrm{S}_{b}\times\mathrm{S}_{b}\right)  }\left\{  \int_{X}ud\mu
-\int_{Y}\varphi d\nu\right\}  , \label{thm_dual_bw0}%
\end{equation}
where $\mathcal{V}_{c}$ is the collection of triples $\left(  u,v,\varphi
\right)  $ of measurable functions such that%
\begin{equation}
u\left(  x\right)  -v\left(  y\right)  \leqslant c\left(  x,y\right)  ,\text{
}\forall x,y. \label{eq_Vc1}%
\end{equation}
and%
\begin{equation}
\varphi\in\mathcal{A}\text{ and }v\left(  y\right)  \leqslant\varphi\left(
y\right)  ,\text{ }\forall y. \label{eq_Vc2}%
\end{equation}
The duality remains true if $\left(  \ref{eq_Vc1}\right)  $ is replaced with%
\begin{equation}
u\left(  x\right)  -v\left(  y\right)  \leqslant c\left(  x,y\right)  ,\text{
}\mu\text{-}a.e.\text{ }x,\text{ }\forall y. \label{eq_Vc1_1}%
\end{equation}
Upon obtaining this duality, we then conclude the proof of the theorem in step 3.

1. To prove $\left(  \ref{thm_dual_bw0}\right)  ,$ we first show that%
\[
\sup_{\mathcal{V}_{c}\cap C_{b}}\left\{  \int_{X}ud\mu-\int_{Y}\varphi
d\nu\right\}  \leqslant\sup_{\mathcal{V}_{c}\cap\left(  L^{1}\times
\mathrm{S}_{b}\times\mathrm{S}_{b}\right)  }\left\{  \int_{X}ud\mu-\int%
_{Y}\varphi d\nu\right\}  \leqslant\mathcal{T}_{c}\left(  \mu,\mathrm{P}%
_{\leqslant\nu}^{\mathcal{A}}\right)  .
\]
Only the second inequality needs explanation. For $\tilde{\mu}\in
\mathrm{P}_{\leqslant\nu}^{\mathcal{A}}$, $\pi\in\Pi\left(  \mu,\tilde{\mu
}\right)  ,$ $\left(  u,v,\varphi\right)  \in\mathcal{V}_{c}\cap\left(
L^{1}\times\mathrm{S}_{b}\times\mathrm{S}_{b}\right)  ,$ we have $\tilde{\mu
}\in P\left(  Y\right)  $ and $v,$ $\varphi$ are bounded, hence $\int
vd\tilde{\mu},$ $\int\varphi d\tilde{\mu}$ exist and are finite. Then we can
write%
\begin{align*}
\int_{X}ud\mu-\int_{Y}\varphi d\nu &  =\int_{X}ud\mu-\int_{Y}vd\tilde{\mu
}+\int_{Y}vd\tilde{\mu}-\int_{Y}\varphi d\nu\\
&  \leqslant\int_{X}ud\mu-\int_{Y}vd\tilde{\mu}+\int_{Y}\varphi d\tilde{\mu
}-\int_{Y}\varphi d\nu\\
&  \leqslant\int_{X}ud\mu-\int_{Y}vd\tilde{\mu}\\
&  \leqslant\int_{X\times Y}c\left(  x,y\right)  d\pi
\end{align*}
The first inequality uses $\left(  \ref{eq_Vc2}\right)  .$ The second
inequality uses the fact $\varphi\in\left(  \mathcal{A\cap}\mathrm{S}%
_{b}\right)  $ and the stochastis order relationship. The third inequality
uses $\left(  \ref{eq_Vc1}\right)  $. Now taking infimum over $\tilde{\mu}%
\in\mathrm{P}_{\leqslant\nu}^{\mathcal{A}},$ $\pi\in\Pi\left(  \mu,\tilde{\mu
}\right)  $ and supremum over $\left(  u,v,\varphi\right)  \in\mathcal{V}%
_{c}\cap\left(  L^{1}\times\mathrm{S}_{b}\times\mathrm{S}_{b}\right)  ,$ we
obtain%
\[
\sup_{\mathcal{V}_{c}\cap\left(  L^{1}\times\mathrm{S}_{b}\times\mathrm{S}%
_{b}\right)  }\left\{  \int_{X}ud\mu-\int_{Y}\varphi d\nu\right\}
\leqslant\mathcal{T}_{c}\left(  \mu,\mathrm{P}_{\leqslant\nu}^{\mathcal{A}%
}\right)  .
\]
Clearly we get the same result if $\left(  \ref{eq_Vc1}\right)  $ is replaced
with $\left(  \ref{eq_Vc1_1}\right)  ,$ since $\left(  \ref{eq_Vc1_1}\right)
$ implies%
\[
u\left(  x\right)  -v\left(  y\right)  \leqslant c\left(  x,y\right)  ,\text{
}\pi\text{-}a.e.\text{ }\left(  x,y\right)  ,
\]
for any $\tilde{\mu}\in\mathrm{P}_{\leqslant\nu}^{\mathcal{A}},$ $\pi\in
\Pi\left(  \mu,\tilde{\mu}\right)  .$

\textbf{2}. Next we prove the duality%
\begin{equation}
\mathcal{T}_{c}\left(  \mu,\mathrm{P}_{\leqslant\nu}^{\mathcal{A}}\right)
=\sup_{\left(  u,v,\varphi\right)  \in\mathcal{V}_{c}\cap C_{b}}\left\{
\int_{X}ud\mu-\int_{Y}\varphi d\nu\right\}  . \label{thm_dual_bw00}%
\end{equation}
For this, we introduce the functionals%
\[
\mathcal{G}:p\in C_{b}\left(  X\times Y\right)  \mapsto\left\{
\begin{array}
[c]{ll}%
0, & p\left(  x,y\right)  \geqslant-c\left(  x,y\right)  ,\\
\infty, & \text{otherwise}.
\end{array}
\right.
\]%
\[
\mathcal{H}:q\in C_{b}\left(  Y\right)  \mapsto\left\{
\begin{array}
[c]{ll}%
0, & q\left(  y\right)  \geqslant0,\\
\infty, & \text{otherwise}.
\end{array}
\right.
\]%
\begin{align*}
\mathcal{I}  &  :\left(  p,q\right)  \in C_{b}\left(  X\times Y\right)  \times
C_{b}\left(  Y\right) \\
&  \mapsto\left\{
\begin{array}
[c]{ll}%
{\displaystyle\int_{Y}}
\varphi d\nu-%
{\displaystyle\int_{X}}
ud\mu, &
\begin{array}
[c]{l}%
p\left(  x,y\right)  =v\left(  y\right)  -u\left(  x\right)  \text{ for some
}\left(  u,v\right)  \in C_{b},\\
q\left(  y\right)  =\varphi\left(  y\right)  -v\left(  y\right)  \text{ for
some }\varphi\in\mathcal{A\cap}C_{b}\text{,}%
\end{array}
\\
\infty, & \text{otherwise}.
\end{array}
\right.
\end{align*}
Note $\mathcal{I}$ is convex in view of the definition of $\mathcal{A}$.
$\mathcal{I}$ is\ well-defined, indeed, suppose
\[
p\left(  x,y\right)  =v_{1}\left(  y\right)  -u_{1}\left(  x\right)
=v_{2}\left(  y\right)  -u_{2}\left(  x\right)  ,\text{ }\forall x,\text{ }y,
\]%
\[
q\left(  y\right)  =\varphi_{1}\left(  y\right)  -v_{1}\left(  y\right)
=\varphi_{2}\left(  y\right)  -v_{2}\left(  y\right)  ,\text{ }\forall
x,\text{ }y,
\]
then, for some constant $a\in\mathbb{R}$,
\[
u_{1}=u_{2}-a,\text{ }v_{1}=v_{2}-a,\text{ }\varphi_{1}=\varphi_{2}-a.
\]
Hence%
\[
\int_{Y}\varphi_{1}\left(  y\right)  d\nu\left(  y\right)  -\int_{X}%
u_{1}\left(  x\right)  d\mu\left(  x\right)  =\int_{Y}\varphi_{2}\left(
y\right)  d\nu\left(  y\right)  -\int_{X}u_{2}\left(  x\right)  d\mu\left(
x\right)  ,
\]
showing that the definition of $\mathcal{I}$ does not depend on the way
$p\left(  x,y\right)  $ and $q\left(  x\right)  $ are split. Straightforward
calculations\ yield the Legendre transforms of the above functionals,%
\begin{align*}
\mathcal{G}^{\ast}\left(  -\pi\right)   &  =\sup_{p\in C_{b}\left(  X\times
Y\right)  ,\text{ }p\geqslant-c}\left\{  -\int_{X\times Y}p\left(  x,y\right)
d\pi\left(  x,y\right)  \right\} \\
&  =\left\{
\begin{array}
[c]{ll}%
\int_{X\times Y}c\left(  x,y\right)  d\pi\left(  x,y\right)  , & \pi\in
M_{+}\left(  X\times Y\right)  ,\\
\infty, & \text{otherwise}.
\end{array}
\right.
\end{align*}%
\[
\mathcal{H}^{\ast}\left(  -\bar{\mu}\right)  =\sup_{q\in C_{b}\left(
Y\right)  ,\text{ }q\geqslant0}\left\{  -\int_{Y}q\left(  y\right)  d\bar{\mu
}\left(  y\right)  \right\}  =\left\{
\begin{array}
[c]{ll}%
0, & \bar{\mu}\in M_{+}\left(  Y\right)  ,\\
\infty, & \text{otherwise}.
\end{array}
\right.  .
\]%
\begin{align*}
\mathcal{I}^{\ast}\left(  \pi,\bar{\mu}\right)   &  =\sup_{\substack{\left(
u,v\right)  \in C_{b}\\\varphi\in\mathcal{A\cap}C_{b}}}\left\{  \int_{X\times
Y}\left(  v\left(  y\right)  -u\left(  x\right)  \right)  d\pi\left(
x,y\right)  +\int_{X}u\left(  x\right)  d\mu\left(  x\right)  -\int_{Y}%
\varphi\left(  y\right)  d\nu\left(  y\right)  \right. \\
&  \left.  +\int_{Y}\varphi\left(  y\right)  d\bar{\mu}\left(  y\right)
-\int_{Y}v\left(  y\right)  d\bar{\mu}\left(  y\right)  \right\} \\
&  =\sup_{_{\substack{\left(  u,v\right)  \in C_{b}\\\varphi\in\mathcal{A\cap
}C_{b}}}}\left\{  \int_{X\times Y}\left(  v\left(  y\right)  -u\left(
x\right)  \right)  d\pi\left(  x,y\right)  +\int_{X}u\left(  x\right)
d\mu\left(  x\right)  -\int_{Y}v\left(  y\right)  d\bar{\mu}\left(  y\right)
\right. \\
&  \left.  +\int_{Y}\varphi\left(  y\right)  d\bar{\mu}\left(  y\right)
-\int_{Y}\varphi\left(  y\right)  d\nu\left(  y\right)  \right\} \\
&  =\left\{
\begin{array}
[c]{ll}%
0, & \bar{\mu}\in\mathrm{P}_{\leqslant\nu}^{\mathcal{A}},\text{ }\pi\in
\Pi\left(  \mu,\bar{\mu}\right)  ,\\
\infty, & \text{otherwise}.
\end{array}
\right.  \text{.}%
\end{align*}
In the calculation of $\mathcal{G}^{\ast}\left(  -\pi\right)  ,$ the
assumption on the cost function $c$ is used. The last equality in
$\mathcal{I}^{\ast}\left(  \pi,\bar{\mu}\right)  $ uses Lemma
\ref{lm_marginal}, Lemma \ref{lm_stoch} and the assumption that $\mathcal{A}$
and $\mathcal{A\cap}C_{b}$ define the same stochastic order. Let%
\[
\Theta\left(  p,q\right)  =\mathcal{G}\left(  p\right)  +\mathcal{H}\left(
q\right)  ,\text{ }\Xi\left(  p,q\right)  =\mathcal{I}\left(  p,q\right)  .
\]
Since $0\in\mathcal{A\cap}C_{b},$ we can set $\varphi\left(  y\right)
\equiv0,$ $v\left(  y\right)  \equiv-1,$ $u\left(  x\right)  \equiv-2$ and
define%
\[
p_{0}=v\left(  y\right)  -u\left(  x\right)  =1,\text{ }q_{0}=\varphi\left(
y\right)  -v\left(  y\right)  =1.
\]
Then $\left(  p_{0},q_{0}\right)  $ is in the effective domain of $\Theta$ and
$\Xi,$%
\[
\Theta\left(  p_{0},q_{0}\right)  =0<\infty,\text{ }\Xi\left(  p_{0}%
,q_{0}\right)  =2<\infty.
\]
Moreover $\Theta$ is obviously continuous at $\left(  p_{0},q_{0}\right)  .$
We are then in a position to invoke Fenchel-Rockafellar theorem \cite[Theorem
1.9]{villani2003topics},%
\[
\inf_{\left(  p,q\right)  }\left\{  \Theta\left(  p,q\right)  +\Xi\left(
p,q\right)  \right\}  =\sup_{\left(  \pi,\bar{\mu}\right)  }\left\{
-\Theta^{\ast}\left(  -\pi,-\bar{\mu}\right)  -\Xi\left(  \pi,\bar{\mu
}\right)  \right\}  .
\]
Plugging in early calculations, we obtain $\left(  \ref{thm_dual_bw00}\right)
,$ which together with step 1 proves the duality $\left(  \ref{thm_dual_bw0}%
\right)  $.

\textbf{3}. Finally we show that the duality $\left(  \ref{thm_dual_bw0}%
\right)  $ leads to the dualities of the theorem. It is easy to see that, for
any triple $\left(  u,v,\varphi\right)  \in\mathcal{V}_{c}\cap C_{b}$,
$\left(  u,\varphi\right)  $ is an admissible pair in $\mathcal{V}_{c}^{\ast
}\cap C_{b}.$ On the other hand, for any $\left(  u,\varphi\right)
\in\mathcal{V}_{c}^{\ast}\cap C_{b},$ the triple $\left(  u,\varphi
,\varphi\right)  $\ is in $\mathcal{V}_{c}\cap C_{b}$. Hence
\begin{equation}
\sup_{\left(  u,v,\varphi\right)  \in\mathcal{V}_{c}\cap C_{b}}\left\{
\int_{X}ud\mu-\int_{Y}\varphi d\nu\right\}  =\sup_{\left(  u,\varphi\right)
\in\mathcal{V}_{c}^{\ast}\cap C_{b}}\left\{  \int_{X}ud\mu-\int_{Y}\varphi
d\nu\right\}  . \label{thm_dual_bw1}%
\end{equation}
The LHS of $\left(  \ref{thm_dual_bw1}\right)  $\ equals $\mathcal{T}%
_{c}\left(  \mu,\mathrm{P}_{\leqslant\nu}^{\mathcal{A}}\right)  $ by step 2.
Hence the duality in (i) follows. To see the alternative duality form in (ii),
we first note that%
\[
\sup_{\left(  u,v,\varphi\right)  \in\mathcal{V}_{c}\cap C_{b}}\left\{
\int_{X}ud\mu-\int_{Y}\varphi d\nu\right\}  \leqslant\sup_{\varphi
\in\mathcal{A\cap}C_{b}}\left\{  \int_{X}Q_{c}\left(  \varphi\right)
d\mu-\int_{Y}\varphi d\nu\right\}  .
\]
On the other hand, following the idea of step 1, we have that%
\[
\sup_{\varphi\in\mathcal{A\cap}C_{b}}\left\{  \int_{X}Q_{c}\left(
\varphi\right)  d\mu-\int_{Y}\varphi d\nu\right\}  \leqslant\mathcal{T}%
_{c}\left(  \mu,\mathrm{P}_{\leqslant\nu}^{\mathcal{A}}\right)  .
\]
Therefore, in view of step 2, the proof of (ii) is completed.
\end{proof}

\subsection{Forward projection}

Now we derive the duality formula for forward projection.

\begin{theorem}
\label{thm_dual_fw}Let $X,$ $Y$ be compact, $\mu\in P\left(  X\right)  $,
$\nu\in P\left(  Y\right)  $, $c:X\times Y\mapsto\left[  0,\infty\right]  $ be
lower semicontinuous and $\mathcal{A}$ be a defining function class as defined
in Definition \ref{def_lin_order}. Assume that $\mathcal{A}$ and
$\mathcal{A\cap}C_{b}$ define the same stochastic order.

(i) The foward duality holds%
\[
\mathcal{T}_{c}\left(  \mathrm{P}_{\mu\leqslant}^{\mathcal{A}},\nu\right)
=\sup_{\left(  \varphi,v\right)  \in\mathcal{U}_{c}^{\ast}\cap C_{b}}\left\{
\int_{X}\varphi d\mu-\int_{Y}vd\nu\right\}  ,
\]
where $\mathcal{U}_{c}^{\ast}$ is the set of measurable functions $\left(
\varphi,v\right)  $ such that%
\begin{equation}
\varphi\in\mathcal{A}\text{ and }\varphi\left(  x\right)  -v\left(  y\right)
\leqslant c\left(  x,y\right)  ,\text{ }\forall x,y. \label{eq_Uc*}%
\end{equation}

(ii) The alternative form of forward duality holds%
\[
\mathcal{T}_{c}\left(  \mathrm{P}_{\mu\leqslant}^{\mathcal{A}},\nu\right)
=\sup_{\varphi\in\mathcal{A\cap}C_{b}}\left\{  \int_{X}\varphi d\mu-\int%
_{Y}Q_{\bar{c}}\left(  \varphi\right)  d\nu\right\}  ,
\]
where%
\begin{equation}
Q_{\bar{c}}\left(  \varphi\right)  \left(  y\right)  =\sup_{x\in X}\left\{
\varphi\left(  x\right)  -c\left(  x,y\right)  \right\}  . \label{eq_Qcbar}%
\end{equation}
In both forms (i) and (ii), $\varphi\in\mathcal{A\cap}C_{b}$ can be relaxed to
$\varphi\in\mathcal{A\cap}\mathrm{S}_{b}$.
\end{theorem}

\begin{proof}
The proof is similar to the backward case, Theorem \ref{thm_dual_bw}. We only
indicate the difference. As before, the proof of the theorem boils down to the
intermediate duality
\[
\mathcal{T}_{c}\left(  \mathrm{P}_{\mu\leqslant}^{\mathcal{A}},\nu\right)
=\sup_{\left(  \varphi,u,v\right)  \in\mathcal{V}_{c}\cap\left(
\mathrm{S}_{b}\times\mathrm{S}_{b}\times L^{1}\right)  }\left\{  \int%
_{X}\varphi d\mu-\int_{Y}vd\nu\right\}  ,
\]
where $\mathcal{U}_{c}$\ is the collection of triples $\left(  \varphi
,u,v\right)  $ of measurable functions such that%
\begin{equation}
u\left(  x\right)  -v\left(  y\right)  \leqslant c\left(  x,y\right)  ,\text{
}\forall x,y, \label{eq_Uc1}%
\end{equation}
and%
\begin{equation}
\varphi\in\mathcal{A}\text{ and }\varphi\left(  x\right)  \leqslant u\left(
x\right)  ,\text{ }\forall x. \label{eq_Uc2}%
\end{equation}
The duality remains true if $\left(  \ref{eq_Uc1}\right)  $ is replaced with%
\[
u\left(  x\right)  -v\left(  y\right)  \leqslant c\left(  x,y\right)  ,\text{
}\forall x,\text{ }\nu\text{-}a.e.\text{ }y.
\]
This is proved in 2 steps.

\textbf{1}. Similar to step 1 in the proof of Theorem \ref{thm_dual_bw}, we
have%
\[
\sup_{\mathcal{U}_{c}\cap C_{b}}\left\{  \int\varphi d\mu-\int vd\nu\right\}
\leqslant\sup_{\mathcal{U}_{c}\cap\left(  \mathrm{S}_{b}\times L^{1}\times
L^{1}\right)  }\left\{  \int\varphi d\mu-\int vd\nu\right\}  \leqslant
\mathcal{T}_{c}\left(  \mathrm{P}_{\mu\leqslant}^{\mathcal{A}},\nu\right)  .
\]

\textbf{2}. Next we have%
\[
\mathcal{T}_{c}\left(  \mathrm{P}_{\mu\leqslant}^{\mathcal{A}},\nu\right)
=\sup_{\left(  \varphi,u,v\right)  \in\mathcal{U}_{c}\cap C_{b}}\left\{
\int_{X}\varphi d\mu-\int_{Y}vd\nu\right\}  .
\]
This is proved by following the same lines as in step 2 of Theorem
\ref{thm_dual_bw}, except that the functional $\mathcal{I}$ is now defined as%
\begin{align*}
\mathcal{I}  &  :\left(  p,q\right)  \in C_{b}\left(  X\times Y\right)  \times
C_{b}\left(  X\right) \\
&  \mapsto\left\{
\begin{array}
[c]{ll}%
{\displaystyle\int_{Y}}
vd\nu-%
{\displaystyle\int_{X}}
\varphi d\mu, &
\begin{array}
[c]{l}%
p\left(  x,y\right)  =v\left(  y\right)  -u\left(  x\right)  \text{ for some
}\left(  u,v\right)  \in C_{b},\\
q\left(  x\right)  =u\left(  x\right)  -\varphi\left(  x\right)  \text{ for
some }\varphi\in\mathcal{A\cap}C_{b}\text{,}%
\end{array}
\\
\infty, & \text{otherwise}.
\end{array}
\right.
\end{align*}

\end{proof}

\section{Duality the general case}

In the previous section, we have proved the duality formulas for backward and
forward projection in the case the underlying spaces $X,$ $Y$ (resp. $X\times
Y$) are compact.

At first sight, one might think that the proof there can be carried out easily
to the general case, e.g. $\mathbb{R}^{d}$. The typical \textit{cutting and
gluing} technique would be the first to appear in our mind. A careful
thinking, however, reveals the difficulty in such an approach. The reason lies
in the special structure of measures in stochastic order, which makes it hard
to localize.\ Another obstacle comes from the function spaces. If we use
$C_{0}$ instead of $C_{b}$ in this case, the function $\mathcal{I}$ in Theorem
\ref{thm_dual_bw} and Theorem \ref{thm_dual_fw} would be useless, since the
decomposition of $p\left(  x,y\right)  $ as the sum of functions of individial
variables is possible only in a trivial way. The admissible set $\mathcal{A}$
adds a further layer of difficulty, it renders both the spaces $C_{0}$ and
$C_{b}$ useless in the general case, since those spaces are not large enough
to accommodate nontrivial test functions in $\mathcal{A}$. Consider for
instance the convex order defined by the set of all convex functions$,$ there
are no non-constant convex functions in $C_{b}\left(  \mathbb{R}^{d}\right)  $.

We introduce appropriate function spaces to address these issues.

\subsection{The space $C_{b,k}$ and its dual $\left(  C_{b,k}\right)  ^{\ast}%
$}

Let $X$ be a locally compact polish space. Let $k\geqslant0$ be an integer$.$
We introduce the function space $C_{b,k}\left(  X\right)  $ defined by%
\[
C_{b,k}\left(  X\right)  =\left\{  u\in C\left(  X\right)  :\frac
{u}{1+\left\Vert x\right\Vert ^{k}}\in C_{b}\left(  X\right)  \right\}  .
\]
with the norm%
\[
\left\Vert u\right\Vert _{b,k}=\sup_{x\in X}\frac{\left\vert u\left(
x\right)  \right\vert }{1+\left\Vert x\right\Vert ^{k}}.
\]
Then%
\[
C_{0,k}\left(  X\right)  =\left\{  u\in C\left(  X\right)  :\frac
{u}{1+\left\Vert x\right\Vert ^{k}}\in C_{0}\left(  X\right)  \right\}
\]
is a closed subspace of $C_{b,k}\left(  X\right)  $. Its topological dual
space is identified with the space of finite Borel measures with $k$-th
moment, i.e.%
\[
\left(  C_{0,k}\left(  X\right)  \right)  ^{\ast}\cong M_{k}\triangleq\left\{
\eta\in M\left(  X\right)  :\left(  1+\left\Vert x\right\Vert ^{k}\right)
\eta\in M\left(  X\right)  \right\}  .
\]
We denote by $M_{k,+}$ the set of nonnegative measures in $M_{k}$.

We also introduce%
\[
\mathrm{S}_{b,k}\left(  X\right)  =\left\{  u\text{ measurable}:\frac
{u}{1+\left\Vert x\right\Vert ^{k}}\in\mathrm{S}_{b}\left(  X\right)
\right\}  .
\]
Clearly when $k=0,$ $C_{b,0}\left(  X\right)  $ (resp. $C_{0,0}\left(
X\right)  ,$ $\mathrm{S}_{b,0}\left(  X\right)  $)\ reduces to the usual space
$C_{b}\left(  X\right)  $ (resp. $C_{0}\left(  X\right)  ,$ $\mathrm{S}%
_{b}\left(  X\right)  $). In addition, if $X$ is bounded, then $C_{0,k}\left(
X\right)  =C_{0}\left(  X\right)  ,$ $C_{b,k}\left(  X\right)  =C_{b}\left(
X\right)  ,$ $\mathrm{S}_{b,k}\left(  X\right)  =\mathrm{S}_{b}\left(
X\right)  ,$ for $k\geqslant0.$

The following crucial lemmas provide decomposition and representation of
continuous linear functionals on $C_{b,k}.$

\begin{lemma}
\label{lm_decmp_Cbk}Let $k\geqslant0$ $\ $be an integer and $X,$ $Y$ be
locally compact, $\sigma$-compact polish spaces$.$ Let $L$ be a nonnegative
continuous functional on $C_{b,k}\left(  X\times Y\right)  $. Then%
\[
L=\pi+R,
\]
where $\pi\in M_{k,+}\left(  X\times Y\right)  $ and $R$ a nonnegative
continuous linear functional supported at infinity, i.e.,%
\begin{equation}
\left\langle R,u\right\rangle =0,\text{ }\forall u\in C_{0,k}\left(  X\times
Y\right)  . \label{lm_decmp_Cbk1}%
\end{equation}

\end{lemma}

\begin{proof}
The proof is similar to \cite[Lemma 1.24]{villani2003topics}.
\end{proof}

\begin{lemma}
\label{lm_margin_Cbk}Let $X$, $Y$ be locally compact, $\sigma$-compact polish
spaces and $k\geqslant0$. Let $\mu\in M_{+}\left(  X\right)  $, $\nu\in
P_{k}\left(  Y\right)  $ be Borel probabilities. If $L\in\left(
C_{b,k}\left(  X\times Y\right)  \right)  ^{\ast}$ is nonnegative such that,
for all $u\in C_{b,k}\left(  X\right)  ,$ $v\in C_{b,k}\left(  Y\right)  ,$%
\begin{equation}
\left\langle L,u+v\right\rangle =\int_{X}u\left(  x\right)  d\mu+\int%
_{Y}v\left(  y\right)  d\nu, \label{lm_margin_Cbk1}%
\end{equation}
then $\mu\in P_{k}\left(  X\right)  ,$ $L\in P_{k}\left(  X\times Y\right)  $
and $L\in\Pi\left(  \mu,\nu\right)  .$
\end{lemma}

Note the assumption is that $\mu$ is a nonnegative measure. It is part of the
conclusion that $\mu$ is a probability of $k$-th moment.

\begin{proof}
\textbf{1}. When $k=0$ or both $X$ and $Y$ are bounded (thus $C_{b,k}$ reduces
to $C_{b}$), this is \cite[Lemma 1.25]{villani2003topics}.

\textbf{2}. Now consider $k\geqslant1$ and either $X$ or $Y$ is unbounded, say
$X$ is unbounded. To simplify notation, we assume w.l.g. that the origin $0$
is in $X$ and $Y$, so that we can use $0$ as the reference point for the
metrics on $X$ and $Y$. In vew of Lemma \ref{lm_decmp_Cbk}, we can write%
\[
L=\pi+R,
\]
where $\pi\in M_{k,+}\left(  X\times Y\right)  $ and $R$ a nonnegative
continuous linear functional supported at infinity in the sense of $\left(
\ref{lm_decmp_Cbk1}\right)  .$\ To complete the proof, it suffices to show
that $R=0$. Let $A_{n}$ be an increasing compact sets of $X$ such that%
\[
A_{n}\subset\text{int}\left(  A_{n+1}\right)  ,\text{ and }%
{\textstyle\bigcup\nolimits_{n}}
A_{n}=X.
\]
By Urysohn's Lemma, there are continuous functions $a_{n}\left(  x\right)  \in
C_{0}\left(  X\right)  $ satisfying%
\[
0\leqslant a_{n}\left(  x\right)  \leqslant1,\text{ }a_{n}\left(  x\right)
\text{ is }1\text{ on }A_{n},\text{ }0\text{ on }A_{n+1}^{c}.
\]
Clearly $\left\{  a_{n}\left(  x\right)  \right\}  $ is an increasing sequence
of functions and $a_{n}\left(  x\right)  \rightarrow1.$ Since $k\geqslant1$
and one of the underlying spaces is unbounded$,$ $C_{b}\left(  X\right)
\subset C_{0,k}\left(  X\times Y\right)  $. Therefore $R$ vanishes on
$C_{b}\left(  X\right)  $ by $\left(  \ref{lm_decmp_Cbk1}\right)  $. Taking
$u\in C_{b}\left(  X\right)  ,$ $v=0$ as test functions, we have that%
\[
\int_{X\times Y}u\left(  x\right)  d\pi=\left\langle L,u\right\rangle
=\int_{X}u\left(  x\right)  d\mu,\text{ }\forall u\in C_{b}\left(  X\right)
.
\]
If we substitute $u$ with%
\[
u_{n}\left(  x\right)  =a_{n}\left(  x\right)  \left\Vert x\right\Vert
_{X}^{k}\in C_{b}\left(  X\right)  ,\text{ }n\geqslant1,
\]
we get%
\[
\int_{X\times Y}u_{n}\left(  x\right)  d\pi=\int_{X}u_{n}\left(  x\right)
d\mu,\text{ }\forall n\geqslant1.
\]
It follows, by monotone convergence theorem, that%
\begin{equation}
\int_{X\times Y}\left\Vert x\right\Vert _{X}^{k}d\pi=\int_{X}\left\Vert
x\right\Vert _{X}^{k}d\mu. \label{lm_margin_Cbk2}%
\end{equation}
Since $\pi\in M_{k,+}\left(  X\times Y\right)  $, this indicates that $\mu$
has $k$-th moment, i.e., $\mu\in M_{k}\left(  X\right)  $. Analogously we have
$C_{b}\left(  Y\right)  \subset C_{0,k}\left(  X\times Y\right)  $ and%
\begin{equation}
\int_{X\times Y}\left\Vert y\right\Vert _{Y}^{k}d\pi=\int_{Y}\left\Vert
y\right\Vert _{Y}^{k}d\nu. \label{lm_margin_Cbk3}%
\end{equation}
Now taking $u=\left\Vert x\right\Vert _{X}^{k},$ $v=\left\Vert y\right\Vert
_{Y}^{k}$ as test functions in $\left(  \ref{lm_margin_Cbk1}\right)  $ we get%
\[
\int_{X\times Y}\left(  \left\Vert x\right\Vert _{X}^{k}+\left\Vert
y\right\Vert _{Y}^{k}\right)  d\pi+\left\langle R,\left\Vert x\right\Vert
_{X}^{k}+\left\Vert y\right\Vert _{Y}^{k}\right\rangle =\int_{X}\left\Vert
x\right\Vert _{X}^{k}d\mu+\int_{Y}\left\Vert y\right\Vert _{Y}^{k}d\nu.
\]
This together with $\left(  \ref{lm_margin_Cbk2}\right)  $ and $\left(
\ref{lm_margin_Cbk3}\right)  $ implies%
\[
\left\langle R,\left\Vert x\right\Vert _{X}^{k}+\left\Vert y\right\Vert
_{Y}^{k}\right\rangle =0.
\]
From these we conclude that $R=0.$ Indeed, for any $u\in C_{b,k}\left(
X\times Y\right)  ,$ there exist constants $a>0,$ $b>0$ such that%
\[
-a\left(  \left\Vert x\right\Vert _{X}^{k}+\left\Vert y\right\Vert _{Y}%
^{k}\right)  -b\leqslant u\left(  x,y\right)  \leqslant a\left(  \left\Vert
x\right\Vert _{X}^{k}+\left\Vert y\right\Vert _{Y}^{k}\right)  +b.
\]
Note $1\in C_{0,k}\left(  X\times Y\right)  $, therefore, when $R$ acts on the
first and last term of the above inequalities, they all vanish. Since $R$ is
nonnegative, we see that%
\[
\left\langle R,u-\left[  -a\left(  \left\Vert x\right\Vert _{X}^{k}+\left\Vert
y\right\Vert _{Y}^{k}\right)  -b\right]  \right\rangle \geqslant0\text{ which
yields }\left\langle R,u\right\rangle \geqslant0,
\]
and%
\[
\left\langle R,\left[  a\left(  \left\Vert x\right\Vert _{X}^{k}+\left\Vert
y\right\Vert _{Y}^{k}\right)  +b\right]  -u\right\rangle \geqslant0\text{
which yields }\left\langle R,u\right\rangle \leqslant0.
\]
Hence%
\[
\left\langle R,u\right\rangle =0,\text{ }\forall u\in C_{b,k}\left(  X\times
Y\right)  \text{.}%
\]

\textbf{3}. Now we know $L\in M_{k,+}\left(  X\times Y\right)  $, the
remaining proof is similar to Lemma \ref{lm_marginal}.
\end{proof}

\subsection{General duality for Wasserstein projections\label{sec_dual_g}}

There is one caveat before we prove the general duality. The stochastic cones
we defined earlier in $\left(  \ref{bw_cone}\right)  $ and $\left(
\ref{fw_cone}\right)  $ do not explicitly require moments of measures, this
will not cause problems in the compact case, since moments of the
probabilities exist automatically. However, we have to make this requirement
precise in the general case. So we retain similar notation of the cones, but
make it more precise for the general duality that the cones are contained in
the space of probability measures with $k$-th moment, i.e. for $\mu\in
P_{k}\left(  X\right)  ,$%
\[
\mathrm{P}_{k,\leqslant\nu}^{\mathcal{A}}=\left\{  \eta\in P_{k}\left(
Y\right)  :\eta\leqslant_{\mathcal{A}}\nu\right\}
\]
and for $\nu\in P_{k}\left(  Y\right)  ,$%
\[
\mathrm{P}_{k,\mu\leqslant}^{\mathcal{A}}=\left\{  \xi\in P_{k}\left(
X\right)  :\mu\leqslant_{\mathcal{A}}\xi\right\}  .
\]
Whenever the underlying spaces in question are bounded, the moments are not
relevent, the subscripts $k$ will be omitted so that they are consistent with
the notations introduced earlier.

\begin{theorem}
[\textbf{Backward duality}]\label{thm_dual_bwg}Let $X,$ $Y$ be locally compact
Polish spaces, $j\geqslant k\geqslant0$ be integers, $\mu\in P_{j}\left(
X\right)  ,$ $\nu\in P_{j}\left(  Y\right)  $ and $\mathcal{A}$ be a defining
function class as defined in Definition \ref{def_lin_order}. Assume that

(a1) $\mathcal{A}$ and $\mathcal{A\cap}C_{b,k}$ define the same stochastic order,

(a2) the cost $c\left(  x,y\right)  $ is nonnegative, lower semicontinuous and
there exist $\alpha>0$ and $\beta\in\mathbb{R}$ such that
\[
c\left(  x,y\right)  +\frac{1}{\alpha}\left\Vert x\right\Vert _{X}^{k}%
-\alpha\left\Vert y\right\Vert _{Y}^{k}\geqslant\beta,\text{ }\forall x,y.
\]

\noindent Then we have the following.

(i) Let $\mathcal{V}_{c}^{\ast}$ be defined in $\left(  \ref{eq_Vc*}\right)
,$ then
\[
\mathcal{T}_{c}\left(  \mu,\mathrm{P}_{k,\leqslant\nu}^{\mathcal{A}}\right)
=\sup_{\left(  u,\varphi\right)  \in\mathcal{V}_{c}^{\ast}\cap C_{b,k}%
}\left\{  \int_{X}ud\mu-\int_{Y}\varphi d\nu\right\}  .
\]

(ii) Let $Q_{c}\left(  \cdot\right)  $ be defined in $\left(  \ref{eq_Qc}%
\right)  ,$ then%
\[
\mathcal{T}_{c}\left(  \mu,\mathrm{P}_{k,\leqslant\nu}^{\mathcal{A}}\right)
=\sup_{\varphi\in\mathcal{A\cap}C_{b,k}}\left\{  \int_{X}Q_{c}\left(
\varphi\right)  d\mu-\int_{Y}\varphi d\nu\right\}  .
\]
In both (i) and (ii), $\varphi\in\mathcal{A\cap}C_{b,k}$ can be relaxed to
$\varphi\in\mathcal{A\cap}\mathrm{S}_{b,k}$.
\end{theorem}

\begin{proof}
Since the a probability measure on a Polish space has $\sigma$-compact
support, we may assume in the following that $X,$ $Y$ are $\sigma$-compact. We
would like to follow the steps as in Theorem \ref{thm_dual_bw}. The conclusion
of the theorem follows easily, once we can prove%
\begin{align*}
\mathcal{T}_{c}\left(  \mu,\mathrm{P}_{k,\leqslant\nu}^{\mathcal{A}}\right)
&  =\sup_{\left(  u,v,\varphi\right)  \in\mathcal{V}_{c}\cap C_{b,k}}\left\{
\int_{X}ud\mu-\int_{Y}\varphi d\nu\right\} \\
&  =\sup_{\left(  u,v,\varphi\right)  \in\mathcal{V}_{c}\cap\left(
L^{1}\times\mathrm{S}_{b,k}\times\mathrm{S}_{b,k}\right)  }\left\{  \int%
_{X}ud\mu-\int_{Y}\varphi d\nu\right\}
\end{align*}
where $\mathcal{V}_{c}$ is defined through $\left(  \ref{eq_Vc1}\right)
\left(  \ref{eq_Vc2}\right)  $. In contrast with Theorem \ref{thm_dual_bw},
the correct domain for the functionals $\Theta$ and $\Xi$ is $C_{b,k}\left(
X\times Y\right)  \times C_{0}\left(  Y\right)  .$ However, if we take a look
at the functionals $\mathcal{G}\left(  \cdot\right)  ,$ $\mathcal{H}\left(
\cdot\right)  $ defined in Theorem \ref{thm_dual_bw}, we would soon realize
that they are never continuous in $C_{0}\left(  Y\right)  $ with the supremum
norm induced by $C_{b}\left(  Y\right)  $. To fix this problem, we employ a
usual perturbation trick to circumvent this difficulty. Given $\epsilon_{1},$
$\epsilon_{2}>0,$\ we define%
\[
\mathcal{G}_{\epsilon_{1}}:p\in C_{b,k}\left(  X\times Y\right)
\mapsto\left\{
\begin{array}
[c]{ll}%
0, & p\left(  x,y\right)  \geqslant-c\left(  x,y\right)  -\epsilon_{1},\\
\infty, & \text{otherwise}.
\end{array}
\right.
\]%
\[
\mathcal{H}_{\epsilon_{2}}:q\in C_{0}\left(  Y\right)  \mapsto\left\{
\begin{array}
[c]{ll}%
0, & q\left(  y\right)  \geqslant-\epsilon_{2},\\
\infty, & \text{otherwise}.
\end{array}
\right.
\]%
\begin{align*}
\mathcal{I}  &  :\left(  p,q\right)  \in C_{b,k}\left(  X\times Y\right)
\times C_{0}\left(  Y\right) \\
&  \mapsto\left\{
\begin{array}
[c]{ll}%
{\displaystyle\int_{Y}}
\varphi d\nu-%
{\displaystyle\int_{X}}
ud\mu, &
\begin{array}
[c]{l}%
p\left(  x,y\right)  =v\left(  y\right)  -u\left(  x\right)  \text{ for some
}\left(  u,v\right)  \in C_{b,k},\\
q\left(  y\right)  =\varphi\left(  y\right)  -v\left(  y\right)  \text{ for
some }\varphi\in\mathcal{A\cap}C_{b,k},
\end{array}
\\
\infty, & \text{otherwise}.
\end{array}
\right.
\end{align*}
Note $\mathcal{I}$ is convex, well-defined and nontrivial. We have the
following,%
\[
\mathcal{G}_{\epsilon_{1}}^{\ast}\left(  -L\right)  =\sup_{p\in C_{b,k}\left(
X\times Y\right)  ,\text{ }p\geqslant-c-\epsilon_{1}}\left\{  -\left\langle
L,p\right\rangle \right\}  =\sup_{p\in C_{b,k}\left(  X\times Y\right)
,\text{ }p\geqslant c}\left\{  -\left\langle L,p\right\rangle \right\}
+\epsilon_{1},
\]
which is $\infty$ unless $L\in\left(  C_{b,k}\left(  X\times Y\right)
\right)  ^{\ast}$ is nonnegative$,$%
\[
\mathcal{H}^{\ast}\left(  -\bar{\mu}\right)  =\sup_{q\in C_{0}\left(
Y\right)  ,\text{ }q\geqslant-\epsilon_{2}}\left\{  -\int_{Y}q\left(
y\right)  d\bar{\mu}\right\}  =\left\{
\begin{array}
[c]{ll}%
\epsilon_{2}, & \bar{\mu}\in M_{+}\left(  Y\right)  ,\\
\infty, & \text{otherwise}.
\end{array}
\right.  .
\]%
\begin{align*}
&  \mathcal{I}^{\ast}\left(  L,\bar{\mu}\right) \\
&  =\sup_{_{\substack{\left(  u,v\right)  \in C_{b,k}\\\varphi\in
\mathcal{A\cap}C_{b,k}}}}\left\{  \left\langle L,v\left(  y\right)  -u\left(
x\right)  \right\rangle +\int_{X}u\left(  x\right)  d\mu\left(  x\right)
-\int_{Y}\varphi\left(  y\right)  d\nu\left(  y\right)  +\int_{Y}\left(
\varphi\left(  y\right)  -v\left(  y\right)  \right)  d\bar{\mu}\right\} \\
&  =\sup_{_{\substack{\left(  u,v\right)  \in C_{b,k}\\\varphi\in
\mathcal{A\cap}C_{b,k}}}}\left\{  \left\langle L,v\left(  y\right)  -u\left(
x\right)  \right\rangle -\int_{Y}v\left(  y\right)  d\bar{\mu}+\int%
_{X}u\left(  x\right)  d\mu\left(  x\right)  +\int_{Y}\varphi\left(  y\right)
d\bar{\mu}-\int_{Y}\varphi\left(  y\right)  d\nu\left(  y\right)  \right\}  .
\end{align*}
By virtue of Lemma \ref{lm_margin_Cbk}, $\mathcal{I}^{\ast}\left(  L,\bar{\mu
}\right)  =0$ if and only if $\bar{\mu}\leqslant_{\mathcal{A}}\nu$ and
$L\in\Pi\left(  \mu,\bar{\mu}\right)  .$ When this happens, we have $\bar{\mu
}\in P_{k}\left(  Y\right)  $ and $L\in P_{k}\left(  X\times Y\right)  $.
Therefore%
\[
\mathcal{I}^{\ast}\left(  L,\bar{\mu}\right)  =\left\{
\begin{array}
[c]{ll}%
0, & \bar{\mu}\in\mathrm{P}_{\leqslant\nu}^{\mathcal{A}},\text{ }L\in
\Pi\left(  \mu,\bar{\mu}\right)  ,\\
\infty, & \text{otherwise}.
\end{array}
\right.
\]
Let%
\[
\Theta_{\epsilon}\left(  p,q\right)  =\mathcal{G}_{\epsilon_{1}}\left(
p\right)  +\mathcal{H}_{\epsilon_{2}}\left(  q\right)  ,\text{ }\Xi\left(
p,q\right)  =\mathcal{I}\left(  p,q\right)  .
\]
To complete the proof, it remains to show that there exists $\left(
p_{0},q_{0}\right)  \in C_{b,k}\left(  X\times Y\right)  \times C_{0}\left(
Y\right)  $ such that%
\[
\Xi\left(  p_{0},q_{0}\right)  <\infty,\text{ }\Theta_{\epsilon}\left(
p_{0},q_{0}\right)  <\infty\text{ and }\Theta_{\epsilon}\text{ is continuous
at }\left(  p_{0},q_{0}\right)  .
\]
Define%
\[
\varphi_{0}\left(  y\right)  =0,\text{ }v_{0}\left(  y\right)  =0,\text{
}u_{0}\left(  x\right)  =-\frac{1}{a}\left\Vert x\right\Vert _{X}^{k}-b,
\]
where $a,$ $b$ are constants to be determined later. Clearly $\varphi_{0}%
\in\mathcal{A\cap}C_{b,k},$ $u_{0},$ $v_{0}\in C_{b,k}.$ Set%
\[
p_{0}\left(  x,y\right)  =v_{0}\left(  y\right)  -u_{0}\left(  x\right)
,\text{ }q_{0}\left(  y\right)  =\varphi_{0}\left(  y\right)  -v_{0}\left(
y\right)  ,\text{ }\forall x,y.
\]
Then%
\[
p_{0}=\frac{1}{a}\left\Vert x\right\Vert _{X}^{k}+b\in C_{b,k}\left(  X\times
Y\right)  ,\text{ }q_{0}\equiv0\in C_{0}\left(  Y\right)  .
\]
A $\delta$-neighbourhood $U_{\delta}\left(  p_{0},q_{0}\right)  $\ of $\left(
p_{0},q_{0}\right)  $ in $C_{b,k}\left(  X\times Y\right)  \times C_{0}\left(
Y\right)  $ is given by all functions $\left(  p,q\right)  \in C_{b,k}\left(
X\times Y\right)  \times C_{0}\left(  Y\right)  $ satisfying
\begin{align*}
&  \left\Vert p-p_{0}\right\Vert _{b,k}+\left\Vert q-q_{0}\right\Vert _{b}\\
&  =\sup_{\left(  x,y\right)  \in X\times Y}\frac{\left\vert p\left(
x,y\right)  -p_{0}\left(  x,y\right)  \right\vert }{1+\left(  \left\Vert
x\right\Vert _{X}+\left\Vert y\right\Vert _{Y}\right)  ^{k}}+\sup_{y\in
Y}\left\vert q\left(  y\right)  -q_{0}\left(  y\right)  \right\vert <\delta.
\end{align*}
It follows that, for any $\left(  p,q\right)  \in U_{\delta}\left(
p_{0},q_{0}\right)  $,%
\begin{align*}
p\left(  x,y\right)   &  \geqslant p_{0}-\delta-\delta2^{k-1}\left(
\left\Vert x\right\Vert _{X}^{k}+\left\Vert y\right\Vert _{Y}^{k}\right) \\
&  =\left(  \frac{1}{a}-\delta2^{k-1}\right)  \left\Vert x\right\Vert _{X}%
^{k}-\delta2^{k-1}\left\Vert y\right\Vert _{Y}^{k}+b-\delta,\text{ }\forall
x,y,
\end{align*}
and%
\[
q\left(  x,y\right)  \geqslant q_{0}-\delta=-\delta,\text{ }\forall x,y.
\]
Now choose $\epsilon_{1}>0,$ $\epsilon_{2}>0,$ $b>0$ large$,$ $a>0,$
$\delta>0$ small such that%
\[
-\delta>-\epsilon_{2},\text{ }\frac{1}{a}-\delta2^{k-1}\geqslant\frac
{1}{\alpha},\text{ }\alpha\geqslant\delta2^{k-1}\text{ and }b-\delta
>-\beta-\epsilon_{1}.
\]
With these constants and the assumption on the cost function $c$, we see that%
\[
p\geqslant-c-\epsilon_{1},\text{ }q\geqslant-\epsilon_{2},\text{ }%
\forall\left(  p,q\right)  \in U_{\delta}\left(  p_{0},q_{0}\right)  .
\]
Hence%
\[
\Theta_{\epsilon}\left(  p,q\right)  =0,\text{ }\forall\left(  p,q\right)  \in
U_{\delta}\left(  p_{0},q_{0}\right)  .
\]
Therefore $\Theta_{\epsilon}\left(  p_{0},q_{0}\right)  =0<\infty$ and
$\Theta_{\epsilon}$ is continuous at $\left(  p_{0},q_{0}\right)  .$ Moreover,
since $\mu,$ $\nu$ are probabilities with finite $k$-th moments,
\[
\Xi\left(  p_{0},q_{0}\right)  =\mathcal{I}\left(  p_{0},q_{0}\right)
<\infty.
\]
Now we can invoke Fenchel-Rockafellar theorem to get the conclusion of the
theorem. In particular, if the total transport cost $\mathcal{T}_{c}\left(
\mu,\mathrm{P}_{k,\leqslant\nu}^{\mathcal{A}}\right)  $ is finite, then we
obtain%
\[
\inf_{\substack{L\in P_{k}\left(  X\times Y\right)  ,\bar{\mu}\in P_{k}\left(
Y\right)  \\L\in\Pi\left(  \mu,\bar{\mu}\right)  ,\bar{\mu}\leqslant
_{\mathcal{A}}\nu}}\left\langle L,c+\epsilon_{1}\right\rangle +\epsilon
_{2}=\sup_{\substack{\left(  u,v\right)  \in C_{b,k},\varphi\in\mathcal{A\cap
}C_{b,k}\\u\left(  x\right)  -v\left(  y\right)  \leqslant c\left(
x,y\right)  +\epsilon_{1}\\\varphi\left(  y\right)  -v\left(  y\right)
\geqslant-\epsilon_{2}}}\left\{  \int_{X}ud\mu-\int_{Y}\varphi d\nu\right\}
.
\]
The RHS equals%
\[
\sup_{\substack{\left(  u,v\right)  \in C_{b,k},\varphi\in\mathcal{A\cap
}C_{b,k}\\u\left(  x\right)  -v\left(  y\right)  \leqslant c\left(
x,y\right)  +\epsilon_{1}+\epsilon_{2}\\\varphi\left(  y\right)  -v\left(
y\right)  \geqslant0}}\left\{  \int_{X}ud\mu-\int_{Y}\varphi d\nu\right\}
=\sup_{\left(  u,v,\varphi\right)  \in\mathcal{V}_{c}\cap C_{b,k}}\left\{
\int_{X}ud\mu-\int_{Y}\varphi d\nu\right\}  +\epsilon_{1}+\epsilon_{2}.
\]
Therefore the duality is proved upon cancelling $\epsilon_{1}+\epsilon_{2}$
from both sides.
\end{proof}

The duality for forward projection can be proved as Theorem \ref{thm_dual_bwg}%
. We leave its proof to the reader.

\begin{theorem}
[\textbf{Forward duality}]\label{thm_dual_fwg}Let $X,$ $Y$ be locally compact
Polish spaces, $j\geqslant k\geqslant0$ be integers, $\mu\in P_{j}\left(
X\right)  ,$ $\nu\in P_{j}\left(  Y\right)  $ and $\mathcal{A}$ be a defining
function class as defined in Definition \ref{def_lin_order}. Assume that

(a1) $\mathcal{A}$ and $\mathcal{A\cap}C_{b,k}$ define the same stochastic order,

(a2) the cost $c\left(  x,y\right)  $ is nonnegative, lower semicontinuous and
there exist $\alpha>0$ and $\beta\in\mathbb{R}$ such that
\[
c\left(  x,y\right)  +\frac{1}{\alpha}\left\Vert y\right\Vert _{X}^{k}%
-\alpha\left\Vert x\right\Vert _{Y}^{k}\geqslant\beta,\text{ }\forall x,y.
\]

\noindent Then we have the following.

(i) Let $\mathcal{U}_{c}^{\ast}$ be defined in $\left(  \ref{eq_Uc*}\right)
,$ then
\[
\mathcal{T}_{c}\left(  \mathrm{P}_{k,\mu\leqslant}^{\mathcal{A}},\nu\right)
=\sup_{\left(  \varphi,v\right)  \in\mathcal{U}_{c}^{\ast}\cap C_{b,k}%
}\left\{  \int_{X}\varphi d\mu-\int_{Y}vd\nu\right\}  .
\]

(ii) Let $Q_{\bar{c}}\left(  \cdot\right)  $ be defined in $\left(
\ref{eq_Qcbar}\right)  ,$ then%
\[
\mathcal{T}_{c}\left(  \mathrm{P}_{k,\mu\leqslant}^{\mathcal{A}},\nu\right)
=\sup_{\varphi\in\mathcal{A\cap}C_{b,k}}\left\{  \int_{X}\varphi d\mu-\int%
_{Y}Q_{\bar{c}}\left(  \varphi\right)  d\nu\right\}  .
\]
In both (i) and (ii), $\varphi\in\mathcal{A\cap}C_{b,k}$ can be relaxed
$\varphi\in\mathcal{A\cap}\mathrm{S}_{b,k}$.
\end{theorem}

The assumption (a2) on the cost function of the duality theorems (Theorem
\ref{thm_dual_bwg} and Theorem \ref{thm_dual_fwg}) might not be the most
general one, but it already includes important examples encountered in
applications. If the underlying spaces are bounded, then (a2) is automatic. In
the general case, the following example shows that all power functions satisfy
this assumption.

\begin{example}
Let $k\geqslant1$ be an integer$,$ $h\left(  s\right)  :\left[  0,\infty
\right)  \mapsto\left[  0,\infty\right)  $ be a continuous function such that%
\[
m_{h}\triangleq\inf_{s\geqslant a}\frac{h\left(  s\right)  }{s^{k}}>0\text{
for some }a>0.
\]
Then the cost function $c\left(  x,y\right)  =h\left(  \left\vert
x-y\right\vert \right)  $ on $\mathbb{R}^{d}\times\mathbb{R}^{d}$\ satisfies
assumption (a2) of Theorem \ref{thm_dual_bwg} (resp. Theorem
\ref{thm_dual_fwg}). In particular, the quadratic cost $c\left(  x,y\right)
=\left\vert x-y\right\vert ^{2}$ satisfies assumption (a2) with $k=1$ or $2.$
\end{example}

\begin{proof}
We show that assumption (a2) of Theorem \ref{thm_dual_bwg} is satisfied by the
cost function $c\left(  x,y\right)  =h\left(  \left\vert x-y\right\vert
\right)  $, the proof for Theorem \ref{thm_dual_fwg} is similar. Let
$0<\epsilon<1.$\ Consider%
\[
f\left(  x,y\right)  =h\left(  \left\vert x-y\right\vert \right)  +\frac
{1}{\epsilon}\left\vert x\right\vert ^{k}-\epsilon\left\vert y\right\vert
^{k}.
\]
With a change of variable $z=x-y,$ we can rewrite the function as%
\[
f\left(  z,y\right)  =h\left(  \left\vert z\right\vert \right)  +\frac
{1}{\epsilon}\left\vert z+y\right\vert ^{k}-\epsilon\left\vert y\right\vert
^{k}.
\]
Since we only need a lower bound of $f$ and it is bounded from below when
$y=0$ or $z=0$, we may assume $y\neq0$ and $z\neq0$ in the following$.$ Using
triangle inequality we have%
\[
f\left(  z,y\right)  \geqslant h\left(  \left\vert z\right\vert \right)
+\frac{1}{\epsilon}\left\vert \left\vert z\right\vert -\left\vert y\right\vert
\right\vert ^{k}-\epsilon\left\vert y\right\vert ^{k}=\left\vert y\right\vert
^{k}\left[  \frac{h\left(  \left\vert z\right\vert \right)  }{\left\vert
z\right\vert ^{k}}\frac{\left\vert z\right\vert ^{k}}{\left\vert y\right\vert
^{k}}+\frac{1}{\epsilon}\left\vert \frac{\left\vert z\right\vert }{\left\vert
y\right\vert }-1\right\vert ^{k}-\epsilon\right]  .
\]
We distinguish two scenarios:%
\[
\text{(s1) }\left\vert \frac{\left\vert z\right\vert }{\left\vert y\right\vert
}-1\right\vert >\frac{1}{2}\text{; (s2) }\left\vert \frac{\left\vert
z\right\vert }{\left\vert y\right\vert }-1\right\vert \leqslant\frac{1}%
{2}\text{ or }\frac{1}{2}\leqslant\frac{\left\vert z\right\vert }{\left\vert
y\right\vert }\leqslant\frac{3}{2}.
\]
In scenario (s1), we have%
\[
f\left(  z,y\right)  \geqslant\left\vert y\right\vert ^{k}\left[  \frac
{1}{2^{k}\epsilon}-\epsilon\right]  .
\]
In scenario (s2), if $\left\vert z\right\vert \geqslant a,$ then%
\[
f\left(  z,y\right)  \geqslant\left\vert y\right\vert ^{k}\left[  \frac{m_{h}%
}{2^{k}}-\epsilon\right]  .
\]
If $\left\vert z\right\vert <a,$ then $\left\vert y\right\vert \leqslant
2\left\vert z\right\vert <2a$, thus%
\[
f\left(  z,y\right)  \geqslant-\epsilon\left\vert y\right\vert ^{k}%
\geqslant-\epsilon\left(  2a\right)  ^{k}.
\]
Therefore we can choose $\epsilon>0$ small such that%
\[
\frac{m_{h}}{2^{k}}-\epsilon>0\text{ and }\frac{1}{2^{k}\epsilon}-\epsilon>0,
\]
then in all cases the function $f$ is bounded from below by $-\epsilon\left(
2a\right)  ^{k}$.
\end{proof}

We make a few remarks concerning backward projection Theorem
\ref{thm_dual_bwg} and forward projection Theorem \ref{thm_dual_fwg}.

\begin{remark}
\label{rmk_finiteness}Theorem \ref{thm_dual_bwg} and Theorem
\ref{thm_dual_fwg} do not assume the transportation costs $\mathcal{T}%
_{c}\left(  \mu,\mathrm{P}_{k,\leqslant\nu}^{\mathcal{A}}\right)  ,$
$\mathcal{T}_{c}\left(  \mathrm{P}_{k,\mu\leqslant}^{\mathcal{A}},\nu\right)
$ are finite, therefore, the optimal projections and optimal couplings are
generally nonnegative linear functionals on $C_{b}$ or $C_{b,k}$, as can be
seen from the proof. However, once the transportation costs are finite, then
they become true probabilties in $P_{k}$. Under mild conditions of the cost
function, the transportation costs are finite. Consider for example
$k\geqslant1,$ $\mu\in P_{k}\left(  \mathbb{R}^{d}\right)  ,$ $\nu\in
P_{k}\left(  \mathbb{R}^{d}\right)  $ and there is $A>0$ such that%
\[
0\leqslant c\left(  x,y\right)  \leqslant A\left(  \left\vert x\right\vert
^{k}+\left\vert y\right\vert ^{k}\right)  ,
\]
Since $\nu\in\mathrm{P}_{k,\leqslant\nu}^{\mathcal{A}}$,%
\begin{align*}
0  &  \leqslant\mathcal{T}_{c}\left(  \mu,\mathrm{P}_{k,\leqslant\nu
}^{\mathcal{A}}\right)  \leqslant\mathcal{T}_{c}\left(  \mu,\nu\right)
=\inf_{\pi\in\Pi\left(  \mu,\nu\right)  }\int_{\mathbb{R}^{d}\times
\mathbb{R}^{d}}c\left(  x,y\right)  d\pi\left(  x,y\right) \\
&  \leqslant A\int_{\mathbb{R}^{d}\times\mathbb{R}^{d}}\left(  \left\vert
x\right\vert ^{k}+\left\vert y\right\vert ^{k}\right)  d\pi\left(  x,y\right)
\text{ }\left(  \forall\pi\in\Pi\left(  \mu,\nu\right)  \right) \\
&  =A\left(  \int_{\mathbb{R}^{d}}\left\vert x\right\vert ^{k}d\mu
+\int_{\mathbb{R}^{d}}\left\vert y\right\vert ^{k}d\nu\right)  <\infty.
\end{align*}
The same estimate holds for forward projection $\mathcal{T}_{c}\left(
\mathrm{P}_{k,\mu\leqslant}^{\mathcal{A}},\nu\right)  $.
\end{remark}

\begin{remark}
\label{rmk_existence}It is part of Fenchel-Rockafellar theorem that, once the
optimal transportation costs are finite, then the optimal values
$\mathcal{T}_{c}\left(  \mu,\mathrm{P}_{\leqslant\nu}^{\mathcal{A}}\right)  ,$
$\mathcal{T}_{c}\left(  \mathrm{P}_{\mu\leqslant}^{\mathcal{A}},\nu\right)  ,$
$\mathcal{T}_{c}\left(  \mu,\mathrm{P}_{k,\leqslant\nu}^{\mathcal{A}}\right)
,$ $\mathcal{T}_{c}\left(  \mathrm{P}_{k,\mu\leqslant}^{\mathcal{A}}%
,\nu\right)  $ etc. are attained.
\end{remark}

\begin{remark}
\label{rmk_relax_phi}In the duality theorems, the requirement $\varphi
\in\mathcal{A\cap}\mathrm{S}_{b,k}$ cannot be further relaxed to $\varphi
\in\mathcal{A\cap}L^{1}$, because at some point of the proof, we will need to
integrate $\varphi$ w.r.t. to the free marginal (i.e. the projection). If
$\varphi$ is only known to be integrable w.r.t. the fixed marginal, then there
is not enough information to ensure the integration of $\varphi$ w.r.t. to the
free marginal exists. However, for backward projection, it is possible to
further relax $\mathcal{A\cap}\mathrm{S}_{b,k}$ to:%
\[
\left\{  \varphi:\varphi\in\mathcal{A\cap}L^{1}\left(  X,d\nu\right)  \text{
and }\varphi\geqslant f_{\varphi}\text{ for some }f_{\varphi}\in
\mathrm{S}_{b,k}\right\}  .
\]
When this relaxed condition holds, then, for any $\tilde{\mu}\in
\mathrm{P}_{k,\leqslant\nu}^{\mathcal{A}},$ $\tilde{\mu}\leqslant
_{\mathcal{A}}\nu,$ we have%
\[
\int f_{\varphi}d\tilde{\mu}\leqslant\int\varphi d\tilde{\mu}\leqslant
\int\varphi d\nu<\infty,
\]
which implies $\int\varphi d\tilde{\mu}$ is finite. Similarly, for forward
projection, it is possible to further relax $\mathcal{A\cap}\mathrm{S}_{b,k}$
to:%
\[
\left\{  \varphi:\varphi\in\mathcal{A\cap}L^{1}\left(  Y,d\mu\right)  \text{
and }\varphi\leqslant g_{\varphi}\text{ for some }g_{\varphi}\in
\mathrm{S}_{b,k}\right\}  .
\]

\end{remark}

\begin{remark}
\label{rmk_trivial}If $\mathcal{A}$ defines a trivial order (see Example
\ref{eg_trivial_order}) and $k=0$, then the dualities of Wasserstein
projections reduce to the classical Kantorovich duality.
\end{remark}

\begin{remark}
\label{rmk_adm_equv}If the defining classes $\mathcal{A}_{1},$ $\mathcal{A}%
_{2}$, $\mathcal{A}_{1}\mathcal{\cap}\mathrm{S}_{b,k}$ $\mathcal{A}%
_{2}\mathcal{\cap}\mathrm{S}_{b,k}$ define the same stochastic order$,$ then
they give equal optimal dual value. Take backward projection for exmaple,%
\[
\sup_{\varphi\in\mathcal{A}_{1}\mathcal{\cap}\mathrm{S}_{b,k}}\left\{
\int_{X}Q_{c}\left(  \varphi\right)  d\mu-\int_{Y}\varphi d\nu\right\}
=\sup_{\varphi\in\mathcal{A}_{2}\mathcal{\cap}\mathrm{S}_{b,k}}\left\{
\int_{X}Q_{c}\left(  \varphi\right)  d\mu-\int_{Y}\varphi d\nu\right\}  .
\]
This follows from the proof of the duality formula and the fact that these
defining classes produce the same stochastic order cones, i.e. $\mathrm{P}%
_{k,\leqslant\nu}^{\mathcal{A}_{1}}=\mathrm{P}_{k,\leqslant\nu}^{\mathcal{A}%
_{2}}$. The same is true for forward projections. It is worth noting that it
is not immediately obvious that running the supremum over different sets
should result in equal optimal dual value.
\end{remark}

The following theorem gives the relationship between optimal primal solutions
and optimal dual solutions of Wasserstein projections.

\begin{theorem}
\label{thm_aproperty}Let $X,$ $Y$ be locally compact Polish spaces,
$k\geqslant0$ be an integer, $\mu\in P_{k}\left(  X\right)  ,$ $\nu\in
P_{k}\left(  Y\right)  $ and $\mathcal{A}$ be a defining function class as
defined in Definition \ref{def_lin_order}.

(i) Assume that the conditions of Theorem \ref{thm_dual_bwg} are satisfied.
Let $\bar{\mu}$ be the optimizer for $\mathcal{T}_{c}\left(  \mu
,\mathrm{P}_{k,\leqslant\nu}^{\mathcal{A}}\right)  $ and suppose $\varphi
\in\mathcal{A\cap}L^{1}\left(  d\nu\right)  $ is an optimal dual solution for
backward projection of Theorem \ref{thm_dual_bwg} which is bounded from below
by some function in $\mathrm{S}_{b,k}$, then%
\[
\int\varphi d\bar{\mu}=\int\varphi d\nu.
\]
In particular $\varphi$ is an optimal potential for $\mathcal{T}_{c}\left(
\mu,\bar{\mu}\right)  .$

(ii) Assume that the conditions of Theorem \ref{thm_dual_fwg} are satisfied.
Let $\bar{\nu}$ be the optimizer for $\mathcal{T}_{c}\left(  \mathrm{P}%
_{k,\mu\leqslant}^{\mathcal{A}},\nu\right)  $ and suppose $\varphi
\in\mathcal{A\cap}L^{1}\left(  d\mu\right)  $ is an optimal dual solution for
forward projection of Theorem \ref{thm_dual_fwg} which is bounded from above
by some function in $\mathrm{S}_{b,k}$, then%
\[
\int\varphi d\mu=\int\varphi d\bar{\nu}.
\]
In particular $\varphi$ is an optimal potential for $\mathcal{T}_{c}\left(
\bar{\nu},\nu\right)  .$
\end{theorem}

\begin{proof}
We only prove (i), the proof of (ii) is similar. In view of Remark
\ref{rmk_relax_phi}, the integral $\int\varphi d\bar{\mu}$ is finite. Then
using the optimality of $\bar{\mu}$ and $\varphi$, we obtain
\begin{align*}
\mathcal{T}_{c}\left(  \mu,\bar{\mu}\right)   &  =\mathcal{T}_{c}\left(
\mu,\mathrm{P}_{k,\leqslant\nu}^{\mathcal{A}}\right)  =\int Q_{c}\left(
\varphi\right)  d\mu-\int\varphi d\nu\\
&  =\int Q_{c}\left(  \varphi\right)  d\mu-\int\varphi d\bar{\mu}+\int\varphi
d\bar{\mu}-\int\varphi d\nu\\
&  \leqslant\mathcal{T}_{c}\left(  \mu,\bar{\mu}\right)  +\left(  \int\varphi
d\bar{\mu}-\int\varphi d\nu\right)  \leqslant\mathcal{T}_{c}\left(  \mu
,\bar{\mu}\right)  .
\end{align*}
The first inequality is due to the fact that $\left(  Q_{c}\left(
\varphi\right)  ,\varphi\right)  $ is an admissible pair for the
transportation $\mathcal{T}_{c}\left(  \mu,\bar{\mu}\right)  .$ The second
inequality follows from the fact that $\bar{\mu}\leqslant_{\mathcal{A}}\nu.$
\end{proof}

\subsection{\label{sec_unique}Uniqueness}

Now we turn to the uniqueness of the projections. If the cones $\mathrm{P}%
_{k,\mu\leqslant}^{\mathcal{A}}$, $\mathrm{P}_{\leqslant\nu}^{\mathcal{A}}%
$\ are convex along Wasserstein geodesics, then the uniqueness would be
immediate. However, except in a few special situations, these cones are
generally not convex along Wasserstein geodesics. That being said, the
uniqueness of the projection can still be obtained using the convexity of
these cones under linear interpolation. The case of convex order projection is
proved in \cite{alfonsi2020sampling}. The proof follows the classical strict
convexity argument.

\begin{theorem}
[Uniqueness]\label{thm_uniqueness}Let $k\geqslant0$, $X,$ $Y$ be convex
subsets of $\mathbb{R}^{d},$ $\mu\in P_{k}\left(  X\right)  ,$ $\nu\in
P_{k}\left(  Y\right)  $ and $\mathcal{A}$ be a defining function class as
defined in Definition \ref{def_lin_order}. The cost function $c\left(
x,y\right)  =h\left(  x-y\right)  $ for some strictly convex function
$h:\mathbb{R}^{d}\mapsto\left[  0,\infty\right)  .$ If $\mu\in P_{k}%
^{ac}\left(  X\right)  $ and the transport cost of the backward projection is
finite, then the projection exists and is unique. Similarly, if $\nu\in
P_{k}^{ac}\left(  Y\right)  ,$ then forward projection exists and is unique.
\end{theorem}

\begin{proof}
We consider forward projection and $X=Y=\mathbb{R}^{d}$, other cases are
proved similarly. Assume that $\nu\in P_{k}^{ac}\left(  \mathbb{R}^{d}\right)
$ and the transport cost of the forward projection\ is finite. In view of
Remark \ref{rmk_existence}, the forward projection exists. Let$\ \bar{\nu}%
_{0}$, $\bar{\nu}_{1}$\ be forward projections of $\nu$ onto $\mathrm{P}%
_{k,\mu\leqslant}^{\mathcal{A}}$, i.e.%
\[
\mathcal{T}_{c}\left(  \bar{\nu}_{0},\nu\right)  =\mathcal{T}_{c}\left(
\bar{\nu}_{1},\nu\right)  =\mathcal{T}_{c}\left(  \mathrm{P}_{k,\mu\leqslant
}^{\mathcal{A}},\nu\right)  .
\]
Denote by $\pi_{i}$ $\left(  i=0,1\right)  $ be the optimal coupling between
$\bar{\nu}_{i}$ and $\nu.$ Then, for $s\in\left(  0,1\right)  $ fixed$,$
\begin{equation}
\pi_{s}\triangleq\left(  1-s\right)  \pi_{0}+s\pi_{1}, \label{thm_uniqueness1}%
\end{equation}
is a coupling between $\bar{\nu}_{s}\triangleq\left(  1-s\right)  \bar{\nu
}_{0}+s\bar{\nu}_{1}$ and $\nu.$ Since $\mathrm{P}_{k,\mu\leqslant
}^{\mathcal{A}}$ is convex under linear interpolation, $\bar{\nu}_{s}%
\in\mathrm{P}_{k,\mu\leqslant}^{\mathcal{A}}.$ Hence%
\begin{align*}
\mathcal{T}_{c}\left(  \mathrm{P}_{k,\mu\leqslant}^{\mathcal{A}},\nu\right)
&  \leqslant\mathcal{T}_{c}\left(  \bar{\nu}_{s},\nu\right)  \leqslant
\int_{\mathbb{R}^{d}\times\mathbb{R}^{d}}c\left(  x,y\right)  d\pi_{s}\\
&  =\left(  1-s\right)  \int_{\mathbb{R}^{d}\times\mathbb{R}^{d}}c\left(
x,y\right)  d\pi_{0}+s\int_{\mathbb{R}^{d}\times\mathbb{R}^{d}}c\left(
x,y\right)  d\pi_{1}\\
&  =\left(  1-s\right)  \mathcal{T}_{c}\left(  \bar{\nu}_{0},\nu\right)
+s\mathcal{T}_{c}\left(  \bar{\nu}_{1},\nu\right)  =\mathcal{T}_{c}\left(
\mathrm{P}_{k,\mu\leqslant}^{\mathcal{A}},\nu\right)  .
\end{align*}
It follows that%
\[
\mathcal{T}_{c}\left(  \mathrm{P}_{k,\mu\leqslant}^{\mathcal{A}},\nu\right)
=\mathcal{T}_{c}\left(  \bar{\nu}_{s},\nu\right)  =\int_{\mathbb{R}^{d}%
\times\mathbb{R}^{d}}c\left(  x,y\right)  d\pi_{s},
\]
i.e. $\bar{\nu}_{s}$ is a forward projection of $\nu$ onto $\mathrm{P}%
_{k,\mu\leqslant}^{\mathcal{A}}$ and $\pi_{s}$ is the optimal coupling between
$\bar{\nu}_{s}$ and $\nu$. Since $\nu$ is absolutely continuous w.r.t. the
Lebesgue measure$,$ for each $i=0,$ $s,$ $1,$ there exists a $\nu$-unqiue
optimal mapping $T_{i}$ from $\nu$ to $\bar{\nu}_{i}$, i.e.,%
\[
d\pi_{i}\left(  x,y\right)  =\delta_{T_{i}\left(  y\right)  }\left(  x\right)
d\nu\left(  y\right)  ,\text{ }i=0,\text{ }s,\text{ }1.
\]
In view of $\left(  \ref{thm_uniqueness1}\right)  $,%
\[
T_{s}\left(  y\right)  =\left(  1-s\right)  T_{0}\left(  y\right)
+sT_{1}\left(  y\right)  ,\text{ }\nu\text{-}a.e.y.
\]
We claim that $T_{0}\left(  y\right)  =T_{s}\left(  y\right)  =T_{1}\left(
y\right)  $ for $\nu$-$a.e.$ $y.$ Otherwise we would have, due to the strict
convexity of the cost, that%
\begin{align*}
\mathcal{T}_{c}\left(  \mathrm{P}_{k,\mu\leqslant}^{\mathcal{A}},\nu\right)
&  =\mathcal{T}_{c}\left(  \bar{\nu}_{s},\nu\right)  =\int_{\mathbb{R}^{d}%
}c\left(  T_{s}\left(  y\right)  ,y\right)  d\nu=\int_{\mathbb{R}^{d}}c\left(
\left(  1-s\right)  T_{0}\left(  y\right)  +sT_{1}\left(  y\right)  ,y\right)
d\nu\\
&  <\left(  1-s\right)  \int_{\mathbb{R}^{d}}c\left(  T_{0}\left(  y\right)
,y\right)  d\nu+s\int_{\mathbb{R}^{d}}c\left(  T_{1}\left(  y\right)
,y\right)  d\nu\\
&  =\left(  1-s\right)  \mathcal{T}_{c}\left(  \bar{\nu}_{0},\nu\right)
+s\mathcal{T}_{c}\left(  \bar{\nu}_{1},\nu\right)  =\mathcal{T}_{c}\left(
\mathrm{P}_{k,\mu\leqslant}^{\mathcal{A}},\nu\right)  ,
\end{align*}
which is a contradiction. Therefore%
\[
\delta_{T_{0}\left(  y\right)  }\left(  x\right)  d\nu\left(  y\right)
=\delta_{T_{s}\left(  y\right)  }\left(  x\right)  d\nu\left(  y\right)
=\delta_{T_{1}\left(  y\right)  }\left(  x\right)  d\nu\left(  y\right)  .
\]
It follows $\pi_{0},$ $\pi_{s},$ $\pi_{1}$ are equal, thus have equal first
marginals, i.e. $\bar{\nu}_{0}=\bar{\nu}_{s}=\bar{\nu}_{1}.$
\end{proof}

\section{\label{sec_cx_proj}Convex Order Projections}

This section is devoted to convex order projections and their duality theorems.

\begin{definition}
[\textbf{Convex order}]\label{def:cx_order}Given two probabilities $\mu$,
$\nu\in P_{1}\left(  \mathbb{R}^{d}\right)  ,$ we call $\mu$ smaller than
$\nu$ in convex order, denoted by $\mu\leqslant_{\text{cx}}\nu,$ if the
inequality%
\begin{equation}
\int\varphi d\mu\leqslant\int\varphi d\nu\label{inq_cx_phi}%
\end{equation}
holds for all real-valued convex function $\varphi$ such that both integrals
of $\varphi$ w.r.t. $\mu$, $\nu$ exist in the extended sense.
\end{definition}

The following lemma gives several equivalent definitions of convex order.

\begin{lemma}
\label{lm_cx_bdbelow}Let $\mu$, $\nu\in P_{1}\left(  \mathbb{R}^{d}\right)  $.
Regarding the inequality $\left(  \ref{inq_cx_phi}\right)  $\ the following
statements are equivalent.

(i) $\left(  \ref{inq_cx_phi}\right)  $ holds for any lower semicontinuous
proper convex function $\varphi$\ such that both integrals exist in the
extended sense.

(ii) $\left(  \ref{inq_cx_phi}\right)  $ holds for any lower semicontinuous
proper convex function $\varphi$ which is bounded from below.

(iii) $\left(  \ref{inq_cx_phi}\right)  $ holds for any convex function
$\varphi$\ which is Lipschitz continuous.
\end{lemma}

\begin{proof}
\textbf{1}. First prove (i) and (ii) are equivalent. It suffices to show (ii)
implies (i). Let $\varphi$ be any lower semicontinuous proper convex function
such that both integrals w.r.t. $\mu$, $\nu$ exist in the extended sense.
Then, for $\forall K\leqslant0,$ the function $\max\left\{  \varphi,K\right\}
$ is convex and bounded from below. Hence%
\begin{align*}
\int\varphi^{+}d\mu+\int_{\left\{  \varphi<0\right\}  }\max\left\{
\varphi,K\right\}  d\mu &  =\int\max\left\{  \varphi,K\right\}  d\mu\\
&  \leqslant\int\max\left\{  \varphi,K\right\}  d\nu\\
&  =\int\varphi^{+}d\nu+\int_{\left\{  \varphi<0\right\}  }\max\left\{
\varphi,K\right\}  d\nu.
\end{align*}
Note $\int_{\left\{  \varphi<0\right\}  }\max\left\{  \varphi,K\right\}  d\mu
$, $\int_{\left\{  \varphi<0\right\}  }\max\left\{  \varphi,K\right\}  d\nu$
are finite. If $\int\varphi^{+}d\mu=\infty$, then from the above inequality we
infer that $\int\varphi^{+}d\nu=\infty$. Since both integrals of $\varphi$
w.r.t. $\mu$, $\nu$ exist in the extended sense, it follows that $\int\varphi
d\mu=\int\varphi d\nu=\infty.$ In this case $\left(  \ref{inq_cx_phi}\right)
$ holds trivially with both sides equal to $\infty$. In the other extreme case
where $\int\varphi^{+}d\mu<\infty$ and $\int\varphi^{+}d\nu=\infty,$ $\left(
\ref{inq_cx_phi}\right)  $ holds trivially too. It remains to consider the
case where $\int\varphi^{+}d\mu<\infty$, $\int\varphi^{+}d\nu<\infty$. In this
case, we can let $K\rightarrow-\infty$ and use monotone convergence theorem to
obtain%
\[
\int_{\left\{  \varphi<0\right\}  }\max\left\{  \varphi,K\right\}
d\mu\rightarrow\int_{\left\{  \varphi<0\right\}  }\varphi d\mu,\text{ }%
\int_{\left\{  \varphi<0\right\}  }\max\left\{  \varphi,K\right\}
d\nu\rightarrow\int_{\left\{  \varphi<0\right\}  }\varphi d\nu.
\]
It follows that%
\[
\int\varphi d\mu\leqslant\int\varphi d\nu.
\]
Therefore, (i) holds in all cases.

\textbf{2}. For the equivalence between (ii) and (iii), it suffices to show
(iii) implies (ii). Consider any lower semicontinuous proper convex function
$\varphi$ on $\mathbb{R}^{d}$. Since $\varphi$ is lower semicontinuous, there
is a sequence of $n$-Lipschitz convex functions $\varphi_{n}$\ which increase
to $\varphi$ in a pointwise manner as $n\rightarrow\infty$. Note $\varphi
_{n}\in L^{1}\left(  d\mu\right)  \cap$ $L^{1}\left(  d\nu\right)  ,$ $\forall
n.$ Then using monotone convergence theorem, we obtain that $\varphi$
satisfies $\left(  \ref{inq_cx_phi}\right)  .$ Thus (ii) is proved.
\end{proof}

Define%
\[
\mathcal{A}_{\text{cx}}=\left\{  \varphi:\varphi\text{ proper convex, lower
semiconitnuous}\right\}  .
\]
Note that for any $\mu\in P_{1}\left(  \mathbb{R}^{d}\right)  $ and any proper
convex function $\varphi,\,$the integral $\int\varphi d\mu\ $always exists in
the extended sense. This results from the fact that $\varphi$ is supported by
a linear function, thus $\int\min\left\{  \varphi,0\right\}  d\mu>-\infty$.
This together with Lemma \ref{lm_cx_bdbelow} indicates that the three defining
classes $\mathcal{A}_{\text{cx}},$ $\mathcal{A}_{\text{cx}}\cap\mathrm{S}%
_{b,1}$, $\mathcal{A}_{\text{cx}}\cap\mathrm{S}_{b,2}$\ produce the same
convex order relation. Instead of $\mathcal{A}_{\text{cx}},$ we can also use
the following class to get the same convex order,%
\[
\mathcal{A}_{\text{cx,lb}}=\left\{  \varphi\in\mathcal{A}_{\text{cx}}%
:\varphi\text{ bounded from below}\right\}  .
\]
Note we have
\begin{equation}
\mathcal{A}_{\text{cx}}\cap C_{b,1}=\mathcal{A}_{\text{cx}}\cap\mathrm{S}%
_{b,1}=\left\{  \varphi:\varphi\text{ convex, Lipschitz}\right\}  .
\label{Acx_lip}%
\end{equation}

Let $k\geqslant1.$ The backward and forward convex order cones are denoted by%
\[
\mathrm{P}_{k,\leqslant\nu}^{\text{cx}}=\left\{  \eta\in P_{k}\left(
\mathbb{R}^{d}\right)  :\eta\leqslant_{\text{cx}}\nu\right\}  .
\]
and%
\[
\mathrm{P}_{k,\mu\leqslant}^{\text{cx}}=\left\{  \xi\in P_{k}\left(
\mathbb{R}^{d}\right)  :\mu\leqslant_{\text{cx}}\xi\right\}  .
\]
In view of Lemma \ref{lm_cx_bdbelow}, the defining classes $\mathcal{A}%
_{\text{cx}},$ $\mathcal{A}_{\text{cx}}\cap\mathrm{S}_{b,1},$ $\mathcal{A}%
_{\text{cx,lb}}$ etc. give the same convex order cones.

\subsection{The duality theorems}

Now we state the dualities for backward and forward convex order projections.

\begin{theorem}
\label{thm_dual_cx}Let $\mu\in P_{2}\left(  \mathbb{R}^{d}\right)  ,$ $\nu\in
P_{2}\left(  \mathbb{R}^{d}\right)  .$

(i) For $k=1,$ $2,$ the duality for backward convex order projection holds%
\begin{align*}
\mathcal{T}_{2}\left(  \mu,\mathrm{P}_{1,\leqslant\nu}^{\text{cx}}\right)   &
=\mathcal{T}_{2}\left(  \mu,\mathrm{P}_{2,\leqslant\nu}^{\text{cx}}\right)
=\sup_{\varphi\in\mathcal{A}_{\text{cx}}\mathcal{\cap}C_{b,k}}\left\{
\int_{\mathbb{R}^{d}}Q_{2}\left(  \varphi\right)  d\mu-\int_{\mathbb{R}^{d}%
}\varphi d\nu\right\} \\
&  =\sup_{\varphi\in\mathcal{A}_{\text{cx}}\cap L^{1}\left(  d\nu\right)
}\left\{  \int_{\mathbb{R}^{d}}Q_{2}\left(  \varphi\right)  d\mu
-\int_{\mathbb{R}^{d}}\varphi d\nu\right\}  ,
\end{align*}
where%
\begin{equation}
Q_{2}\left(  \varphi\right)  \left(  x\right)  =\inf_{y\in\mathbb{R}^{d}%
}\left\{  \varphi\left(  y\right)  +\left\vert x-y\right\vert ^{2}\right\}
\label{eq_Qc2}%
\end{equation}

(ii) For $k=1,$ $2,$ the duality for forward convex order projection holds%
\[
\mathcal{T}_{2}\left(  \mathrm{P}_{1,\mu\leqslant}^{\text{cx}},\nu\right)
=\mathcal{T}_{2}\left(  \mathrm{P}_{2,\mu\leqslant}^{\text{cx}},\nu\right)
=\sup_{\varphi\in\mathcal{A}_{\text{cx}}\mathcal{\cap}C_{b,k}}\left\{
\int_{\mathbb{R}^{d}}\varphi d\mu-\int_{\mathbb{R}^{d}}Q_{\bar{2}}\left(
\varphi\right)  d\nu\right\}
\]
where%
\begin{equation}
Q_{\bar{2}}\left(  \varphi\right)  \left(  y\right)  =\sup_{x\in\mathbb{R}%
^{d}}\left\{  \varphi\left(  x\right)  -\left\vert x-y\right\vert
^{2}\right\}  . \label{eq_Qcbar2}%
\end{equation}

\noindent Note in both (i) and (ii), the optimal values are equal regardless
of $k$, moreover $\varphi\in\mathcal{A}_{\text{cx}}\mathcal{\cap}C_{b,k}$ can
be relaxed to $\varphi\in\mathcal{A\cap}\mathrm{S}_{b,k}$.
\end{theorem}

\begin{proof}
The dualities for individual $k$ are obtained by invoking Theorem
\ref{thm_dual_bwg} and Theorem \ref{thm_dual_fwg}. To see that the optimal
values are equal regardless of $k=1,2,$ we proceed as below.

\textbf{1}. We prove that%
\[
\mathrm{P}_{1,\leqslant\nu}^{\text{cx}}=\mathrm{P}_{2,\leqslant\nu}%
^{\text{cx}}\text{,}%
\]
whence%
\[
\mathcal{T}_{2}\left(  \mu,\mathrm{P}_{1,\leqslant\nu}^{\text{cx}}\right)
=\mathcal{T}_{2}\left(  \mu,\mathrm{P}_{2,\leqslant\nu}^{\text{cx}}\right)  .
\]
Since $\mathrm{P}_{2,\leqslant\nu}^{\text{cx}}\subset\mathrm{P}_{1,\leqslant
\nu}^{\text{cx}},$ so it suffices to prove the opposite includsion. For any
$\eta\in P_{1}\left(  \mathbb{R}^{d}\right)  $ and $\eta\leqslant_{\text{cx}%
}\nu$, we have%
\[
\int_{\mathbb{R}^{d}}\varphi d\eta\leqslant\int_{\mathbb{R}^{d}}\varphi
d\nu\text{ for all }\varphi\in\mathcal{A}_{\text{cx}}\cap\mathrm{S}_{b,2}.
\]
Since $\nu\in P_{2}\left(  \mathbb{R}^{d}\right)  ,$ we can take as $\varphi$
the convex function $\left\vert y\right\vert ^{2}$ to obtain%
\[
\int_{\mathbb{R}^{d}}\left\vert y\right\vert ^{2}d\eta\leqslant\int%
_{\mathbb{R}^{d}}\left\vert y\right\vert ^{2}d\nu<\infty,
\]
which shows that $\eta\in P_{2}\left(  \mathbb{R}^{d}\right)  .$ Therefore
$\mathrm{P}_{1,\leqslant\nu}^{\text{cx}}\subset\mathrm{P}_{2,\leqslant\nu
}^{\text{cx}}.$

\textbf{2}. That $\varphi\in\mathcal{A}_{\text{cx}}\mathcal{\cap}C_{b,k}$ in
the backward duality can be relaxed to $\varphi\in\mathcal{A}_{\text{cx}}\cap
L^{1}\left(  d\nu\right)  $ follows from Remark \ref{rmk_relax_phi} and the
fact that real-valued convex functions are supported by linear functions. This
completes the proof of item (i).

\textbf{3}. We prove that%
\begin{equation}
\mathcal{T}_{2}\left(  \mathrm{P}_{1,\mu\leqslant}^{\text{cx}},\nu\right)
=\mathcal{T}_{2}\left(  \mathrm{P}_{2,\mu\leqslant}^{\text{cx}},\nu\right)  .
\label{thm_dual_cx0}%
\end{equation}
The idea of step 1 does not apply here. In general we only have $\mathrm{P}%
_{2,\mu\leqslant}^{\text{cx}}\subset\mathrm{P}_{1,\mu\leqslant}^{\text{cx}},$
whence%
\[
\mathcal{T}_{2}\left(  \mathrm{P}_{1,\mu\leqslant}^{\text{cx}},\nu\right)
\leqslant\mathcal{T}_{2}\left(  \mathrm{P}_{2,\mu\leqslant}^{\text{cx}}%
,\nu\right)  .
\]
To show the reverse inequality, we prove that there exists $\bar{\nu}%
\in\mathrm{P}_{2,\mu\leqslant}^{\text{cx}}$ such that%
\begin{equation}
\mathcal{T}_{2}\left(  \bar{\nu},\nu\right)  =\mathcal{T}_{2}\left(
\mathrm{P}_{1,\mu\leqslant}^{\text{cx}},\nu\right)  , \label{thm_dual_cx1}%
\end{equation}
which indicates%
\[
\mathcal{T}_{2}\left(  \mathrm{P}_{2,\mu\leqslant}^{\text{cx}},\nu\right)
\leqslant\mathcal{T}_{2}\left(  \bar{\nu},\nu\right)  =\mathcal{T}_{2}\left(
\mathrm{P}_{1,\mu\leqslant}^{\text{cx}},\nu\right)  .
\]
Therefore, once $\left(  \ref{thm_dual_cx1}\right)  $ is established, the
proof of $\left(  \ref{thm_dual_cx0}\right)  $\ would be completed. Note this
actually yields a direct proof of the existence for$\mathcal{\ T}_{2}\left(
\mathrm{P}_{1,\mu\leqslant}^{\text{cx}},\nu\right)  $, $\mathcal{T}_{2}\left(
\mathrm{P}_{2,\mu\leqslant}^{\text{cx}},\nu\right)  .$ To prove $\left(
\ref{thm_dual_cx1}\right)  ,$ first note, in view of Remark
\ref{rmk_finiteness}, $\mathcal{T}_{2}\left(  \mathrm{P}_{1,\mu\leqslant
}^{\text{cx}},\nu\right)  $ is finite. Let $\bar{\nu}_{n}\in\mathrm{P}%
_{1,\mu\leqslant}^{\text{cx}}$ be a minimizing sequence for $\mathcal{T}%
_{2}\left(  \mathrm{P}_{1,\mu\leqslant}^{\text{cx}},\nu\right)  $. Then the
sequence $\mathcal{T}_{2}\left(  \bar{\nu}_{n},\nu\right)  $ is bounded by
some constant $M>0.$ By $\mathcal{W}_{2}$-triangle inequality,
\begin{align*}
\int_{\mathbb{R}^{d}}\left\vert x\right\vert ^{2}d\bar{\nu}_{n}  &
=\mathcal{T}_{2}\left(  \bar{\nu}_{n},\delta_{0}\right)  \leqslant\left(
\mathcal{W}_{2}\left(  \bar{\nu}_{n},\nu\right)  +\mathcal{W}_{2}\left(
\nu,\delta_{0}\right)  \right)  ^{2}\\
&  \leqslant2\left(  \mathcal{T}_{2}\left(  \bar{\nu}_{n},\nu\right)
+\mathcal{T}_{2}\left(  \nu,\delta_{0}\right)  \right)  \leqslant2\left(
M+\int_{\mathbb{R}^{d}}\left\vert x\right\vert ^{2}d\nu\right)  <\infty,\text{
}\forall n.
\end{align*}
Here $\delta_{0}$ is the Dirac measure concentrated at the origin. By virtue
of Markov inequality the sequence $\bar{\nu}_{n}$ is tight (see e.g.
\cite{billingsley2013convergence}). Therefore there exists a subsequence
$\bar{\nu}_{n_{i}}$ converging weakly to some probability $\bar{\nu}.$
Meanwhile%
\[
\sup_{i}\int_{\left\vert x\right\vert \geqslant R}\left\vert x\right\vert
d\bar{\nu}_{n_{i}}\leqslant\sup_{i}\left(  \frac{1}{R}\int_{\left\vert
x\right\vert \geqslant R}\left\vert x\right\vert ^{2}d\bar{\nu}_{n_{i}%
}\right)  \leqslant\frac{1}{R}\sup_{n}\int_{\mathbb{R}^{d}}\left\vert
x\right\vert ^{2}d\bar{\nu}_{n}\rightarrow0\text{ as }R\rightarrow\infty.
\]
By \cite[Theorem 7.12 (ii)]{villani2003topics},%
\[
\mathcal{W}_{1}\left(  \bar{\nu}_{n_{i}},\bar{\nu}\right)  \rightarrow0\text{
as }i\rightarrow\infty\text{.}%
\]
Hence for any $g\in\mathrm{S}_{b,1}$,%
\[
\int_{\mathbb{R}^{d}}gd\bar{\nu}_{n}\rightarrow\int_{\mathbb{R}^{d}}gd\bar
{\nu}\text{ as }n\rightarrow\infty\text{.}%
\]
Using Lemma \ref{lm_cx_bdbelow} and $\mu\leqslant\bar{\nu}_{n}$, we infer that
$\mu\leqslant_{\text{cx}}\bar{\nu}.$ Moreover $\bar{\nu}$ has second-order
moment. Indeed, by virtue of Layer cake theorem \cite[p26, 1.13]%
{lieb2001analysis},%
\[
\int_{\mathbb{R}^{d}}\left\vert x\right\vert ^{2}d\bar{\nu}_{n_{i}}=2\int%
_{0}^{\infty}t\bar{\nu}_{n_{i}}\left(  \left\vert x\right\vert >t\right)  dt.
\]
Since $\bar{\nu}_{n_{i}}$ converges weakly to $\bar{\nu},$
\[
\lim_{i\rightarrow\infty}\bar{\nu}_{n_{i}}\left(  \left\vert x\right\vert
>t\right)  =\bar{\nu}\left(  \left\vert x\right\vert >t\right)  ,\text{ for
}a.e.\text{-}t.
\]
Now using Fatou lemma%
\begin{align*}
\int_{\mathbb{R}^{d}}\left\vert x\right\vert ^{2}d\bar{\nu}  &  =2\int%
_{0}^{\infty}t\bar{\nu}\left(  \left\vert x\right\vert >t\right)
dt\leqslant\liminf_{i\rightarrow\infty}2\int_{0}^{\infty}t\bar{\nu}_{n_{i}%
}\left(  \left\vert x\right\vert >t\right)  dt\\
&  =\liminf_{i\rightarrow\infty}\int_{\mathbb{R}^{d}}\left\vert x\right\vert
^{2}d\bar{\nu}_{n_{i}}\leqslant\sup_{n}\int_{\mathbb{R}^{d}}\left\vert
x\right\vert ^{2}d\bar{\nu}_{n}<\infty.
\end{align*}
Hence $\bar{\nu}\in P_{2}\left(  \mathbb{R}^{d}\right)  .$ Therefore $\bar
{\nu}\in\mathrm{P}_{2,\mu\leqslant}^{\text{cx}}\subset\mathrm{P}%
_{1,\mu\leqslant}^{\text{cx}}.$ Finally, we show that $\bar{\nu}$ is optimal
for $\mathcal{T}_{2}\left(  \mathrm{P}_{1,\mu\leqslant}^{\text{cx}}%
,\nu\right)  .$ Let $\bar{\pi}_{n}\in\Pi\left(  \bar{\nu}_{n},\nu\right)  $ be
the optimal couplings for $\mathcal{T}_{2}\left(  \bar{\nu}_{n},\nu\right)  $.
Since $\bar{\nu}_{n_{i}}\in\mathrm{P}_{1,\mu\leqslant}^{\text{cx}}$ converges
weakly to $\bar{\nu}$, we may assume that $\bar{\pi}_{n_{i}}$ converges weakly
to some $\bar{\pi}\in\Pi\left(  \bar{\nu},\nu\right)  .$\ For any $R>0,$%
\begin{align*}
\int_{\mathbb{R}^{d}\times\mathbb{R}^{d}}\left(  \left\vert x-y\right\vert
^{2}\wedge R\right)  d\bar{\pi}\left(  x,y\right)   &  =\lim_{i\rightarrow
\infty}\int_{\mathbb{R}^{d}\times\mathbb{R}^{d}}\left(  \left\vert
x-y\right\vert ^{2}\wedge R\right)  d\bar{\pi}_{n_{i}}\left(  x,y\right) \\
&  \leqslant\lim_{i\rightarrow\infty}\int_{\mathbb{R}^{d}\times\mathbb{R}^{d}%
}\left\vert x-y\right\vert ^{2}d\bar{\pi}_{n_{i}}\left(  x,y\right)
=\mathcal{T}_{2}\left(  \mathrm{P}_{1,\mu\leqslant}^{\text{cx}},\nu\right)  .
\end{align*}
Sending $R\rightarrow\infty$ and using monotone convergence theorem,%
\[
\int_{\mathbb{R}^{d}\times\mathbb{R}^{d}}\left\vert x-y\right\vert ^{2}%
d\bar{\pi}\left(  x,y\right)  \leqslant\mathcal{T}_{2}\left(  \mathrm{P}%
_{1,\mu\leqslant}^{\text{cx}},\nu\right)  .
\]
Since $\bar{\nu}\in\mathrm{P}_{1,\mu\leqslant}^{\text{cx}},$ this shows
$\bar{\pi}$ is an optimal coupling which achieves $\mathcal{T}_{2}\left(
\mathrm{P}_{1,\mu\leqslant}^{\text{cx}},\nu\right)  $, thus $\left(
\ref{thm_dual_cx1}\right)  $ is satisfied.
\end{proof}

\begin{remark}
From the proof of Theorem \ref{thm_dual_cx}, any projection onto
$\mathrm{P}_{1,\leqslant\nu}^{\text{cx}}$ is also a projection onto
$\mathrm{P}_{2,\leqslant\nu}^{\text{cx}},$ and vice versa.\ The same is true
for $\mathrm{P}_{1,\mu\leqslant}^{\text{cx}}$ and $\mathrm{P}_{2,\mu\leqslant
}^{\text{cx}}.$
\end{remark}

\section{\label{sec_att_cx}Dual attainment for convex order projections}

In this section we will prove the attainment of the optimal dual value
associated with backward convex order projection and forward convex order
projection. Characterization of these optimal dual potentials will be given in
the following section. These are tackled by combining the treatments in the
classical optimal transportation with duality formulas proved in the earlier sections.

\subsection{Backward projection}

\begin{theorem}
\label{thm_dual_att_cx_bw}Let $\mu\in P_{2}\left(  \mathbb{R}^{d}\right)  ,$
$\nu\in P_{2}\left(  \mathbb{R}^{d}\right)  $. Then, for each $k=1$ or $2,$
there is a convex function $\varphi_{0}\in\mathcal{A}_{\text{cx}}\cap
L^{1}\left(  d\nu\right)  $ with values in $\mathbb{R}\cup\left\{
\infty\right\}  $ which achieves the optimal dual value
\[
\mathcal{D}_{2}\left(  \mu,\mathrm{P}_{k,\leqslant\nu}^{\text{cx}}\right)
\triangleq\sup_{\varphi\in\mathcal{A}_{\text{cx}}\cap\mathrm{S}_{b,k}}\left\{
\int Q_{2}\left(  \varphi\right)  d\mu-\int\varphi d\nu\right\}  ,
\]
where $Q_{2}\left(  \cdot\right)  $ is defined in $\left(  \ref{eq_Qc2}%
\right)  .$ If both $\mu$ and $\nu$ have compact supports, then there exists a
Lipschitz convex optimal dual solution. Moreover%
\[
\mathcal{D}_{2}\left(  \mu,\mathrm{P}_{1,\leqslant\nu}^{\text{cx}}\right)
=\mathcal{D}_{2}\left(  \mu,\mathrm{P}_{2,\leqslant\nu}^{\text{cx}}\right)  .
\]

\end{theorem}

\begin{proof}
Using the equivalence between backward convex order projection and weak
optimal transport, it is proved in \cite[Proposition 1.1]{gozlan2020mixture}
that%
\[
\mathcal{T}_{2}\left(  \mu,\mathrm{P}_{1,\leqslant\nu}^{\text{cx}}\right)
=\mathcal{D}_{2,\text{lb}}\left(  \mu,\mathrm{P}_{1,\leqslant\nu}^{\text{cx}%
}\right)  \triangleq\sup_{\varphi\in\mathcal{A}_{\text{cx,lb}}\cap
\mathrm{S}_{b,1}}\left\{  \int_{\mathbb{R}^{d}}Q_{2}\left(  \varphi\right)
d\mu-\int_{\mathbb{R}^{d}}\varphi d\nu\right\}  .
\]
The optimal dual value $\mathcal{D}_{2,\text{lb}}\left(  \mu,\mathrm{P}%
_{1,\leqslant\nu}^{\text{cx}}\right)  $ is attained at some $\varphi_{0}%
\in\mathcal{A}_{\text{cx}}\cap L^{1}\left(  d\nu\right)  $ with values in
$\mathbb{R}\cup\left\{  \infty\right\}  $. Note in general, $\varphi_{0}$
might not be a member of $\mathcal{A}_{\text{cx}}\cap\mathrm{S}_{b,1}$. It is
also proved that, if both $\mu$ and $\nu$ have compact supports, then
$\varphi_{0}$ can be chosen to be Lipschitz convex. In view of Theorem
\ref{thm_dual_cx},%
\[
\mathcal{T}_{2}\left(  \mu,\mathrm{P}_{1,\leqslant\nu}^{\text{cx}}\right)
=\mathcal{D}_{2}\left(  \mu,\mathrm{P}_{1,\leqslant\nu}^{\text{cx}}\right)
=\mathcal{D}_{2}\left(  \mu,\mathrm{P}_{2,\leqslant\nu}^{\text{cx}}\right)  ,
\]
hence the same function $\varphi_{0}$ which attains $\mathcal{D}_{2,\text{lb}%
}\left(  \mu,\mathrm{P}_{1,\leqslant\nu}^{\text{cx}}\right)  $ also achieves
$\mathcal{D}_{2}\left(  \mu,\mathrm{P}_{1,\leqslant\nu}^{\text{cx}}\right)  $
and $\mathcal{D}_{2}\left(  \mu,\mathrm{P}_{2,\leqslant\nu}^{\text{cx}%
}\right)  .$ Therefore $\varphi_{0}$ is the desired function we look for.
\end{proof}

\subsection{Forward projection}

The next result is about the attainment of forward convex order projection.
The proof combines a preliminary trick with the strategy used in the proof of
\cite[Theorem 1.2]{gozlan2020mixture}.

\begin{theorem}
\label{thm_dual_att_cx}Let $\mu\in P_{2}\left(  \mathbb{R}^{d}\right)  ,$
$\nu\in P_{2}\left(  \mathbb{R}^{d}\right)  $. Then, for each $k=1$ or $2,$
there exists a convex function $\bar{\varphi}_{0}\in\mathcal{A}_{\text{cx}%
}\cap C_{b,2}$ which satisfies $\bar{\varphi}_{0}\leqslant\left\vert
x\right\vert ^{2}$ and achieves the optimal dual value
\[
\mathcal{D}_{2}\left(  \mathrm{P}_{k,\mu\leqslant}^{\text{cx}},\nu\right)
\triangleq\sup_{\varphi\in\mathcal{A}_{\text{cx}}\cap\mathrm{S}_{b,k}}\left\{
\int\varphi d\mu-\int Q_{\bar{2}}\left(  \varphi\right)  d\nu\right\}  .
\]
where $Q_{\bar{2}}\left(  \cdot\right)  $ is defined in $\left(
\ref{eq_Qcbar}\right)  .$ Moreover%
\[
\mathcal{D}_{2}\left(  \mathrm{P}_{1,\mu\leqslant}^{\text{cx}},\nu\right)
=\mathcal{D}_{2}\left(  \mathrm{P}_{2,\mu\leqslant}^{\text{cx}},\nu\right)  .
\]

\end{theorem}

\begin{proof}
That the two optimal dual values are equal is a result of Theorem
\ref{thm_dual_cx}. The proofs of attainment for $k=1$ and $2$ are similar, but
note that the function which achieves $\mathcal{D}_{2}\left(  \mathrm{P}%
_{1,\mu\leqslant}^{\text{cx}},\nu\right)  $ might not be in $\mathrm{S}_{b,1}%
$. Here we only give proof for $k=2.$

\textbf{1}. We first show that the optimization can be restricted to
$\mathcal{A}_{\text{cx}}^{0}\cap\mathrm{S}_{b,2},$ where%
\[
\mathcal{A}_{\text{cx}}^{0}=\left\{  \bar{\varphi}\in\mathcal{A}_{\text{cx}%
}:\bar{\varphi}=Q_{2}\left(  Q_{\bar{2}}\left(  \bar{\varphi}\right)  \right)
,\text{ }Q_{\bar{2}}\left(  \bar{\varphi}\right)  \left(  0\right)
=0\right\}  .
\]
Indeed, first recall Remark \ref{rmk_finiteness}, the optimal dual value is
finite under the conditions of the theorem.\ If $\varphi\in\mathcal{A}%
_{\text{cx}}\cap\mathrm{S}_{b,2}$ is such that $Q_{\bar{2}}\left(
\varphi\right)  $ is identically $\infty$, then the dual value is $-\infty,$
since the integral $\int\varphi d\mu$ is finite. So we can avoid these
functions in the supremum in $\mathcal{D}_{2}\left(  \mathrm{P}_{2,\mu
\leqslant}^{\text{cx}},\nu\right)  $. For any $\varphi\in\mathcal{A}%
_{\text{cx}}\cap\mathrm{S}_{b,2}$ such that $Q_{\bar{2}}\left(  \varphi
\right)  $ is not identically $\infty,$ there is $y_{0}\in\mathbb{R}^{d}$
(depending on $\varphi$) for which%
\[
Q_{\bar{2}}\left(  \varphi\right)  \left(  y_{0}\right)  <\infty.
\]
Let%
\[
\bar{\varphi}\left(  x\right)  \triangleq\inf_{y\in\mathbb{R}^{d}}\left\{
Q_{\bar{2}}\left(  \varphi\right)  \left(  y\right)  +\left\vert
x-y\right\vert ^{2}\right\}  =Q_{2}\left(  Q_{\bar{2}}\left(  \varphi\right)
\right)  \left(  x\right)  .
\]
Then%
\[
\bar{\varphi}\left(  x\right)  \leqslant Q_{\bar{2}}\left(  \varphi\right)
\left(  y_{0}\right)  +\left\vert x-y_{0}\right\vert ^{2}<\infty,\text{
}\forall x.
\]
Since $\bar{\varphi}$ is convex and $\varphi\leqslant\bar{\varphi}$, we see
that $\bar{\varphi}\in\mathcal{A}_{\text{cx}}\cap C_{b,2}.$ By $\left(
\ref{eq_Qccbar12}\right)  ,$ we have%
\[
Q_{\bar{2}}\left(  \varphi\right)  \left(  y\right)  =Q_{\bar{2}}\left(
Q_{2}\left(  Q_{\bar{2}}\left(  \varphi\right)  \right)  \right)  \left(
y\right)  =Q_{\bar{2}}\left(  \bar{\varphi}\right)  \left(  y\right)  ,
\]
whence%
\[
\bar{\varphi}\left(  x\right)  =Q_{2}\left(  Q_{\bar{2}}\left(  \varphi
\right)  \right)  \left(  x\right)  =Q_{2}\left(  Q_{\bar{2}}\left(
\bar{\varphi}\right)  \right)  \left(  x\right)  .
\]
So%
\[
\int\varphi d\mu-\int Q_{\bar{2}}\left(  \varphi\right)  d\nu\leqslant\int%
\bar{\varphi}d\mu-\int Q_{\bar{2}}\left(  \bar{\varphi}\right)  d\nu
\leqslant\mathcal{D}_{2}\left(  \mathrm{P}_{2,\mu\leqslant}^{\text{cx}}%
,\nu\right)  .
\]
The last inequality is due to $\bar{\varphi}\in\mathcal{A}_{\text{cx}}%
\cap\mathrm{S}_{b,2}$, so it is admissible to the supremum in $\mathcal{D}%
_{2}\left(  \mathrm{P}_{2,\mu\leqslant}^{\text{cx}},\nu\right)  $. Moreover
adding a constant to $\bar{\varphi}$ will not change the difference of the
integrals, we may assume that $Q_{\bar{2}}\left(  \bar{\varphi}\right)
\left(  0\right)  =0.$ Therefore we have proved%
\[
\mathcal{D}_{2}\left(  \mathrm{P}_{2,\mu\leqslant}^{\text{cx}},\nu\right)
=\sup_{\bar{\varphi}\in\mathcal{A}_{\text{cx}}^{0}\cap\mathrm{S}_{b,2}%
}\left\{  \int\bar{\varphi}d\mu-\int Q_{\bar{2}}\left(  \bar{\varphi}\right)
d\nu\right\}  .
\]

\textbf{2}. Let $\bar{\varphi}_{n}\in\mathcal{A}_{\text{cx}}^{0}\cap
\mathrm{S}_{b,2}$\ be a maximization sequence for $\mathcal{D}_{2}\left(
\mathrm{P}_{2,\mu\leqslant}^{\text{cx}},\nu\right)  $. Write%
\[
\varrho_{n}=Q_{\bar{2}}\left(  \bar{\varphi}_{n}\right)  .
\]
Then $\bar{\varphi}_{n}=Q_{2}\left(  \varrho_{n}\right)  $ with $\varrho
_{n}\left(  0\right)  =0.$ We assume, without loss of generality, that $\int
yd\nu=0.$ First we have an upper bound (uniform in $n$) on $\bar{\varphi}%
_{n},$%
\begin{equation}
\bar{\varphi}_{n}\left(  x\right)  =Q_{2}\left(  \varrho_{n}\right)  \left(
x\right)  =\inf_{y}\left\{  \varrho_{n}\left(  y\right)  +\left\vert
x-y\right\vert ^{2}\right\}  \leqslant\left\vert x\right\vert ^{2}.
\label{thm_dual_att_cx1}%
\end{equation}
Next we will obtain a pointwise lower bound (uniform in $n$) on $\bar{\varphi
}_{n}$. Since%
\[
\left\vert x\right\vert ^{2}-\bar{\varphi}_{n}\left(  x\right)  =2\sup
_{y}\left\{  x\cdot y-\frac{1}{2}\left(  \left\vert y\right\vert ^{2}%
-\varrho_{n}\left(  y\right)  \right)  \right\}
\]
is convex, we get by Jensen inequality%
\[
\left\vert \bar{x}\right\vert ^{2}-\bar{\varphi}_{n}\left(  \bar{x}\right)
\leqslant\int\left(  \left\vert x\right\vert ^{2}-\bar{\varphi}_{n}\left(
x\right)  \right)  d\mu,
\]
where%
\[
\bar{x}=\int xd\mu.
\]
It follows that%
\begin{align*}
\bar{\varphi}_{n}\left(  \bar{x}\right)   &  \geqslant\left\vert \bar
{x}\right\vert ^{2}-\int\left(  \left\vert x\right\vert ^{2}-\bar{\varphi}%
_{n}\left(  x\right)  \right)  d\mu\\
&  =\left\vert \bar{x}\right\vert ^{2}+\int\bar{\varphi}_{n}\left(  x\right)
d\mu-\int\varrho_{n}\left(  y\right)  d\nu-\int\left\vert x\right\vert
^{2}d\mu+\int\varrho_{n}\left(  y\right)  d\nu.
\end{align*}
Note%
\[
\int\bar{\varphi}_{n}\left(  x\right)  d\mu-\int\varrho_{n}\left(  y\right)
d\nu\text{ is bounded.}%
\]
Since $\varrho_{n}$ is convex,%
\[
\int\varrho_{n}\left(  y\right)  d\nu\geqslant\varrho_{n}\left(  \int
yd\nu\right)  =\varrho_{n}\left(  0\right)  =0.
\]
Therefore $\bar{\varphi}_{n}\left(  \bar{x}\right)  $ is bounded from below.
Using the convexity of $\bar{\varphi}_{n}$ and $\left(  \ref{thm_dual_att_cx1}%
\right)  ,$ we get%
\[
\bar{\varphi}_{n}\left(  \bar{x}\right)  \leqslant\frac{1}{2}\bar{\varphi}%
_{n}\left(  2\bar{x}-x\right)  +\frac{1}{2}\bar{\varphi}_{n}\left(  x\right)
\leqslant\frac{1}{2}\left\vert 2\bar{x}-x\right\vert ^{2}+\frac{1}{2}%
\bar{\varphi}_{n}\left(  x\right)  .
\]
It follows that%
\begin{equation}
\bar{\varphi}_{n}\left(  x\right)  \geqslant\frac{1}{2}\bar{\varphi}%
_{n}\left(  \bar{x}\right)  -\frac{1}{2}\left\vert 2\bar{x}-x\right\vert ^{2}.
\label{thm_dual_att_cx2}%
\end{equation}
In view of $\left(  \ref{thm_dual_att_cx1}\right)  $ and the fact that
$\bar{\varphi}_{n}\left(  \bar{x}\right)  $ is bounded from below, we have
$\bar{\varphi}_{n}\left(  x\right)  $ is bounded for each $x$. Therefore we
may assume, up to a subsequence, that%
\[
\bar{\varphi}_{n}\left(  x\right)  \rightarrow\bar{\varphi}_{0}\left(
x\right)  ,\text{ }\forall x,
\]
where $\bar{\varphi}_{0}\left(  x\right)  $ is a convex function. Now we show
$\bar{\varphi}_{0}$ is the optimal dual solution we look for. Clearly
$\bar{\varphi}_{0}\in\mathcal{A}_{\text{cx}}\cap C_{b,2}.$ Note%
\begin{equation}
\varrho_{n}\left(  y\right)  \geqslant\bar{\varphi}_{n}\left(  x\right)
-\left\vert x-y\right\vert ^{2},\text{ }\forall x,y. \label{thm_dual_att_cx3}%
\end{equation}
Hence%
\[
\liminf_{n\rightarrow\infty}\varrho_{n}\left(  y\right)  \geqslant\bar
{\varphi}_{0}\left(  x\right)  -\left\vert x-y\right\vert ^{2},\text{ }\forall
x,y.
\]
So%
\[
\liminf_{n\rightarrow\infty}\varrho_{n}\left(  y\right)  \geqslant\sup
_{x}\left\{  \bar{\varphi}_{0}\left(  x\right)  -\left\vert x-y\right\vert
^{2}\right\}  =Q_{\bar{2}}\left(  \bar{\varphi}_{0}\right)  \left(  y\right)
.
\]
Note $\bar{\varphi}_{n}$ is bounded from above by a $\mu$-integrable function
(ref. $\left(  \ref{thm_dual_att_cx1}\right)  $) and $\varrho_{n}$ is bounded
from below by a common $\nu$-integrable quadratic function (ref. $\left(
\ref{thm_dual_att_cx2}\right)  $ and $\left(  \ref{thm_dual_att_cx3}\right)
$)$.$ Therefore, using Fatou lemma,%
\begin{align*}
\mathcal{D}_{2}\left(  \mathrm{P}_{2,\mu\leqslant}^{\text{cx}},\nu\right)   &
=\lim_{n\rightarrow\infty}\left\{  \int\bar{\varphi}_{n}\left(  x\right)
d\mu-\int\varrho_{n}\left(  y\right)  d\nu\right\} \\
&  \leqslant\limsup_{n\rightarrow\infty}\int\bar{\varphi}_{n}\left(  x\right)
d\mu-\liminf_{n\rightarrow\infty}\int\varrho_{n}\left(  y\right)  d\nu\\
&  \leqslant\int\bar{\varphi}_{0}\left(  x\right)  d\mu-\int Q_{\bar{2}%
}\left(  \bar{\varphi}_{0}\right)  \left(  y\right)  d\nu\leqslant
\mathcal{D}_{2}\left(  \mathrm{P}_{2,\mu\leqslant}^{\text{cx}},\nu\right)  .
\end{align*}
So $\bar{\varphi}_{0}$ is the optimal solution we look for.
\end{proof}

Unlike the backward case in Theorem \ref{thm_dual_att_cx_bw}, requiring $\mu$
and $\nu$ to have compact supports does not ease the way to find a more
regular optimal dual solution. This results from the major difference between
backward and forward projection, which we will discuss in Section \ref{sec_bf}.

\section{\label{sec_char_cx}Characterization of convex order projections}

In this section, we present results on the characterization of optimal
mappings for backward and forward convex order projections. The optimal
mappings possess special properties which we call convex contraction and
convex expansion. These are consequences of the duality and attainment
obtained in previous sections. Recall that the Legendre transform of a
function $\phi$ is denoted by $\phi^{\ast}.$

\begin{definition}
[\textbf{Convex contraction}]\label{def_cx_contr}Let $\varphi$ be a proper
lower semicontinuous convex function. We call $\varphi$ a convex contraction
if $\varphi=\phi^{\ast}$ for some proper function $\phi$ such that $D^{2}%
\phi\geqslant Id$ in the sense of distribution.
\end{definition}

\begin{definition}
[\textbf{Convex expansion}]\label{def_cx_exp}Let $\varphi$ be a proper lower
semicontinuous convex function. We call $\varphi$ a convex expansion if
$\varphi=\phi^{\ast}$ for some proper function $\phi$ such that $D^{2}%
\phi\leqslant Id$ in the sense of distribution.
\end{definition}

\begin{remark}
\label{rmk_ctrexp}That $\varphi$ is a convex contraction is equivalent to
$D^{2}\varphi\leqslant Id$. Indeed, in view of Lemma \ref{lm_HJ_cx_duality},
if $\varphi$ is a convex contraction then $D^{2}\varphi\leqslant Id.$
Conversely, if $\varphi$ is a proper lower semicontinuous convex function such
that $D^{2}\varphi\leqslant Id,$ then using again Lemma \ref{lm_HJ_cx_duality}%
, $D^{2}\varphi^{\ast}\geqslant Id.$ Now $\varphi=\left(  \varphi^{\ast
}\right)  ^{\ast},$ so $\varphi$ is a convex contraction. Similarly, that
$\varphi$ is a convex expansion is equivalent to $D^{2}\varphi\geqslant Id.$
The definition of convex contraction and expansion through Legendre transforms
is natural and consistent with the subharmonic order setting. See Definition
\ref{def_Lap_contraction}, Definition \ref{def_Lap_expansion} and Remark
\ref{rmk_Lap}.
\end{remark}

\subsection{Backward projection}

The following characterization of optimal mapping for backward convex order
projection is proved in\ \cite[Theorem 1.2, Theorem 2.1]{gozlan2020mixture},
we include them and restated in our terms for readers' convenience.

\begin{theorem}
\label{thm_cx_bw_prop}Suppose that $\mu\in P_{2}\left(  \mathbb{R}^{d}\right)
,$ $\nu\in P_{2}\left(  \mathbb{R}^{d}\right)  $ and $\varphi\in
\mathcal{A}_{\text{cx}}\cap L^{1}\left(  d\nu\right)  $ is the optimizer of
the dual $\mathcal{D}_{2}\left(  \mu,\mathrm{P}_{1,\leqslant\nu}^{\text{cx}%
}\right)  $ obtained in Theorem \ref{thm_dual_att_cx_bw}$.$ Let%
\[
\varphi_{0}\left(  y\right)  \triangleq\frac{1}{2}\left\vert y\right\vert
^{2}+\varphi\left(  y\right)  .
\]
Then

(i) $\varphi_{0}^{\ast}\in C^{1}\left(  \mathbb{R}^{d}\right)  $ is a convex contraction.

(ii) $\left(  \nabla\varphi_{0}^{\ast}\right)  _{\#}\mu$ is the unique
projection of $\mu$ onto $\mathrm{P}_{1,\leqslant\nu}^{\text{cx}}.$
\end{theorem}

\begin{theorem}
\label{thm_cx_bw_char}Let $\mu\in P_{2}\left(  \mathbb{R}^{d}\right)  ,$
$\nu\in P_{2}\left(  \mathbb{R}^{d}\right)  .$ Then the following are equivalent.

(i) The projection of $\mu\,$onto$\ \mathrm{P}_{1,\leqslant\nu}^{\text{cx}}$
coincides with $\nu$.

(ii) There exists a convex contraction $\phi$ such that $\left(  \nabla
\phi\right)  _{\#}\mu=\nu.$
\end{theorem}

Note in view of the step 1 in the proof of Theorem \ref{thm_dual_cx},
$\mathrm{P}_{1,\leqslant\nu}^{\text{cx}}=\mathrm{P}_{2,\leqslant\nu
}^{\text{cx}}.$ So we can also replace $\mathrm{P}_{1,\leqslant\nu}%
^{\text{cx}}$ with $\mathrm{P}_{2,\leqslant\nu}^{\text{cx}}$ in the above theorems.

\subsection{Forward projection}

A key feature of our duality formulation of the Wasserstein projection in
stochastic order is that the property of the optimal dual solution is
inherited from the defining class of the stochastic order (e.g. the convexity
in the convex order case). The optimal dual solution, if exists, not only
gives rise to the optimal mapping to the projection (Theorem
\ref{thm_aproperty}), but also produces special properties of the optimal
mapping. Moreover, our duality formula allows us to handle the backward and
forward case in a unified manner. The backward convex order projection in the
previous theorem yields a special optimal mapping: a convex contraction. The
next theorem will give us the exact opposite of the backward case: a convex
expansion. Note that we even have a very precise relation between the optimal
mappings of the backward and forward convex order projections: they are
actually inverse to each other, which is somewhat surprising (see Section
\ref{sec_bf}).

\begin{theorem}
\label{thm_cx_prop}Suppose that $\mu\in P_{2}\left(  \mathbb{R}^{d}\right)  ,$
$\nu\in P_{2}^{ac}\left(  \mathbb{R}^{d}\right)  $ and $\varphi\in
\mathcal{A}_{\text{cx}}\cap C_{b,2}$ is the optimizer of the dual
$\mathcal{D}_{2}\left(  \mathrm{P}_{2,\mu\leqslant}^{\text{cx}},\nu\right)  $
obtained in Theorem \ref{thm_dual_att_cx}$.$ Let%
\[
\bar{\varphi}_{0}\left(  x\right)  =\frac{1}{2}\left\vert x\right\vert
^{2}-\frac{1}{2}\varphi\left(  x\right)  .
\]
Then

(i) $\bar{\varphi}_{0}^{\ast}$ is a convex expansion. In addition,
$\bar{\varphi}_{0}^{\ast}$ is uniformly convex.

(ii) $\left(  \nabla\bar{\varphi}_{0}^{\ast}\right)  _{\#}\nu$ is the unique
projection of $\nu$ onto $\mathrm{P}_{2,\mu\leqslant}^{\text{cx}}.$
\end{theorem}

\begin{proof}
By Theorem \ref{thm_uniqueness}, the forward projection is unique. Let
$\bar{\nu}$ be the unique forward convex order projection and $\pi\in
\Pi\left(  \bar{\nu},\nu\right)  $ such that%
\[
\mathcal{T}_{2}\left(  \mathrm{P}_{2,\mu\leqslant}^{\text{cx}},\nu\right)
=\mathcal{T}_{2}\left(  \bar{\nu},\nu\right)  =\int\left\vert x-y\right\vert
^{2}d\pi\left(  x,y\right)  .
\]
By Theorem \ref{thm_dual_cx}, the optimality of $\varphi$ and Theorem
\ref{thm_aproperty},%
\begin{align*}
\int\left\vert x-y\right\vert ^{2}d\pi\left(  x,y\right)   &  =\int%
\varphi\left(  x\right)  d\mu-\int Q_{\bar{2}}\left(  \varphi\right)  \left(
y\right)  d\nu\\
&  =\int\varphi\left(  x\right)  d\bar{\nu}-\int Q_{\bar{2}}\left(
\varphi\right)  \left(  y\right)  d\nu\\
&  =\int\left(  \varphi\left(  x\right)  -Q_{\bar{2}}\left(  \varphi\right)
\left(  y\right)  \right)  d\pi\left(  x,y\right)  .
\end{align*}
Hence%
\[
\int\left[  \left\vert x-y\right\vert ^{2}-\left(  \varphi\left(  x\right)
-Q_{\bar{2}}\left(  \varphi\right)  \left(  y\right)  \right)  \right]
d\pi\left(  x,y\right)  =0.
\]
Since the integrand is nonnegative, we have%
\begin{equation}
\left\vert x-y\right\vert ^{2}=\varphi\left(  x\right)  -Q_{\bar{2}}\left(
\varphi\right)  \left(  y\right)  ,\text{ }\pi\text{-}a.e.\text{ }\left(
x,y\right)  . \label{thm_cx_prop0}%
\end{equation}
Using Lemma \ref{lm_Q2Q2bar},%
\begin{equation}
Q_{\bar{2}}\left(  \varphi\right)  \left(  y\right)  =2\bar{\varphi}_{0}%
^{\ast}\left(  y\right)  -\left\vert y\right\vert ^{2},\text{ where }%
\bar{\varphi}_{0}\left(  x\right)  =\frac{1}{2}\left\vert x\right\vert
^{2}-\frac{1}{2}\varphi\left(  x\right)  . \label{thm_cx_prop1}%
\end{equation}
Since $Q_{\bar{2}}\left(  \varphi\right)  $ is convex for convex $\varphi$, we
get from $\left(  \ref{thm_cx_prop1}\right)  $ that $\bar{\varphi}_{0}^{\ast}$
is uniformly convex and $D^{2}\bar{\varphi}_{0}^{\ast}\geqslant Id$. Thus
$\bar{\varphi}_{0}^{\ast}$ is a convex expansion (Remark \ref{rmk_ctrexp}).
Finally, combining $\left(  \ref{thm_cx_prop0}\right)  $ and $\left(
\ref{thm_cx_prop1}\right)  ,$ we have%
\[
\left\vert x-y\right\vert ^{2}=\varphi\left(  x\right)  -\left(  2\bar
{\varphi}_{0}^{\ast}\left(  y\right)  -\left\vert y\right\vert ^{2}\right)
,\text{ }\pi\text{-}a.e.\text{ }\left(  x,y\right)  ,
\]
which can be rewrite as%
\[
x\cdot y=\bar{\varphi}_{0}\left(  x\right)  +\bar{\varphi}_{0}^{\ast}\left(
y\right)  ,\text{ }\pi\text{-}a.e.\text{ }\left(  x,y\right)  .
\]
Since $Q_{\bar{2}}\left(  \varphi\right)  \in L^{1}\left(  Y,d\nu\right)  ,$
it is finite for $\nu$-$a.e.$ $y.$ Hence $\bar{\varphi}_{0}^{\ast}$ is finite
for $\nu$-$a.e.$ $y.$ This together with the absolute continuity of $\nu$
implies that $\bar{\varphi}_{0}^{\ast}$ is $\nu$-$a.e.$ differentiable and%
\[
x=\nabla\bar{\varphi}_{0}^{\ast}\left(  y\right)  ,\text{ }\nu\text{-}%
a.e.\text{ }y.
\]
Therefore%
\[
x=\nabla\bar{\varphi}_{0}^{\ast}\left(  y\right)  ,\text{ }\pi\text{-}%
a.e.\text{ }\left(  x,y\right)  .
\]
This shows that $\bar{\nu}=\left(  \nabla\bar{\varphi}_{0}^{\ast}\right)
_{\#}\nu$ is the unique projection of $\nu$ onto $\mathrm{P}_{2,\mu\leqslant
}^{\text{cx}}.$
\end{proof}

In contrast to Theorem \ref{thm_cx_prop}, we need to assume the absolute
continuity of $\nu$ to ensure the uniqueness of the forward projection in
Theorem \ref{thm_cx_prop}, this results from the difference in the geometric
properties of backward and forward convex order cones, see Section
\ref{sec_unique} and\ Section \ref{sec_bf}.

It turns out the properties of the potential given in Theorem
\ref{thm_cx_prop} are also optimal. The follow complements the corresponding
result of backward convex order projection given in Theorem
\ref{thm_cx_bw_char}.

\begin{theorem}
\label{thm_cx_exp}Let $\mu\in P_{2}\left(  \mathbb{R}^{d}\right)  ,$ $\nu\in
P_{2}^{ac}\left(  \mathbb{R}^{d}\right)  .$ Then the following are equivalent.

(i) The projection of $\nu\,$onto$\ \mathrm{P}_{2,\mu\leqslant}^{\text{cx}}$
coincides with $\mu$.

(ii) There exists a convex expansion $\phi$ such that $\left(  \nabla
\phi\right)  _{\#}\nu=\mu.$
\end{theorem}

\begin{proof}
\textbf{1}. That (i) implies (ii) follows from Theorem \ref{thm_cx_prop} and
Lemma \ref{lm_HJ_cx_duality}.

\textbf{2}. Suppose that there exists a convex expansion $\phi$ such that
$\left(  \nabla\phi\right)  _{\#}\nu=\mu.$ Let%
\[
\varphi\left(  x\right)  =\left\vert x\right\vert ^{2}-2\phi^{\ast}\left(
x\right)  .
\]
Since $\phi$ is a convex expansion, hence $\varphi$ is convex. In addition, by
Lemma \ref{lm_HJ_cx_duality} $\phi^{\ast}\in C^{1},\ \nabla\phi^{\ast}$ is
$1$-Lipschitz, thus $\phi^{\ast}$\ has at most quadratic growth. Therefore
$\varphi\in\mathcal{A}_{\text{cx}}\cap C_{b,2}$. Since $\phi$ is convex,%
\[
\nabla\phi\left(  y\right)  \cdot y=\phi^{\ast}\left(  \nabla\phi\left(
y\right)  \right)  +\phi\left(  y\right)  ,\text{ }a.e.\text{ }y.
\]
Noting $\nu$ is absolutely continuous w.r.t. the Lebesgue measure, we get%
\[
\nabla\phi\left(  y\right)  \cdot y=\phi^{\ast}\left(  \nabla\phi\left(
y\right)  \right)  +\phi\left(  y\right)  ,\text{ }\nu\text{-}a.e.\text{ }y.
\]
Therefore%
\begin{align*}
\mathcal{T}_{2}\left(  \mu,\nu\right)   &  \leqslant\int\left\vert \nabla
\phi\left(  y\right)  -y\right\vert ^{2}d\nu=\int\left\vert \nabla\phi\left(
y\right)  \right\vert ^{2}d\nu-2\int\left(  \nabla\phi\left(  y\right)  \cdot
y\right)  d\nu+\int\left\vert y\right\vert ^{2}d\nu\\
&  =\int\left\vert x\right\vert ^{2}d\mu-2\int\left(  \phi\left(  y\right)
+\phi^{\ast}\left(  \nabla\phi\left(  y\right)  \right)  \right)  d\nu
+\int\left\vert y\right\vert ^{2}d\nu\\
&  =\int\left(  \left\vert x\right\vert ^{2}-2\phi^{\ast}\left(  x\right)
\right)  d\mu-\int\left(  2\phi\left(  y\right)  -\left\vert y\right\vert
^{2}\right)  d\nu.
\end{align*}
Using the definition of $\varphi$ and applying Lemma \ref{lm_HJ_cx_duality} to
the term, we continue writing%
\begin{align*}
\mathcal{T}_{2}\left(  \mu,\nu\right)   &  \leqslant\int\varphi\left(
x\right)  d\mu-\int Q_{\bar{2}}\left(  \varphi\right)  \left(  y\right)
d\nu\\
&  \leqslant\sup_{g\in\mathcal{A}_{\text{cx}}\cap\mathrm{S}_{b,2}}\left\{
\int g\left(  x\right)  d\mu-\int Q_{\bar{2}}\left(  g\right)  \left(
y\right)  d\nu\right\}  \text{ (}=\mathcal{D}_{2}\left(  \mathrm{P}%
_{2,\mu\leqslant}^{\text{cx}},\nu\right)  \text{)}\\
&  \leqslant\sup_{g\in L^{1}\left(  d\mu\right)  }\left\{  \int g\left(
x\right)  d\mu-\int Q_{\bar{2}}\left(  g\right)  \left(  y\right)
d\nu\right\}  =\mathcal{T}_{2}\left(  \mu,\nu\right)  .
\end{align*}
This shows that $\varphi$ is an optimizer of the dual, i.e.,%
\[
\mathcal{D}_{2}\left(  \mathrm{P}_{2,\mu\leqslant}^{\text{cx}},\nu\right)
=\int\varphi\left(  x\right)  d\mu-\int Q_{\bar{2}}\left(  \varphi\right)
\left(  y\right)  d\nu.
\]
Hence by Theorem \ref{thm_cx_prop},\ the image $\left(  \nabla\phi\right)
_{\#}\nu,$ which is $\mu$ by assumption, is the projection of $\nu$ onto
$\mathrm{P}_{2,\mu\leqslant}^{\text{cx}}.$ Therefore we have proved that the
projection of $\nu$ onto $\mathrm{P}_{2,\mu\leqslant}^{\text{cx}}$ is $\mu.$
\end{proof}

\section{\label{sec_bf}Backward projection versus forward projection}

In appearance, the backward and forward duality formulas correspond to the two
forms of classical Kantorovich duality where one can go from one form to the
other through a $c$-transform. People might thus be tempted to think that the
backward duality and forward duality for Wasserstein projections can be
bridged by a $c$-transform as in the classical case, hence are equivalent in
some sense. This, however, is not true in general, simply because the defining
class $\mathcal{A}\ $might not be invariant under the $c$-transforms $Q_{c}$
and $Q_{\bar{c}}$. That being said, the case of convex order is an exception,
the backward and forward convex order projections are indeed equivalent in an
appropriate sense. This result is somewhat surprising, since even in the
convex order case, although the backward cone is geodesically convex, the
forward cone is not, see Example \ref{eg_cx_cone}.

\subsection{Supports of measures in convex order}

For any two measures in convex order, their supports, informally speaking, are
increasing with convex order. This implies that for backward projection
problem $\left(  \ref{img_backproj}\right)  $, the support of the backward
projection $\bar{\mu}$ is implicitly known to be contained in the support of
$\nu.$ Forward projection does not enjoy this property. This is made precise
in the following simple result, for which we provide a proof for completeness.

\begin{lemma}
\label{lm_cxsupp_inc}Let $\mu$, $\nu\in P\left(  \mathbb{R}^{d}\right)  $ and
$\mu\leqslant_{\text{cx}}\nu.$ Then
\[
clconv\left(  \text{supp}\left(  \mu\right)  \right)  \subset clconv\left(
\text{supp}\left(  \nu\right)  \right)  ,
\]
where $clconv\left(  \cdot\right)  $ denotes the closure of the convex hull of
a given set.
\end{lemma}

\begin{proof}
Suppose to the contrary that the conclusion is not true. Then there exists a
nonempty open ball $B\left(  x_{0}\right)  $ such that $\mu\left(  B\left(
x_{0}\right)  \right)  >0$ and $\bar{B}\left(  x_{0}\right)  \cap
clconv\left(  \text{supp}\left(  \nu\right)  \right)  =\emptyset.$ Then there
is a linear function $l\left(  x\right)  $ which strictly separates
$clconv\left(  \text{supp}\left(  \nu\right)  \right)  $ and $\bar{B}\left(
x_{0}\right)  $:%
\[
l\left(  x\right)  >0,\text{ }x\in\bar{B}\left(  x_{0}\right)  \text{;
}l\left(  x\right)  <0,\text{ }x\in clconv\left(  \text{supp}\left(
\nu\right)  \right)  .
\]
Consider the convex function%
\[
\varphi\left(  x\right)  =\max\left(  l\left(  x\right)  ,0\right)  ,\text{
}x\in\mathbb{R}^{d}.
\]
It is easy to see that%
\[
\int\varphi\left(  x\right)  d\mu>0=\int\varphi\left(  x\right)  d\nu.
\]
This contradicts the assumption that $\mu\leqslant_{\text{cx}}\nu.$
\end{proof}

\subsection{Convexity of the cones}

Given $\mu,$ $\nu\in P_{2}\left(  \mathbb{R}^{d}\right)  .$ We discuss the
geodesic convexity of the backward cone $\mathrm{P}_{2,\leqslant\nu
}^{\text{cx}}$\ and forward cone $\mathrm{P}_{2,\mu\leqslant}^{\text{cx}}$. In
view of \cite[Proposition 9.3.2]{ambrosio2008gradient}, $\mathrm{P}%
_{2,\leqslant\nu}^{\text{cx}}$ is convex along generalized geodesics, hence
also convex along geodesics. The forward cone $\mathrm{P}_{2,\mu\leqslant
}^{\text{cx}}$, on the other hand, is generally not geodesically convex as
illustrated by the example below which is inpired by
\cite{bruckerhoff2021instability}.

\begin{figure}[h]
\includegraphics[width=0.6\linewidth]{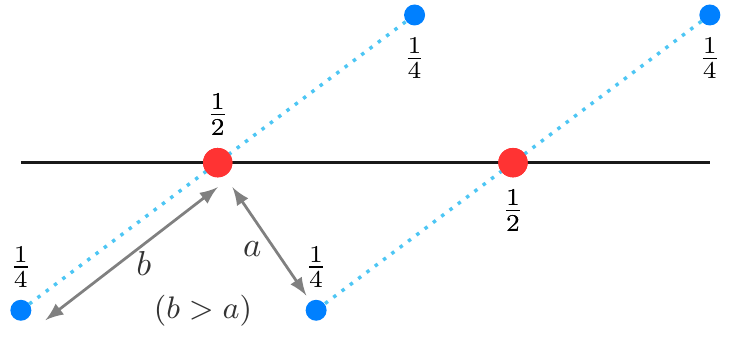}\caption{Probabilities
(solid red and blue) in convex order}%
\label{fig:eg1}%
\end{figure}

\begin{figure}[h]
\includegraphics[width=0.6\linewidth]{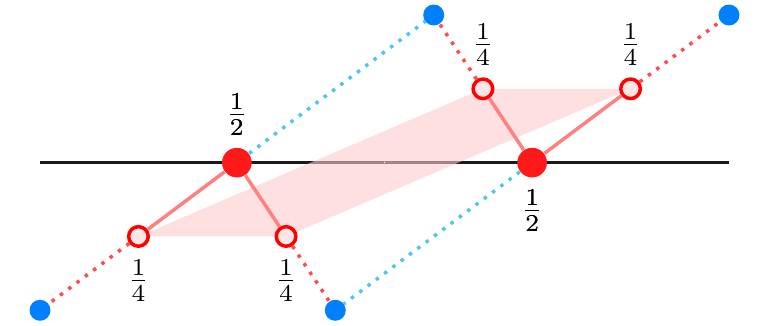}\caption{The
optimal transport from the probability on red balls to the one on blue balls.
Cirles indicate displacement interpolation.}%
\label{fig:eg2}%
\end{figure}

\begin{example}
\label{eg_cx_cone}Let%
\[
\mu=\text{probability supported on solid red balls in Figure \ref{fig:eg1},}%
\]
and%
\[
\nu=\text{probability supported on solid blue balls in Figure \ref{fig:eg1}.}%
\]
Consider the cone $\mathrm{P}_{2,\mu\leqslant}^{\text{cx}}.$\ Clearly $\nu
\in\mathrm{P}_{2,\mu\leqslant}^{\text{cx}}.$ In Figure \ref{fig:eg1}, the
indicated distance $b$ is larger than $a.$ When the Wasserstein distance
between $\mu$ and $\nu$ is computed, it is easy to see that the optimal path
to transport $\mu$ to $\nu$ is along the dotted red lines as shown in Figure
\ref{fig:eg2}, where an intermmediate displacement interpolation, say $\left[
\mu,\nu\right]  _{s}$ with $s\in\left(  0,1\right)  $, is given and supported
on the red circles in Figure \ref{fig:eg2}. The closure of the convex hull of
supp$\left(  \left[  \mu,\nu\right]  _{s}\right)  $ is indicated by shaded
area, which does not include $clconv\left(  \text{supp}\left(  \mu\right)
\right)  $. In view of Lemma \ref{lm_cxsupp_inc}, the relation $\mu
\leqslant_{\text{cx}}\left[  \mu,\nu\right]  _{s}$ does not hold, in other
words, $\left[  \mu,\nu\right]  _{s}$ is not in the cone $\mathrm{P}%
_{2,\mu\leqslant}^{\text{cx}}.$ Therefore $\mathrm{P}_{2,\mu\leqslant
}^{\text{cx}}$ is not geodesically convex.
\end{example}

\subsection{Relation between backward and forward solution}

In the following we show that the backward and forward convex order projection
costs are equal and the optimal backward and forward mappings are inverse to
each other. Given the fact that the backward and forward convex order cones
possess distinct geometric properties as we have already seen, the equality
and inverse relations are surprising.

\begin{theorem}
\label{thm_cx_bfequ}Let $\mu\in P_{2}\left(  \mathbb{R}^{d}\right)  ,$ $\nu\in
P_{2}\left(  \mathbb{R}^{d}\right)  $. Then%
\[
\mathcal{T}_{2}\left(  \mathrm{P}_{2,\mu\leqslant}^{\text{cx}},\nu\right)
=\mathcal{D}_{2}\left(  \mathrm{P}_{2,\mu\leqslant}^{\text{cx}},\nu\right)
=\mathcal{D}_{2}\left(  \mu,\mathrm{P}_{2,\leqslant\nu}^{\text{cx}}\right)
=\mathcal{T}_{2}\left(  \mu,\mathrm{P}_{2,\leqslant\nu}^{\text{cx}}\right)  .
\]
Moreover the following are true.

(i) Let $\varphi^{f}\in\mathcal{A}_{\text{cx}}\cap C_{b,2}$ be the optimal
dual solution for forward projection obtained in Theorem \ref{thm_dual_att_cx}%
, then%
\[
Q_{\bar{2}}\left(  \varphi^{f}\right)  \in\mathcal{A}_{\text{cx}}\cap
L^{1}\left(  d\nu\right)  \text{ is an optimal dual solution for backward
projection}.
\]

(ii) Let $\varphi^{b}\in\mathcal{A}_{\text{cx}}\cap L^{1}\left(  d\nu\right)
$ be the optimal dual solution for backward projection obtained in Theorem
\ref{thm_dual_att_cx_bw}, then%
\[
Q_{2}\left(  \varphi^{b}\right)  \in\mathcal{A}_{\text{cx}}\cap\mathrm{S}%
_{b,2}\text{ is an optimal dual solution for forward projection}.
\]

\end{theorem}

\begin{proof}
It suffices to verify that%
\begin{equation}
\mathcal{D}_{2}\left(  \mu,\mathrm{P}_{2,\leqslant\nu}^{\text{cx}}\right)
=\mathcal{D}_{2}\left(  \mathrm{P}_{2,\mu\leqslant}^{\text{cx}},\nu\right)  .
\label{thm_cx_bfequ0}%
\end{equation}
Recall the dualities in Theorem \ref{thm_dual_cx},%
\begin{equation}
\mathcal{D}_{2}\left(  \mu,\mathrm{P}_{2,\leqslant\nu}^{\text{cx}}\right)
=\sup_{\varphi\in\mathcal{A}_{\text{cx}}\cap L^{1}\left(  d\nu\right)
}\left\{  \int Q_{2}\left(  \varphi\right)  d\mu-\int\varphi d\nu\right\}  ,
\label{thm_cx_bfequ1}%
\end{equation}
and%
\begin{equation}
\mathcal{D}_{2}\left(  \mathrm{P}_{2,\mu\leqslant}^{\text{cx}},\nu\right)
=\sup_{\varphi\in\mathcal{A}_{\text{cx}}\cap\mathrm{S}_{b,2}}\left\{
\int\varphi d\mu-\int Q_{\bar{2}}\left(  \varphi\right)  d\nu\right\}  .
\label{thm_cx_bfequ2}%
\end{equation}
Since for $\varphi\in\mathcal{A}_{\text{cx}}\cap\mathrm{S}_{b,2},$ $Q_{\bar
{2}}\left(  \varphi\right)  >-\infty$ is convex, thus bounded from below by
its supporting plane. Hence $\int Q_{\bar{2}}\left(  \varphi\right)  d\nu$ is
bounded from below. Since $\int Q_{\bar{2}}\left(  \varphi\right)  d\nu
=\infty$ would not contribute to the supremum in $\left(  \ref{thm_cx_bfequ2}%
\right)  $, we can restrict to those $\varphi\in\mathcal{A}_{\text{cx}}%
\cap\mathrm{S}_{b,2}$ such that $\int Q_{\bar{2}}\left(  \varphi\right)  d\nu$
is finite, i.e. $Q_{\bar{2}}\left(  \varphi\right)  \in L^{1}\left(
d\nu\right)  .$ Therefore we can write%
\begin{equation}
\mathcal{D}_{2}\left(  \mathrm{P}_{2,\mu\leqslant}^{\text{cx}},\nu\right)
=\sup_{\substack{\varphi\in\mathcal{A}_{\text{cx}}\cap\mathrm{S}%
_{b,2}\\Q_{\bar{2}}\left(  \varphi\right)  \in L^{1}\left(  d\nu\right)
}}\left\{  \int\varphi d\mu-\int Q_{\bar{2}}\left(  \varphi\right)
d\nu\right\}  . \label{thm_cx_bfequ3}%
\end{equation}
Then, for any $\varphi$ admissible to the supremum of $\left(
\ref{thm_cx_bfequ3}\right)  $, it is legitimate to substitute $\varphi$ with
$Q_{\bar{2}}\left(  \varphi\right)  $ in the supremum of $\left(
\ref{thm_cx_bfequ1}\right)  $ and write%
\begin{equation}
\mathcal{D}_{2}\left(  \mu,\mathrm{P}_{2,\leqslant\nu}^{\text{cx}}\right)
\geqslant\int Q_{2}\left(  Q_{\bar{2}}\left(  \varphi\right)  \right)
d\mu-\int Q_{\bar{2}}\left(  \varphi\right)  d\nu\geqslant\int\varphi
d\mu-\int Q_{\bar{2}}\left(  \varphi\right)  d\nu. \label{thm_cx_bfequ4}%
\end{equation}
Since $\varphi$ runs over all functions admissible to the supremum of $\left(
\ref{thm_cx_bfequ3}\right)  ,$ we get%
\[
\mathcal{D}_{2}\left(  \mu,\mathrm{P}_{2,\leqslant\nu}^{\text{cx}}\right)
\geqslant\mathcal{D}_{2}\left(  \mathrm{P}_{2,\mu\leqslant}^{\text{cx}}%
,\nu\right)  .
\]
On the other hand, for any $\varphi\in\mathcal{A}_{\text{cx}}\cap L^{1}\left(
d\nu\right)  ,$ $Q_{2}\left(  \varphi\right)  $ is convex and $\varphi$ is
$\nu$-$a.e.$ finite, thus $Q_{2}\left(  \varphi\right)  \in\mathcal{A}%
_{\text{cx}}\cap\mathrm{S}_{b,2}.$ So it is legitimate to substitute $\varphi$
with $Q_{2}\left(  \varphi\right)  $ in the supremum of $\left(
\ref{thm_cx_bfequ2}\right)  $ and write%
\begin{equation}
\int Q_{2}\left(  \varphi\right)  d\mu-\int\varphi d\nu\leqslant\int
Q_{2}\left(  \varphi\right)  d\mu-\int Q_{\bar{2}}\left(  Q_{2}\left(
\varphi\right)  \right)  d\nu\leqslant\mathcal{D}_{2}\left(  \mathrm{P}%
_{2,\mu\leqslant}^{\text{cx}},\nu\right)  . \label{thm_cx_bfequ5}%
\end{equation}
Since this holds for any $\varphi\in\mathcal{A}_{\text{cx}}\cap L^{1}\left(
d\nu\right)  ,$ we have%
\[
\mathcal{D}_{2}\left(  \mu,\mathrm{P}_{2,\leqslant\nu}^{\text{cx}}\right)
\leqslant\mathcal{D}_{2}\left(  \mathrm{P}_{2,\mu\leqslant}^{\text{cx}}%
,\nu\right)  .
\]
Thus $\left(  \ref{thm_cx_bfequ0}\right)  $ is proved. To see (i) and (ii), we
note that, by the optimality of $\varphi^{f},$%
\[
\mathcal{D}_{2}\left(  \mathrm{P}_{2,\mu\leqslant}^{\text{cx}},\nu\right)
=\int\varphi^{f}d\mu-\int Q_{\bar{2}}\left(  \varphi^{f}\right)  d\nu.
\]
Since the dual value is finite and $\varphi^{f}\in\mathcal{A}_{\text{cx}}\cap
C_{b,2},$ $Q_{\bar{2}}\left(  \varphi^{f}\right)  \in\mathcal{A}_{\text{cx}%
}\cap L^{1}\left(  d\nu\right)  .$ So we can insert $\varphi^{f}$ into
$\left(  \ref{thm_cx_bfequ4}\right)  $\ to saturate the inequalities.
Therefore $Q_{\bar{2}}\left(  \varphi^{f}\right)  $ is an optimal solution for
$\mathcal{D}_{2}\left(  \mu,\mathrm{P}_{2,\leqslant\nu}^{\text{cx}}\right)  $.
Similarly $\varphi^{b}\in\mathcal{A}_{\text{cx}}\cap L^{1}\left(  d\nu\right)
$ implies $Q_{2}\left(  \varphi^{b}\right)  \in\mathcal{A}_{\text{cx}}%
\cap\mathrm{S}_{b,2}.$ Therefore we can substitute $\varphi^{b}$ into $\left(
\ref{thm_cx_bfequ5}\right)  $\ to see that $Q_{2}\left(  \varphi^{b}\right)  $
is optimal for $\mathcal{D}_{2}\left(  \mathrm{P}_{2,\mu\leqslant}^{\text{cx}%
},\nu\right)  .$
\end{proof}

\begin{remark}
Either following the same argument as above, or using the above result
together with Theorem \ref{thm_dual_cx}, we have$\ \mathcal{D}_{2}\left(
\mathrm{P}_{1,\mu\leqslant}^{\text{cx}},\nu\right)  =\mathcal{D}_{2}\left(
\mu,\mathrm{P}_{1,\leqslant\nu}^{\text{cx}}\right)  .$ Although Theorem
\ref{thm_cx_bfequ} is proved for quadratic cost $c=\left\vert x-y\right\vert
^{2}$, it is easy to generalize to other convex cost functions of the form
$c=h\left(  x-y\right)  $.
\end{remark}

The equality $\mathcal{T}_{k}\left(  \mu,\mathrm{P}_{k,\leqslant\nu
}^{\text{cx}}\right)  =\mathcal{T}_{k}\left(  \mathrm{P}_{k,\mu\leqslant
}^{\text{cx}},\nu\right)  $\ with $k>1$\ is proved in
\cite{alfonsi2020sampling} using a different method based on the primal
formulations of the projection problems. In one dimension, it is shown there
that the optimal mappings for backward and forward convex order projections
are inverse to each other. The following result confirms these properties on
the optimal mappings in general dimensions and sheds more lights on the deep
connection between backward and forward convex order projection.

\begin{corollary}
\label{cor_cx_bfequ}Let $\mu\in P_{2}\left(  \mathbb{R}^{d}\right)  ,$ $\nu\in
P_{2}^{ac}\left(  \mathbb{R}^{d}\right)  $, $\bar{\mu}$ the optimizer of
$\mathcal{T}_{2}\left(  \mu,\mathrm{P}_{2,\leqslant\nu}^{\text{cx}}\right)  $
and $\bar{\nu}$ the optimizer of $\mathcal{T}_{2}\left(  \mathrm{P}%
_{2,\mu\leqslant}^{\text{cx}},\nu\right)  $. There exists a real-valued convex
function $\Phi\in C^{1}$ such that%
\[
\left(  \nabla\Phi\right)  _{\#}\mu=\bar{\mu},\text{ }\left(  \nabla\Phi
^{\ast}\right)  _{\#}\nu=\bar{\nu}.
\]
Moreover $\bar{\nu}\in P_{2}^{ac}\left(  \mathbb{R}^{d}\right)  $.
\end{corollary}

\begin{proof}
Let $\varphi^{f}$ be the optimal dual solution of the forward convex order
projection obtained in Theorem \ref{thm_dual_att_cx}. Recall that $\varphi
^{f}$ is convex continuous and $\varphi^{f}\leqslant\left\vert x\right\vert
^{2}.$ By Theorem \ref{thm_cx_bfequ},\ $Q_{\bar{2}}\left(  \varphi^{f}\right)
$ is an optimal dual solution for backward convex order projection$.$ By
virtue of Theorem \ref{thm_cx_prop}, the unique projection of $\nu$ onto
$\mathrm{P}_{2,\mu\leqslant}^{\text{cx}},\nu$ is given by $\left(  \nabla
\bar{\varphi}_{0}^{\ast}\right)  _{\#}\nu,$ where%
\begin{equation}
\bar{\varphi}_{0}\left(  x\right)  \triangleq\frac{1}{2}\left\vert
x\right\vert ^{2}-\frac{1}{2}\varphi^{f}\left(  x\right)  .
\label{cor_cx_bfequ1}%
\end{equation}
In particular, $D^{2}\bar{\varphi}_{0}^{\ast}\geqslant Id.$ Now according to
Theorem \ref{thm_cx_bw_prop}$,$ the unique projection of $\mu$ onto
$\mathrm{P}_{2,\leqslant\nu}^{\text{cx}}$ is given by $\left(  \nabla
\varphi_{0}^{\ast}\right)  _{\#}\mu,$ where%
\begin{equation}
\varphi_{0}\left(  y\right)  \triangleq\frac{1}{2}\left\vert y\right\vert
^{2}+\frac{1}{2}Q_{\bar{2}}\left(  \varphi^{f}\right)  \left(  y\right)  .
\label{cor_cx_bfequ2}%
\end{equation}
Substituting $\left(  \ref{cor_cx_bfequ1}\right)  $ into $\left(
\ref{cor_cx_bfequ2}\right)  $, we obtain%
\begin{align*}
\varphi_{0}\left(  y\right)   &  =\frac{1}{2}\left\vert y\right\vert
^{2}+\frac{1}{2}Q_{\bar{2}}\left(  \left[  \left\vert x\right\vert ^{2}%
-2\bar{\varphi}_{0}\left(  x\right)  \right]  \right)  \left(  y\right) \\
&  =\frac{1}{2}\left\vert y\right\vert ^{2}+\frac{1}{2}\sup_{x\in
\mathbb{R}^{d}}\left\{  \left\vert x\right\vert ^{2}-2\bar{\varphi}_{0}\left(
x\right)  -\left\vert x-y\right\vert ^{2}\right\} \\
&  =\bar{\varphi}_{0}^{\ast}\left(  y\right)  .
\end{align*}
Let $\Phi=\bar{\varphi}_{0}^{\ast\ast}.$ Then $\Phi$ is the desired
real-valued convex function. Furthermore, $\Phi^{\ast}=\bar{\varphi}_{0}%
^{\ast}$ is uniformly convex, so the forward projection $\bar{\nu}$ under the
mapping $\nabla\Phi^{\ast}$ is absolutely continuous and $\Phi=\Phi^{\ast\ast
}\in C^{1}.$
\end{proof}

\begin{remark}
Convexity is crucial in the proof of Theorem \ref{thm_cx_bfequ} and Corollary
\ref{cor_cx_bfequ}. We do not expect other stochastic orders (e.g. the
subharmonic order) to have the above equality and the inverse relation between
backward and forward projection.
\end{remark}

\section{Subharmonic order projections}

In this section we consider another important instantce of stochastic order,
i.e. subharmonic order.

\begin{definition}
\label{def_subH}An upper semicontinuous function $\psi$ with values in
$\left[  -\infty,\infty\right)  $\ is subharmonic on an open set $X$, if the
sub-mean value inequality%
\[
\psi\left(  x_{0}\right)  \leqslant\frac{1}{\omega_{d}\left(  r\right)  }%
\int_{\partial B_{r}\left(  x_{0}\right)  }\psi d\sigma_{r}%
\]
holds for any ball $B_{r}\left(  x_{0}\right)  $ contained in $X$, where
$\omega_{d}\left(  r\right)  $ is the surface area of the ball $B_{r}\left(
x_{0}\right)  $.
\end{definition}

If a function $\psi$ satisfies the above sub-mean value properties, but not
necessarily upper semicontinuous, then it is called \textit{almost
subharmonic}. This terminology is justified by the fact that every almost
subharmonic function equals an (upper semicontinuous) subharmonic
function\ almost everywhere \cite{szpilrajn1933remarques}.

\begin{definition}
[\textbf{Subharmonic order}]Let $X\subset\mathbb{R}^{d}$ be a convex bounded
open set and $\mu,$ $\nu\in P\left(  X\right)  ,$ we call $\mu$ smaller than
$\nu$ in subharmonic order, denoted by $\mu\leqslant_{\text{sh}}\nu,$ if the
inequality
\begin{equation}
\int_{X}\psi d\mu\leqslant\int_{X}\psi d\nu, \label{inq_SHorder}%
\end{equation}
holds for all subharmonic $\psi\in C_{b}\left(  X\right)  .$
\end{definition}

Subharmonic order is sufficient to induce a Brownian transport between $\mu$
and $\nu$ (see e.g. \cite[Proposition 3.4]{ghoussoub2020optimal}). By
performing convolutions with smooth radial kernels, each subharmonic function
can be approximated in any compact subsets of $X$ by decreasing sequence of
smooth subharmonic functions. Similar to Lemma \ref{lm_cx_bdbelow}, one can
also restrict the integrands of $\left(  \ref{inq_SHorder}\right)  $\ to those
which are bounded from below.

Define%
\[
\mathcal{A}_{\text{sh}}=\left\{  \psi:\psi\text{ subharmonic, bounded from
below}\right\}  .
\]
For notational simplicity, the domain of definition for functions in
$\mathcal{A}_{\text{sh}}$ is not explicitly specified and will be clear from
the context.

Let $\mu\in P\left(  X\right)  ,$ $\nu\in P\left(  Y\right)  .$ Define the
backward and forward subharmonic order cone,%
\[
\mathrm{P}_{\leqslant\nu}^{\text{sh}}=\left\{  \eta\in P\left(  Y\right)
:\eta\leqslant_{\text{sh}}\nu\right\}  .
\]
and%
\[
\mathrm{P}_{\mu\leqslant}^{\text{sh}}=\left\{  \xi\in P\left(  X\right)
:\mu\leqslant_{\text{sh}}\xi\right\}  .
\]

\subsection{The duality theorems}

\begin{theorem}
\label{thm_dual_sh}Let $X,$ $Y\subset\mathbb{R}^{d}$ be bounded convex open
subsets$,$ $\mu\in P\left(  X\right)  ,$ $\nu\in P\left(  Y\right)  .$ Then

(i) for backward subharmonic order projection it holds that%
\[
\mathcal{T}_{2}\left(  \mu,\mathrm{P}_{\leqslant\nu}^{\text{sh}}\right)
=\sup_{\psi\in\mathcal{A}_{\text{sh}}\cap C_{b}\left(  Y\right)  }\left\{
\int_{X}Q_{2}\left(  \psi\right)  d\mu-\int_{Y}\psi d\nu\right\}  .
\]

(ii) for forward subharmonic order projection it holds that%
\[
\mathcal{T}_{2}\left(  \mathrm{P}_{\mu\leqslant}^{\text{sh}},\nu\right)
=\sup_{\psi\in\mathcal{A}_{\text{sh}}\cap C_{b}\left(  X\right)  }\left\{
\int_{X}\psi d\mu-\int_{Y}Q_{\bar{2}}\left(  \psi\right)  d\nu\right\}  .
\]
Here $Q_{2}\left(  \cdot\right)  ,$ $Q_{\bar{2}}\left(  \cdot\right)  $ are
defined in Theorem \ref{thm_dual_cx}. In both (i) and (ii), $\psi
\in\mathcal{A}_{\text{sh}}\mathcal{\cap}C_{b}$ can be relaxed to $\psi
\in\mathcal{A\cap}\mathrm{S}_{b}$.
\end{theorem}

\begin{proof}
The proof is completed by applying Theorem \ref{thm_dual_bwg} and Theorem
\ref{thm_dual_fwg} to the defining class $\mathcal{A}_{\text{sh}}\cap C_{b}.$
\end{proof}

\section{\label{sec_att_sh}Dual attainment for subharmonic order projections}

Obtaining a solution in the required subharmonic function class with
appropriate regularity is challenging. One difficulty lies in the definition
of subharmonic function. To determine whether or not a function is
subharmonic, one usually needs to test sub-mean value property over spheres or
balls. This implicitly uses Lebesgue measure as a reference. In the dualities
of subharmonic order projections, subharmonic functions have to interact with
arbitrary probability measures. To be able to traverse from an arbitrary
probability measure to the Lebesgue measure, assumptions such as absolute
continuity, bounds away from zero are usually unavoidable.

Another property missing from subharmonic functions involves the transforms
$Q_{2}\left(  \cdot\right)  ,$ $Q_{\bar{2}}\left(  \cdot\right)  .$ These
transforms preserve convexity, which plays a key role in the convex order
case. However, $Q_{2}\left(  \cdot\right)  ,$ $Q_{\bar{2}}\left(
\cdot\right)  $ not always preserve subharmonicity for general subharmonic
functions. Subharmonicity is only preserved by $Q_{\bar{2}}\left(
\cdot\right)  $ in subdomain, this is proved by \cite{petrunin2003harmonic} in
Alexandrov space, which includes Euclidean space as a special case
\cite[Example 4.2.1]{burago2001course}. Therefore, the double convexification
trick, which we use in the attainment for convex order projections, no longer
work in the subharmonic order case.

Despite these difficulties, we are able to obtain optimal dual solutions for
subharmonic order projections and a weak parallel of the convex contraction
and expansion, which we call Laplacian contraction and Laplacian expansion,
see Definition \ref{def_Lap_contraction} and Definition
\ref{def_Lap_expansion}.

We will first prove attainment of the duality. The method used here seems to
have the weakest assumptions on the measures.

\subsection{Backward projection}

\begin{theorem}
\label{thm_dual_sh_bw}Let $X,$ $Y$ be bounded convex open subsets of
$\mathbb{R}^{d}$, $\mu\in P\left(  X\right)  $ and $\nu\in P^{ac}\left(
Y\right)  $. Assume that the density of $\nu$ is bounded away from zero in
$Y$.\ Then there exists $\psi_{0}\in\mathcal{A}_{\text{sh}}\cap L^{1}\left(
Y,d\nu\right)  $ which is bounded from below such that
\begin{align*}
\mathcal{D}_{2}\left(  \mu,\mathrm{P}_{\leqslant\nu}^{\text{sh}}\right)   &
\triangleq\sup_{\psi\in\mathcal{A}_{\text{sh}}\cap C_{b}\left(  Y\right)
}\left\{  \int_{X}Q_{2}\left(  \psi\right)  d\mu-\int_{Y}\psi d\nu\right\} \\
&  \leqslant\int_{X}Q_{2}\left(  \psi_{0}\right)  \left(  x\right)  d\mu
-\int_{Y}\psi_{0}\left(  y\right)  d\nu.
\end{align*}
Denote by $\bar{\mu}$ the backward subharmonic order projection of $\mu$ onto
$\mathrm{P}_{\leqslant\nu}^{\text{sh}}.$ If in addition%
\begin{equation}
\int_{Y}\psi_{0}\left(  y\right)  d\bar{\mu}\leqslant\int_{Y}\psi_{0}\left(
y\right)  d\nu, \label{thm_dual_sh_attL00}%
\end{equation}
then $\psi_{0}$ achieves the optimal dual value.
\end{theorem}

\begin{proof}
Since $\mathcal{T}_{2}\left(  \mu,\mathrm{P}_{\leqslant\nu}^{\text{sh}%
}\right)  $ is finite (ref. Remark \ref{rmk_finiteness}),\ the optimal dual
$\mathcal{D}_{2}\left(  \mu,\mathrm{P}_{\leqslant\nu}^{\text{sh}}\right)  $ is
finite. We can also write%
\begin{equation}
\mathcal{D}_{2}\left(  \mu,\mathrm{P}_{\leqslant\nu}^{\text{sh}}\right)
=\sup_{\psi\in\mathcal{A}_{\text{sh}}\cap C_{b}\left(  Y\right)  }\left\{
\int_{X}Q_{2}\left(  Q_{\bar{2}}\left(  Q_{2}\left(  \psi\right)  \right)
\right)  \left(  x\right)  d\mu-\int_{Y}\psi\left(  y\right)  d\nu\right\}  .
\label{thm_dual_sh_attL0}%
\end{equation}
Let $\psi_{n}\in\mathcal{A}_{\text{sh}}\cap C_{b}\left(  Y\right)  $ be a
maximizing sequence and denote $\bar{\psi}_{n}=Q_{\bar{2}}\left(  Q_{2}\left(
\psi_{n}\right)  \right)  .$ Since the measures sit in a compact set
containing both $X$ and $Y$, the cost function is Lipschitz continuous there.
Thus $\bar{\psi}_{n}$ and $Q_{2}\left(  \bar{\psi}_{n}\right)  $ are Lipschitz
continuous with a common Lipschitz constant inherited from the cost function.
When necessary, we can regard these functions as defined on $\bar{Y}$ and
$\bar{X}$ by extending them to the boundaries. By adding constant(s), we may
assume without loss of generality that%
\[
\min_{y\in Y}\bar{\psi}_{n}\left(  y\right)  =0,\text{ }\forall n.
\]
Now it is readily seen that $\bar{\psi}_{n}$ is uniformly bounded in $Y$. It
follows that $Q_{2}\left(  \bar{\psi}_{n}\right)  $ is also uniformly bounded
in $X$. Therefore the sequence $\bar{\psi}_{n},$ and so $Q_{2}\left(
\bar{\psi}_{n}\right)  $, are uniformly bounded and equicontinuous. Hence, by
Arzela-Ascoli theorem, we may suppose that it holds uniformly that%
\[
\bar{\psi}_{n}\rightarrow\bar{\psi}_{0}\text{ and }Q_{2}\left(  \bar{\psi}%
_{n}\right)  \rightarrow Q_{2}\left(  \bar{\psi}_{0}\right)  ,\text{ as
}n\rightarrow\infty.
\]
where $\bar{\psi}_{0}$ is Lipschitz continuous on $Y.$ Hence%
\[
\int_{X}Q_{2}\left(  \bar{\psi}_{n}\right)  \left(  x\right)  d\mu
\rightarrow\int_{X}Q_{2}\left(  \bar{\psi}_{0}\right)  \left(  x\right)
d\mu,\text{ as }n\rightarrow\infty.
\]
It follows that $\int Q_{2}\left(  \bar{\psi}_{n}\right)  d\mu$ is bounded, so
we obtain from $\left(  \ref{thm_dual_sh_attL0}\right)  \ $and the boundedness
of the optimal dual value that $\int\psi_{n}d\nu$ is bounded. Note%
\begin{equation}
0\leqslant\bar{\psi}_{n}\left(  y\right)  \leqslant\psi_{n}\left(  y\right)
,\text{ }\forall y. \label{thm_dual_sh_attL1}%
\end{equation}
By assumption, there exists $a_{0}>0$ such that%
\[
a_{0}\int_{Y}\psi_{n}\left(  y\right)  dy\leqslant\int_{Y}\psi_{n}\left(
y\right)  d\nu\text{,}%
\]
which implies $\int\psi_{n}dy=\int\left\vert \psi_{n}\right\vert dy$ is
bounded. Then by Lemma \ref{lm_convg_subH}, there exists a subharmonic
function $\psi_{0}\in L^{1}\left(  Y,dy\right)  $ such that, up to a
subsequence,%
\begin{equation}
\psi_{n}\left(  y\right)  \rightarrow\psi_{0}\left(  y\right)  \text{,
}a.e.\text{ }y. \label{thm_dual_sh_attL2}%
\end{equation}
It follows that%
\[
\psi_{n}\left(  y\right)  \rightarrow\psi_{0}\left(  y\right)  \text{, }%
\nu\text{-}a.e.\text{ }y.
\]
By Fatou lemma%
\[
\liminf_{n\rightarrow\infty}\int_{Y}\psi_{n}\left(  y\right)  d\nu
\geqslant\int_{Y}\psi_{0}\left(  y\right)  d\nu.
\]
Therefore $\psi_{0}\in L^{1}\left(  Y,d\nu\right)  $. Using $\left(
\ref{thm_dual_sh_attL1}\right)  \left(  \ref{thm_dual_sh_attL2}\right)  ,$%
\[
0\leqslant\bar{\psi}_{0}\left(  y\right)  \leqslant\psi_{0}\left(  y\right)
\text{, }a.e.\text{ }y.
\]
Since $\bar{\psi}_{0}$ is continuous and $\psi_{0}$ is upper semicontinuous,
we obtain,%
\[
0\leqslant\bar{\psi}_{0}\left(  y\right)  \leqslant\psi_{0}\left(  y\right)
,\text{ }\forall y.
\]
Thus the subharmonic function $\psi_{0}$ is bounded from below and we obtain%
\begin{align*}
\mathcal{D}_{2}\left(  \mu,\mathrm{P}_{\leqslant\nu}^{\text{sh}}\right)   &
=\lim_{n\rightarrow\infty}\left\{  \int_{X}Q_{2}\left(  \bar{\psi}_{n}\right)
\left(  x\right)  d\mu-\int_{Y}\psi_{n}\left(  y\right)  d\nu\right\} \\
&  \leqslant\limsup_{n\rightarrow\infty}\int_{X}Q_{2}\left(  \bar{\psi}%
_{n}\right)  \left(  x\right)  d\mu-\liminf_{n\rightarrow\infty}\int_{Y}%
\psi_{n}\left(  y\right)  d\nu\\
&  =\int_{X}Q_{2}\left(  \bar{\psi}_{0}\right)  \left(  x\right)  d\mu
-\int_{Y}\psi_{0}\left(  y\right)  d\nu\\
&  \leqslant\int_{X}Q_{2}\left(  \psi_{0}\right)  \left(  x\right)  d\mu
-\int_{Y}\psi_{0}\left(  y\right)  d\nu.
\end{align*}
If $\psi_{0}$ satisfies $\left(  \ref{thm_dual_sh_attL00}\right)  $, then the
above inequality continues as%
\[
\int_{X}Q_{2}\left(  \psi_{0}\right)  \left(  x\right)  d\mu-\int_{Y}\psi
_{0}\left(  y\right)  d\nu\leqslant\mathcal{D}_{2}\left(  \mu,\mathrm{P}%
_{\leqslant\nu}^{\text{sh}}\right)  .
\]
Therefore the optimal dual value is attained at $\psi_{0}.$
\end{proof}

\begin{remark}
The assumption $\left(  \ref{thm_dual_sh_attL00}\right)  $ is added because a
priori the subharmonic function $\psi_{0}\in\mathcal{A}_{\text{sh}}\cap
L^{1}\left(  Y,d\nu\right)  $ may not satisfy the submartingale inequality,
unless further informations are available. In our context, $\nu$ is generated
by a stopped Brownian motion which has initial distribution $\bar{\mu}$ and
stays in $Y,$ hence one might think that the assumption $\left(
\ref{thm_dual_sh_attL00}\right)  $ is redundant. However, there are examples
showing the submartingale inequality might fail even though the Brownian
motion is stopped before it exits $Y$. Here is an example due to Zhen-Qing
Chen. Let $Y$ be the 2-dimensional unit ball. Up to a conformal transform, we
can reduce $Y$ to the upper half space $H=\left\{  z=\left(  x,y\right)
\in\mathbb{R}^{2}:y>0\right\}  $ and consider the Poisson kernel for $H,$
\[
\psi\left(  z\right)  =\psi\left(  x,y\right)  =\frac{1}{\pi}\frac{y}%
{x^{2}+y^{2}}.
\]
Denote by $\tau_{a}$ the hitting time of the parabolic curve $\Gamma:y=ax^{2}$
($a>0$) by the Brownian motion $W_{z_{0}}\left(  t\right)  $ emanating from
$z_{0}=\left(  0,1\right)  \in\mathbb{R}^{2}$. Since the origin $\left(
0,0\right)  $ is polar for $W_{z_{0}}\left(  t\right)  $, it is almost surely
never reached. Therefore $\tau_{a}$ is strictly less than the exit time
$\tau_{H}$ of $H$. Note $\psi\left(  z\right)  \leqslant\frac{a}{\pi}$ on
$\Gamma.$ Hence%
\[
E\left(  \psi\left(  W_{z_{0}}\left(  \tau_{a}\right)  \right)  \right)
\leqslant\frac{a}{\pi}<\frac{1}{\pi}=E\left(  \psi\left(  W_{z_{0}}\left(
0\right)  \right)  \right)  \text{ for }a\text{ small,}%
\]
i.e. the submartingale inequality fails for the harmonic function $\psi.$ One
can also modify this example to make a similar counterexample where the
distribution of the stopped Brownian motion has a density with a positive
lower bound.
\end{remark}

\subsection{Forward projection}

This section deals with the dual attainment of the forward projection. In
contrast to Theorem \ref{thm_dual_sh_bw},\ we will get a Lipschitz continuous
optimal dual solution for the forward projection.

Recall that, for convex cost $c$, the $c$-transforms $Q_{c}\left(
\cdot\right)  $ and $Q_{\bar{c}}\left(  \cdot\right)  $ defined by $\left(
\ref{eq_Qc}\right)  \left(  \ref{eq_Qcbar}\right)  $ preserves convexity. This
has played a crucial role in our proof of the dual attainment in the convex
order case. However, analogous preservation property does not hold for
subharmonic functions. To address this issue we introduce a composite of the
$c$-transforms with the operation of taking subharmonic envelopes.

Given $g\in C\left(  X\right)  $ on a connected open set $X$ of $\mathbb{R}%
^{d},$ its subharmonic envelope is defined as%
\begin{equation}
g_{e}\triangleq\sup\left\{  \psi:\psi\text{ subharmonic, }\psi\leqslant
g\right\}  .\label{SHenv}%
\end{equation}
It is well-known that $g_{e}\in C\left(  X\right)  $. If $X$ is convex and
$g\in C\left(  \bar{X}\right)  ,$ then $g_{e}\in C\left(  \bar{X}\right)  $.

\begin{definition}
\label{def_Qce}For a function $g$, we write $Q_{ce}\left(  g\right)  $ as the
subharmonic envelope of $Q_{c}\left(  g\right)  .$
\end{definition}

\begin{lemma}
\label{lm_Qcecbar}Let $a\in\mathbb{R},$ the cost $c\left(  x,y\right)  $ is
bounded from below and admit a modulus of continuity$.$ Assume that all
functions as a result of the $c$-transforms are defined on bounded
domains.\ Then for any subharmonic function $\psi$ on a bounded open set
$X\subset\mathbb{R}^{d}$ with values in $\left[  a,\infty\right)  ,$%
\[
Q_{\bar{c}}\left(  \psi\right)  =Q_{\bar{c}}\left(  Q_{ce}\left(  Q_{\bar{c}%
}\left(  \psi\right)  \right)  \right)  ,\text{ }Q_{ce}\left(  \psi\right)
=Q_{ce}Q_{\bar{c}}\left(  Q_{ce}\left(  \psi\right)  \right)  .
\]

\end{lemma}

\begin{proof}
Assume without loss of generality that the cost $c$ is nonnegative. For any
function $g\geqslant a$ on $X$, it is easy to see that $Q_{c}\left(  g\right)
\geqslant a$ and $Q_{\bar{c}}\left(  g\right)  $ is bounded from below, so it
makes sense to consider subharmonic envelopes of these functions as a result
of $Q_{c}\left(  \cdot\right)  $ and $Q_{\bar{c}}\left(  \cdot\right)  $
transforms. Recalling $\left(  \ref{ineq_Qccbar1}\right)  \left(
\ref{ineq_Qccbar2}\right)  ,$ we have
\[
Q_{\bar{c}}\left(  Q_{c}\left(  \psi\right)  \right)  \leqslant\psi,\text{
}Q_{c}\left(  Q_{\bar{c}}\left(  \psi\right)  \right)  \geqslant\psi.
\]
It follows that%
\[
\left(  \text{i}\right)  \text{ }Q_{\bar{c}}\left(  Q_{ce}\left(  \psi\right)
\right)  \leqslant\psi,\text{ }\left(  \text{ii}\right)  \text{ }Q_{ce}\left(
Q_{\bar{c}}\left(  \psi\right)  \right)  \geqslant\psi.
\]
Note in obtaining $\left(  \text{ii}\right)  $ the fact that $\psi$ is
subharmonic is used, while in obtaining $\left(  \text{i}\right)  $ the
subharmonicity of $\psi$ is \textit{not} used. Now applying $Q_{\bar{c}%
}\left(  \cdot\right)  $ on both sides of $\left(  \text{ii}\right)  ,$ we
have%
\[
Q_{\bar{c}}\left(  Q_{ce}\left(  Q_{\bar{c}}\left(  \psi\right)  \right)
\right)  \geqslant Q_{\bar{c}}\left(  \psi\right)  .
\]
Since $Q_{\bar{c}}\left(  \psi\right)  $ is bounded from below, we may replace
$\psi$ with $Q_{\bar{c}}\left(  \psi\right)  $ in $\left(  \text{i}\right)  $
and obtain%
\[
Q_{\bar{c}}\left(  Q_{ce}\left(  Q_{\bar{c}}\left(  \psi\right)  \right)
\right)  \leqslant Q_{\bar{c}}\left(  \psi\right)  .
\]
Therefore%
\[
Q_{\bar{c}}\left(  \psi\right)  =Q_{\bar{c}}\left(  Q_{ce}\left(  Q_{\bar{c}%
}\left(  \psi\right)  \right)  \right)  .
\]
Simiarly applying $Q_{ce}\left(  \cdot\right)  $ on both sides of $\left(
\text{i}\right)  ,$
\[
Q_{ce}Q_{\bar{c}}\left(  Q_{ce}\left(  \psi\right)  \right)  \leqslant
Q_{ce}\left(  \psi\right)  ,
\]
Since $Q_{ce}\left(  \psi\right)  $ is subharmonic, $Q_{ce}\left(
\psi\right)  $ cannot take $\infty,$ we may then replace $\psi$ with
$Q_{ce}\left(  \psi\right)  $ in $\left(  \text{ii}\right)  $ to get%
\[
Q_{ce}Q_{\bar{c}}\left(  Q_{ce}\left(  \psi\right)  \right)  \geqslant
Q_{ce}\left(  \psi\right)  .
\]
Therefore%
\[
Q_{ce}\left(  \psi\right)  =Q_{ce}Q_{\bar{c}}\left(  Q_{ce}\left(
\psi\right)  \right)  .
\]

\end{proof}

Lemma \ref{lm_Qcecbar} is crucial for the dual attainment for forward
subharmonic order projection. It enables us to use a verison of double
convexification trick. Note that the optimal solution in this case is
Lipschitz and subharmonic.

\begin{theorem}
\label{thm_dual_sh_fw}Let $X,$ $Y$ be bounded convex open subsets of
$\mathbb{R}^{d},$ $\mu\in P^{ac}\left(  X\right)  $ and $\nu\in P\left(
Y\right)  $.\ Then there exists $\bar{\psi}_{0}\in\mathcal{A}_{\text{sh}}\cap
Lip\left(  X\right)  $ which achieves the supremum of the dual value,%
\[
\mathcal{D}_{2}\left(  \mathrm{P}_{\mu\leqslant}^{\text{sh}},\nu\right)
\triangleq\sup_{\psi\in\mathcal{A}_{\text{sh}}\cap C_{b}\left(  X\right)
}\left\{  \int_{X}\psi d\mu-\int_{Y}Q_{\bar{2}}\left(  \psi\right)
d\nu\right\}  .
\]

\end{theorem}

\begin{proof}
In view of Lemma \ref{lm_Qcecbar} and Theorem \ref{thm_dual_sh},%
\begin{align*}
\mathcal{D}_{2}\left(  \mathrm{P}_{\mu\leqslant}^{\text{sh}},\nu\right)   &
=\sup_{\psi\in\mathcal{A}_{\text{sh}}\cap C_{b}}\left\{  \int_{X}\psi
d\mu-\int_{Y}Q_{\bar{2}}\left(  Q_{2e}\left(  Q_{\bar{2}}\left(  \psi\right)
\right)  \right)  d\nu\right\} \\
&  \leqslant\sup_{\psi\in\mathcal{A}_{\text{sh}}\cap C_{b}}\left\{  \int%
_{X}Q_{2e}\left(  Q_{\bar{2}}\left(  \psi\right)  \right)  d\mu-\int%
_{Y}Q_{\bar{2}}\left(  Q_{2e}\left(  Q_{\bar{2}}\left(  \psi\right)  \right)
\right)  d\nu\right\} \\
&  =\sup_{\bar{\psi}\in\mathcal{A}_{\text{sh}}^{1}\cap C_{b}}\left\{  \int%
_{X}\bar{\psi}d\mu-\int_{Y}Q_{\bar{2}}\left(  \bar{\psi}\right)  d\nu\right\}
\leqslant\mathcal{D}_{2}\left(  \mathrm{P}_{\mu\leqslant}^{\text{sh}}%
,\nu\right)  ,
\end{align*}
where%
\[
\mathcal{A}_{\text{sh}}^{1}=\left\{  \bar{\psi}:\bar{\psi}=Q_{2e}\left(
Q_{\bar{2}}\left(  \psi\right)  \right)  \text{ for some }\psi\in
\mathcal{A}_{\text{sh}}\right\}  \subset\mathcal{A}_{\text{sh}}.
\]
Let $\bar{\psi}_{n}\in\mathcal{A}_{\text{sh}}^{1}\cap\mathrm{S}_{b}$ be a
maximizing sequence$,$ i.e. $\bar{\psi}_{n}=Q_{2e}\left(  Q_{\bar{2}}\left(
\psi_{n}\right)  \right)  $ for some $\psi_{n}\in\mathcal{A}_{\text{sh}}$.
Note that, similar to Theorem \ref{thm_dual_sh_bw}, $Q_{\bar{2}}\left(
\psi_{n}\right)  $ and $Q_{2}\left(  Q_{\bar{2}}\left(  \psi_{n}\right)
\right)  $ are Lipschitz continuous with a common Lipschitz constant inherited
from the cost function. In view of \cite[Theorem 2]{caffarelli1998obstacle} or
\cite{ruan2021reg}, $\bar{\psi}_{n}$ and $Q_{\bar{2}}\left(  \bar{\psi}%
_{n}\right)  $ are Lipschitz continuous with a common Lipschitz constant$.$
Now we can assume without loss of generality that%
\[
\max_{x}\bar{\psi}_{n}\left(  x\right)  =0,\text{ }\forall n.
\]
Therefore the sequence $\bar{\psi}_{n},$ and thus $Q_{\bar{2}}\left(
\bar{\psi}_{n}\right)  ,$ are uniformly bounded and equicontinuous$.$ By
Arzela-Ascoli theorem, we may suppose that it holds uniformly that%
\[
\bar{\psi}_{n}\rightarrow\bar{\psi}_{0},\text{ }Q_{\bar{2}}\left(  \bar{\psi
}_{n}\right)  \rightarrow Q_{\bar{2}}\left(  \bar{\psi}_{0}\right)  ,
\]
where $\bar{\psi}_{0}$ is a Lipschitz continuous subharmonic function. It
follows that%
\[
\int_{Y}Q_{\bar{2}}\left(  \bar{\psi}_{n}\right)  \left(  y\right)
d\nu\rightarrow\int_{Y}Q_{\bar{2}}\left(  \bar{\psi}_{0}\right)  \left(
y\right)  d\nu.
\]
Therefore%
\begin{align*}
\mathcal{D}_{2}\left(  \mathrm{P}_{\mu\leqslant}^{\text{sh}},\nu\right)   &
=\lim_{n\rightarrow\infty}\left\{  \int_{X}\bar{\psi}_{n}\left(  x\right)
d\mu-\int_{Y}Q_{\bar{2}}\left(  \bar{\psi}_{n}\right)  \left(  y\right)
d\nu\right\} \\
&  =\int_{X}\bar{\psi}_{0}\left(  x\right)  d\mu-\int_{Y}Q_{\bar{2}}\left(
\bar{\psi}_{0}\right)  \left(  y\right)  d\nu\leqslant\mathcal{D}_{2}\left(
\mathrm{P}_{\mu\leqslant}^{\text{sh}},\nu\right)  .
\end{align*}
The last inequality uses the fact that $\bar{\psi}_{0}\in\mathcal{A}%
_{\text{sh}}\cap C_{b}$. Thus $\bar{\psi}_{0}$ is an optimal dual solution,
and, if we repeat the inequalities at the beginning of the proof, then we can
see that $\bar{\psi}_{0}$ satisfies $\bar{\psi}_{0}=Q_{2e}\left(  Q_{\bar{2}%
}\left(  \bar{\psi}_{0}\right)  \right)  .$
\end{proof}

It seems a similar proof cannot be applied to the backward dual attainment,
using the operator $Q_{\bar{2}e}\left(  \cdot\right)  $ defined similarly to
$Q_{2e}\left(  \cdot\right)  $ (Definition \ref{def_Qce}) To make the trick
work, we will need the idenity $Q_{2}\left(  \cdot\right)  =Q_{2}\left(
Q_{\bar{2}e}\left(  Q_{2}\left(  \cdot\right)  \right)  \right)  .$ However,
in general, we only have $Q_{2}\left(  \cdot\right)  =Q_{2}\left(  Q_{\bar{2}%
}\left(  Q_{2}\left(  \cdot\right)  \right)  \right)  .$

\section{\label{sec_char_sh}Characterization of subharmonic order projections}

In parallel with section \ref{sec_char_cx}, we show that the optimal mappings
for backward and forward subharmonic order projecitons are characterized by
Laplacian contraction and Laplacian expansion.

\begin{definition}
[\textbf{Laplacian contraction}]\label{def_Lap_contraction}Let $\psi$ be a
proper lower semicontinuous convex function. We call $\psi$ a Laplacian
contraction if $\psi=\phi^{\ast}$ for some proper function $\phi$ such that
$\phi$ is bounded from below and $\Delta\phi\geqslant d$ in the sense of distribution.
\end{definition}

\begin{definition}
[\textbf{Laplacian expansion}]\label{def_Lap_expansion}Let $\psi$ be a proper
lower semicontinuous convex function. We call $\psi$ a Laplacian expansion if
$\psi=\phi^{\ast}$ for some proper function $\phi$ such that $\phi$ is bounded
from below and $\Delta\phi\leqslant d$ in the sense of distribution.
\end{definition}

\begin{remark}
\label{rmk_Lap}Note the definition of Laplacian contraction and Laplacian
expansion are consistent with that of the convex contraction and convex
expansion defined in section \ref{sec_char_cx}, and the use of Legendre
transform instead of the function itself is essential for subharmonic order
case. While $D^{2}\phi\geqslant Id$ implies $D^{2}\phi^{\ast}\leqslant Id$
(Lemma \ref{lm_HJ_cx_duality})$,$ it is not generally true that $\Delta
\phi\geqslant Id$ implies $\Delta\phi^{\ast}\leqslant Id.$
\end{remark}

Laplacian expansion has a simple but significant geometric consequence.

\begin{lemma}
\label{lm_vol_exp}Let $\psi$ be a Laplacian expansion. Then the map
$\nabla\psi$ satisfies%
\[
\det D^{2}\psi\left(  x\right)  \geqslant1,\text{ }a.e.\text{ }x.
\]
This explains the term \textbf{expansion}, since the mapping $\nabla\psi$
increases volumes. Moreover, if $\nu$ absolutely continuous w.r.t. the
Lebesgue measure, then so is $\bar{\nu}\triangleq\left(  \nabla\psi\right)
_{\#}\nu,$ in addition, (still using $\nu,$ $\bar{\nu}$ to denote their
respective densities),
\[
\bar{\nu}\left(  \nabla\psi\left(  x\right)  \right)  \leqslant\nu\left(
x\right)  ,\text{ }a.e.\text{ }x.
\]

\end{lemma}

\begin{proof}
Let $\phi$ be a function such that $\phi$ is bounded from below, $\psi
=\phi^{\ast}$ and $\Delta\phi\leqslant d$.

\textbf{1}. First consider the case where $\phi$ is convex. An important
observation is that from the arithmetic-geometric inequality, we have for
convex $\phi$,
\[
\left(  \det D^{2}\phi\left(  x\right)  \right)  ^{1/d}\leqslant\frac{1}%
{d}\Delta\phi\left(  x\right)  ,\text{ }a.e.\text{ }x.
\]
Here we used the almost second-order differentiability of convex functions.
Then $\Delta\phi\leqslant d$ implies%
\begin{equation}
\left(  \det D^{2}\phi\left(  x\right)  \right)  ^{1/d}\leqslant1,\text{
}a.e.\text{ }x. \label{lm_vol_incr1}%
\end{equation}
Recall that $\left(  \nabla\psi\right)  _{\#}\nu=\bar{\nu}$ is equivalent to%
\[
\bar{\nu}\left(  E\right)  =\int_{\left(  \nabla\psi\right)  ^{-1}E}%
d\nu\left(  x\right)  \text{ for any measurable }E,
\]
where $\left(  \nabla\psi\right)  ^{-1}\left(  E\right)  $ denotes the inverse
image of $E$. We infer from $\left(  \ref{lm_vol_incr1}\right)  $ that
$\bar{\nu}\left(  E\right)  \geqslant\nu\left(  F\right)  $. If $\bar{\nu}$ is
absolutely continuous, then so is $\nu$. Still using $\nu,$ $\bar{\nu}$ to
denote their respective density, we then have%
\[
\bar{\nu}\left(  \nabla\psi\left(  x\right)  \right)  \leqslant\nu\left(
x\right)  ,\text{ }a.e.\text{ }x,
\]
which means the target density gets smaller than the source under the mapping
$\nabla\psi$.

\textbf{2}. Now consider the case where the convexity of $\phi$ is not
sassumed. Note, the function $\phi^{\ast\ast}=\psi^{\ast}$ is convex and we
have
\[
\Delta\phi^{\ast\ast}\left(  x\right)  \leqslant\Delta\phi\left(  x\right)
\text{ on }S\triangleq\left\{  x:\phi\left(  x\right)  =\phi^{\ast\ast}\left(
x\right)  \right\}  .
\]
Moreover, $\nabla\phi^{\ast\ast}$ (which is nothing but $\nabla\psi^{\ast}$)
is the inverse map (almost surely) to $\nabla\psi$. If a point $x$ is mapped
by $\nabla\psi$ (the optimal map from $\nu$ to $\bar{\nu}$) outside of the
contact set $S$, then the point $x$ belongs to the set where $\psi$ is not
differentiable, and this set has zero mass, so zero mass under $\nu$ by
absolute continuity. Therefore, what happens outside the contact set $S$ does
not affect the densities. As a result, we can replace $\phi$ with $\phi
^{\ast\ast}$ in step 1 and follow the same argument there to conclude the same
result for $\psi$.
\end{proof}

A similar argument does not work in the Laplacian contraction case, since we
cannot use the arithmetic-geometric inequality as above$.$

\subsection{Backward projection}

\begin{theorem}
\label{thm_sh_prop}Let $X,$ $Y$ be bounded convex open subsets of
$\mathbb{R}^{d},$ $\mu\in P\left(  X\right)  $ and $\nu\in P^{ac}\left(
Y\right)  $. Assume that the conditions of Theorem \ref{thm_dual_sh_bw} are
satisfied and $\psi\in\mathcal{A}_{\text{sh}}\cap L^{1}\left(  Y,d\nu\right)
$ is the optimal dual solution for $\mathcal{D}_{2}\left(  \mu,\mathrm{P}%
_{\leqslant\nu}^{\text{sh}}\right)  $ obtained there. Then the unique
projection $\bar{\mu}$\ of $\mu$ onto $\mathrm{P}_{\leqslant\nu}^{\text{sh}}$
is given by $\left(  \nabla\psi_{0}^{\ast}\right)  _{\#}\mu$, where%
\[
\psi_{0}\left(  y\right)  =\frac{1}{2}\left\vert y\right\vert ^{2}+\frac{1}%
{2}\psi\left(  y\right)  .
\]

\end{theorem}

\begin{proof}
We skip the proof, since it is similar to Theorem \ref{thm_cx_prop}. Only note
that since $Q_{2}\left(  \psi\right)  \in L^{1}\left(  X,d\mu\right)  ,$ it is
finite for $\mu$-$a.e.$ $x.$ Hence $\psi_{0}^{\ast}$ is finite for $\mu
$-$a.e.$ $x$. Then we can use the absolute continuity of $\mu$ to infer the
desired conclusion.
\end{proof}

\begin{theorem}
\label{thm_sh_equiv}Let $X,$ $Y$ be bounded convex open subsets of
$\mathbb{R}^{d},$ $\mu\in P\left(  X\right)  $ and $\nu\in P^{ac}\left(
Y\right)  $. Assume that the conditions of Theorem \ref{thm_dual_sh_bw} are
satisfied. Consider the following statements.

(i) The projection of $\mu$ onto $\mathrm{P}_{\leqslant\nu}^{\text{sh}}$ is
$\nu.$

(ii) There is a Laplacian contraction $\phi$\ such that $\left(  \nabla
\phi\right)  _{\#}\mu=\nu$.

\noindent We always have that (i) implies (ii). If $\phi$ in (ii) is such that
$\phi^{\ast}\in C_{b}$ and $\Delta\phi^{\ast}\geqslant d,$ then (ii) implies (i).
\end{theorem}

\begin{proof}
That (i) implies (ii) follows from \ref{thm_sh_prop}. To see the reverse
implication, we note by definition of Laplacian contraction, there is a
function $\xi$ such that $\xi$ is bounded from below, $\Delta\xi\geqslant d$
and $\phi=\xi^{\ast}.$ Define%
\[
\psi\left(  y\right)  \triangleq2\phi^{\ast}\left(  y\right)  -\left\vert
y\right\vert ^{2}=2\xi^{\ast\ast}\left(  y\right)  -\left\vert y\right\vert
^{2}.
\]
Then $\psi$ is continuous and subharmonic by assumption. Using the convexity
of $\phi$ we can write
\[
\nabla\phi\left(  x\right)  \cdot x=\phi^{\ast}\left(  \nabla\phi\left(
x\right)  \right)  +\phi\left(  x\right)  ,\text{ }a.e.\text{ }x.
\]
Therefore, similar to Theorem \ref{thm_cx_exp}, we have%
\begin{align*}
\mathcal{T}_{2}\left(  \mu,\nu\right)   &  \leqslant\int\left(  \left\vert
x\right\vert ^{2}-2\phi\left(  x\right)  \right)  d\mu-\int\left(  2\phi
^{\ast}\left(  y\right)  -\left\vert y\right\vert ^{2}\right)  d\nu\\
&  =\int Q_{2}\left(  \psi\right)  \left(  x\right)  d\mu-\int\psi\left(
y\right)  d\nu\text{ (}\psi\text{ is subharmonic by (ii))}\\
&  \leqslant\sup_{g\in\mathcal{A}_{\text{sh}}\cap C_{b}}\left\{  \int
Q_{2}\left(  g\right)  \left(  x\right)  d\mu-\int g\left(  y\right)
d\nu\right\}  \leqslant\mathcal{T}_{2}\left(  \mu,\nu\right)  ,
\end{align*}
which shows that $\psi$ is an optimizer of $\mathcal{D}_{2}\left(
\mu,\mathrm{P}_{\leqslant\nu}^{\text{sh}}\right)  $. Then the desired
conclusion follows.
\end{proof}

\subsection{Forward projection}

\begin{theorem}
\label{thm_sh_prop_fw}Let $X,$ $Y$ be bounded convex open subsets of
$\mathbb{R}^{d},$ $\mu\in P^{ac}\left(  X\right)  $ and $\nu\in P\left(
Y\right)  $. Assume that $\psi\in\mathcal{A}_{\text{sh}}\cap Lip\left(
X\right)  $ is the optimal dual solution for $\mathcal{D}_{2}\left(
\mathrm{P}_{\mu\leqslant}^{\text{sh}},\nu\right)  $ obtained in Theorem
\ref{thm_dual_sh_fw}. Then the unique projection of $\nu$ onto $\mathrm{P}%
_{2,\mu\leqslant}^{\text{sh}}$ is given by $\left(  \nabla\bar{\psi}_{0}%
^{\ast}\right)  _{\#}\nu$, where%
\[
\bar{\psi}_{0}\left(  x\right)  =\frac{1}{2}\left\vert x\right\vert ^{2}%
-\frac{1}{2}\psi\left(  x\right)  .
\]

\end{theorem}

We omit the proof, since it is similar to Theorem \ref{thm_cx_prop}.

\begin{theorem}
\label{thm_sh_equiv_fw}Let $X,$ $Y$ be bounded convex open subsets of
$\mathbb{R}^{d},$ $\mu\in P^{ac}\left(  X\right)  $ and $\nu\in P\left(
Y\right)  $. Consider the following statements.

(i) The projection of $\nu$ onto $\mathrm{P}_{\mu\leqslant}^{\text{sh}}$ is
$\mu.$

(ii) There is a Laplacian expansion $\phi$\ such that $\left(  \nabla
\phi\right)  _{\#}\nu=\mu$.

\noindent We always have that (i) implies (ii). If $\phi$ in (ii) is such that
$\phi^{\ast}\in C_{b}$ and $\Delta\phi^{\ast}\leqslant d,$ then (ii) implies (i).
\end{theorem}

\begin{proof}
That (i) implies (ii) follows from \ref{thm_sh_prop_fw}. To see the reverse
implication, we note by definition of Laplacian expansion, there is a function
$\eta$ such that $\eta$ is bounded from below, $\Delta\eta\leqslant d$ and
$\phi=\eta^{\ast}.$ Define%
\[
\psi\left(  x\right)  \triangleq\left\vert x\right\vert ^{2}-2\phi^{\ast
}\left(  x\right)  =\left\vert x\right\vert ^{2}-2\eta^{\ast\ast}\left(
x\right)  .
\]
Then the rest of the proof follows as in Theorem \ref{thm_sh_equiv}.
\end{proof}

\appendix

\section{$c$-transforms}

We recall a few properties of the $c$-transforms $Q_{c}\left(  \cdot\right)  $
and $Q_{\bar{c}}\left(  \cdot\right)  $:%
\[
Q_{c}\left(  g\right)  \left(  x\right)  =\inf_{y\in Y}\left\{  g\left(
y\right)  +c\left(  x,y\right)  \right\}  ,\text{ }Q_{\bar{c}}\left(
g\right)  \left(  y\right)  =\sup_{x\in X}\left\{  g\left(  x\right)
-c\left(  x,y\right)  \right\}
\]
It is well-known that $Q_{c}\left(  \cdot\right)  $ preserves Lipschitz
continuity, convexity and concavity. Since%
\[
Q_{\bar{c}}\left(  g\right)  =-Q_{c}\left(  -g\right)  .
\]
$Q_{\bar{c}}$ has the same properties as $Q_{c}.$ Moreover, for any function
$g:Y\mapsto\mathbb{R}\cup\left\{  \infty\right\}  ,$%
\begin{equation}
Q_{\bar{c}}\left(  Q_{c}\left(  g\right)  \right)  \leqslant g,
\label{ineq_Qccbar1}%
\end{equation}
and for any function $g:X\mapsto\mathbb{R}\cup\left\{  -\infty\right\}  ,$%
\begin{equation}
Q_{c}\left(  Q_{\bar{c}}\left(  g\right)  \right)  \geqslant g.
\label{ineq_Qccbar2}%
\end{equation}
In particular,%
\begin{equation}
Q_{c}\left(  g\right)  =Q_{c}\left(  Q_{\bar{c}}\left(  Q_{c}\left(  g\right)
\right)  \right)  ,\text{ }Q_{\bar{c}}\left(  g\right)  =Q_{\bar{c}}\left(
Q_{c}\left(  Q_{\bar{c}}\left(  g\right)  \right)  \right)  .
\label{eq_Qccbar12}%
\end{equation}

When $c\left(  x,y\right)  =\left\vert x-y\right\vert ^{2}$, the
$c$-transforms are written $Q_{2}\left(  \cdot\right)  ,$ $Q_{\bar{2}}\left(
\cdot\right)  $.

\begin{lemma}
\label{lm_Q2Q2bar}Let $g$ be a function defined in a subset $\Omega$ of
$\mathbb{R}^{d}.$ The following identities hold.

(i) For $x\in\mathbb{R}^{d},$
\[
Q_{2}\left(  g\right)  \left(  x\right)  =\left\vert x\right\vert ^{2}%
-2g_{0}^{\ast}\left(  x\right)  ,\text{ where }g_{0}\left(  y\right)
=\frac{1}{2}\left\vert y\right\vert ^{2}+\frac{1}{2}g\left(  y\right)  .
\]

(ii) For $y\in\mathbb{R}^{d},$%
\[
Q_{\bar{2}}\left(  g\right)  \left(  y\right)  =2\bar{g}_{0}^{\ast}\left(
y\right)  -\left\vert y\right\vert ^{2},\text{ where }\bar{g}_{0}\left(
x\right)  =\frac{1}{2}\left\vert x\right\vert ^{2}-\frac{1}{2}g\left(
x\right)  .
\]

\end{lemma}

\begin{proof}
Straightforward calculations yield%
\begin{align*}
Q_{2}\left(  g\right)  \left(  x\right)   &  =\inf_{y\in\Omega}\left\{
g\left(  y\right)  +\left\vert x-y\right\vert ^{2}\right\}  =\inf_{y\in\Omega
}\left\{  g\left(  y\right)  +\left\vert y\right\vert ^{2}-2x\cdot y\right\}
+\left\vert x\right\vert ^{2}\\
&  =\left\vert x\right\vert ^{2}-2\sup_{y\in\Omega}\left\{  x\cdot y-\left(
\frac{1}{2}\left\vert y\right\vert ^{2}+\frac{1}{2}g\left(  y\right)  \right)
\right\} \\
&  =\left\vert x\right\vert ^{2}-2g_{0}^{\ast}\left(  x\right)  .
\end{align*}
Similarly%
\begin{align*}
Q_{\bar{2}}\left(  g\right)  \left(  y\right)   &  =\sup_{x\in\Omega}\left\{
g\left(  x\right)  -\left\vert x-y\right\vert ^{2}\right\}  =\sup_{x\in\Omega
}\left\{  g\left(  x\right)  -\left\vert x\right\vert ^{2}+2x\cdot y\right\}
-\left\vert y\right\vert ^{2}\\
&  =2\sup_{x\in\Omega}\left\{  x\cdot y-\left(  \frac{1}{2}\left\vert
x\right\vert ^{2}-\frac{1}{2}g\left(  x\right)  \right)  \right\}  -\left\vert
y\right\vert ^{2}\\
&  =2\bar{g}_{0}^{\ast}\left(  y\right)  -\left\vert y\right\vert ^{2}.
\end{align*}

\end{proof}

\begin{lemma}
\label{lm_HJ_cx_duality}Let $g:\mathbb{R}^{d}\mapsto\mathbb{R}$ be a function.

(i) If $\frac{1}{2}\left\vert x\right\vert ^{2}-g\left(  x\right)  $ is
convex, then $g^{\ast}\left(  y\right)  -\frac{1}{2}\left\vert y\right\vert
^{2}$ is convex.

(ii) If $g\left(  y\right)  -\frac{1}{2}\left\vert y\right\vert ^{2}$ is
convex, then $\frac{1}{2}\left\vert x\right\vert ^{2}-g^{\ast}\left(
x\right)  $ is convex.

(iii) If $g$ is a lower semicontinuous proper convex function, then%
\[
\frac{1}{2}\left\vert x\right\vert ^{2}-g\left(  x\right)  \text{ is convex if
and only if }g^{\ast}\left(  y\right)  -\frac{1}{2}\left\vert y\right\vert
^{2}\text{ is convex.}%
\]
In either case of (iii), we have $g\in C^{1}.$ In matrix form, (iii) can be
rewritten as: $D^{2}g\leqslant Id$ if and only if $D^{2}g^{\ast}\geqslant Id$.
Second order derivatives of convex functions are understood in distributional
sense and $Id$ is the $d\times d$ identity matrix.
\end{lemma}

\begin{proof}
(i) and (ii) follow from the straighforward calculations%
\[
Q_{\bar{2}}\left(  \left[  \left\vert x\right\vert ^{2}-2g\right]  \right)
\left(  y\right)  =\sup_{x}\left\{  \left\vert x\right\vert ^{2}-2g\left(
x\right)  -\left\vert x-y\right\vert ^{2}\right\}  =2g^{\ast}\left(  y\right)
-\left\vert y\right\vert ^{2},
\]
and%
\[
Q_{2}\left(  \left[  2g-\left\vert y\right\vert ^{2}\right]  \right)  \left(
x\right)  =\inf_{y}\left\{  2g\left(  y\right)  -\left\vert y\right\vert
^{2}+\left\vert x-y\right\vert ^{2}\right\}  =\left\vert x\right\vert
^{2}-2g^{\ast}\left(  x\right)  .
\]
If $g$ is a lower semicontinuous proper convex function, then $g=g^{\ast\ast
}.$ So (iii) follows from (ii). If either case of (iii) is true, then%
\[
g^{\ast}\left(  y\right)  =\left(  g^{\ast}\left(  y\right)  -\frac{1}%
{2}\left\vert y\right\vert ^{2}\right)  +\frac{1}{2}\left\vert y\right\vert
^{2}%
\]
is uniformly convex, hence $g=g^{\ast\ast}\in C^{1}$.
\end{proof}

\section{Convergence of subharmonic functions}

The following result is adapted from \cite[Theorem 4.1.9]%
{hormander1980analysis} to serve its purpose in our setting.

\begin{lemma}
\label{lm_convg_subH}Let $X$ be a connected open subset of $\mathbb{R}^{d}$
and $\psi_{n}$ a sequence of subharmonic functions such that $\int%
_{X}\left\vert \psi_{n}\right\vert dx$ is bounded by some constant $A>0$. Then
up to a subsequence $\psi_{n}$ converges almost surely to a subharmonic
function $\psi\in L^{1}\left(  X,dx\right)  $ in the usual sense of Definition
\ref{def_subH}.
\end{lemma}

\begin{proof}
\textbf{1}. By identifying $L^{1}\left(  X,dx\right)  $ with a subset of
$M\left(  X\right)  ,$ we may assume that $\psi_{n}$ converges weakly to some
$\xi\in M\left(  X\right)  .$ Since $\psi_{n}$ is subharmonic, $\Delta\psi
_{n}$ is a nonnegative distribution, thus a nonnegative measure by
\cite[Theorem 2.1.7]{hormander1980analysis}. Hence we may also assume that
$\Delta\psi_{n}$ converges weakly to some nonnegative distribution
$\eta\geqslant0.$ Then for any $\phi\in C_{0}^{\infty}\left(  X\right)  ,$%
\[
\int_{X}\Delta\phi d\xi\left(  x\right)  =\lim_{n\rightarrow\infty}\int%
_{X}\Delta\phi\psi_{n}dx=\lim_{n\rightarrow\infty}\int_{X}\phi\Delta\psi
_{n}dx=\int_{X}\phi d\eta\left(  x\right)  .
\]
Therefore $\Delta\xi=\eta\geqslant0$ in the sense of distribution. By
\cite[Theorem 4.1.8]{hormander1980analysis}, $\xi$ is a subharmonic function
$\psi\in L_{loc}^{1}\left(  X,dx\right)  $ in the usual sense of Definition
\ref{def_subH}.

\textbf{2}. Let $B_{r}=\left\{  \left\vert x\right\vert \leqslant r\right\}
,$ $\delta>0$ and $0\leqslant\rho\in C_{0}^{\infty}\left(  \mathbb{R}%
^{d}\right)  $ be a radially symmetric function supported in $B_{1}$ such that
$\int\rho=1$. Given a compact set $K\subset X,$ we claim that, as
$n\rightarrow\infty,$%
\begin{equation}
\psi_{n}\ast\rho_{\delta}\left(  x\right)  \rightarrow\psi\ast\rho_{\delta
}\left(  x\right)  \text{ uniformly for }x\in K. \label{lm_convg_subH0}%
\end{equation}
By assumption%
\begin{equation}
\left\vert \int\psi_{n}\varphi dx\right\vert \leqslant A\sup_{x}\left\vert
\varphi\right\vert \text{, }\forall\varphi\in C_{0}^{\infty}\left(  K_{\delta
}\right)  , \label{lm_convg_subH1}%
\end{equation}
where $K_{\delta}$ is the compact set given by $K+B_{\delta}$.\ $\delta$ is
small so that $K_{\delta}\subset X$. Since $\psi_{n}$ converges to the
function $\psi$ in distribution, it follows that%
\begin{equation}
\left\vert \int\psi\varphi dx\right\vert \leqslant A\sup_{x}\left\vert
\varphi\right\vert \text{, }\forall\varphi\in C_{0}^{\infty}\left(  K_{\delta
}\right)  . \label{lm_convg_subH2}%
\end{equation}
Let $\epsilon>0.$ Since $K$ is compact, there is a finite net $K_{net}%
=\left\{  x_{j}:j=1,...,N\right\}  \subset K$ satisfying: for each $x\in K$
there is a $x_{j_{0}}\in K_{net}$ such that%
\[
2A\sup_{z\in\mathbb{R}^{d}}\left\vert \rho\left(  \frac{x-z}{\delta}\right)
-\rho\left(  \frac{x_{j_{0}}-z}{\delta}\right)  \right\vert \leqslant
\frac{\epsilon}{2}.
\]
This is feasible since $\rho$ is uniformly continuous on $\mathbb{R}^{d}.$ For
each $x_{j}\in K_{net},$ we have%
\[
\psi_{n}\ast\rho_{\delta}\left(  x_{j}\right)  =\int\psi_{n}\left(  z\right)
\rho\left(  \frac{x_{j}-z}{\delta}\right)  dz\rightarrow\int\psi\left(
z\right)  \rho\left(  \frac{x_{j}-z}{\delta}\right)  dz.
\]
Hence, there is $n_{0}$ such that%
\[
\left\vert \psi_{n}\ast\rho_{\delta}\left(  x_{j}\right)  -\psi\ast
\rho_{\delta}\left(  x_{j}\right)  \right\vert \leqslant\frac{\epsilon}%
{2}\text{ for }j=1,...,N,\text{ }n\geqslant n_{0}.
\]
Therefore, for $x\in K,$ $n\geqslant n_{0}$ and some $x_{j_{0}}\in K_{net}$
depending on $x,$%
\begin{align*}
&  \left\vert \psi_{n}\ast\rho_{\delta}\left(  x\right)  -\psi\ast\rho
_{\delta}\left(  x\right)  \right\vert \\
&  \leqslant\left\vert \psi_{n}\ast\rho_{\delta}\left(  x\right)  -\psi
\ast\rho_{\delta}\left(  x\right)  -\left(  \psi_{n}\ast\rho_{\delta}\left(
x_{j_{0}}\right)  -\psi\ast\rho_{\delta}\left(  x_{j_{0}}\right)  \right)
\right\vert +\frac{\epsilon}{2}\\
&  =\left\vert \int\left(  \psi_{n}\left(  z\right)  -\psi\left(  z\right)
\right)  \left(  \rho\left(  \frac{x-z}{\delta}\right)  -\rho\left(
\frac{x_{j_{0}}-z}{\delta}\right)  \right)  dz\right\vert +\frac{\epsilon}%
{2}\\
&  \leqslant2A\sup_{z\in\mathbb{R}^{d}}\left\vert \rho\left(  \frac
{x-z}{\delta}\right)  -\rho\left(  \frac{x_{j_{0}}-z}{\delta}\right)
\right\vert +\frac{\epsilon}{2}\leqslant\epsilon\text{.}%
\end{align*}
Note the second inequality is due to the fact that%
\[
z\mapsto\rho\left(  \frac{x-z}{\delta}\right)  -\rho\left(  \frac{x_{j_{0}}%
-z}{\delta}\right)
\]
is a function in $C_{0}^{\infty}\left(  K_{\delta}\right)  $ so that $\left(
\ref{lm_convg_subH1}\right)  \left(  \ref{lm_convg_subH2}\right)  $ apply.

\textbf{3}. Let $\varrho_{K}\in C_{0}^{\infty}\left(  X\right)  $ such that
$0\leqslant\varrho_{K}\leqslant1$ and $\varrho_{K}=1$ on $K.$ Let $\epsilon>0$
and $\delta>0$ be smaller than the distance between $\partial X$ and
$supp\left(  \varrho_{K}\right)  .$ By the definition of $\rho_{\delta}$ (in
step 2) and the subharmonicity,
\[
\psi_{n}\ast\rho_{\delta}\left(  x\right)  \geqslant\psi_{n}\left(  x\right)
,\text{ }\psi\ast\rho_{\delta}\left(  x\right)  \geqslant\psi\left(  x\right)
,\text{ }\forall x\in K.
\]
Together with $\left(  \ref{lm_convg_subH0}\right)  $, this implies that for
large $n,$%
\[
\psi\ast\rho_{\delta}\left(  x\right)  +\epsilon-\psi_{n}\left(  x\right)
>0,\text{ }\psi\ast\rho_{\delta}\left(  x\right)  +\epsilon-\psi\left(
x\right)  >0,\text{ }\forall x\in K.
\]
Then, for $x\in K,$%
\begin{align*}
\left\vert \psi_{n}\left(  x\right)  -\psi\left(  x\right)  \right\vert  &
\leqslant\left\vert \psi_{n}\left(  x\right)  -\psi\left(  x\right)
\right\vert \varrho_{K}\left(  x\right) \\
&  \leqslant\left\vert \psi\ast\rho_{\delta}\left(  x\right)  +\epsilon
-\psi_{n}\left(  x\right)  \right\vert \varrho_{K}\left(  x\right)
+\left\vert \psi\ast\rho_{\delta}\left(  x\right)  +\epsilon-\psi\left(
x\right)  \right\vert \varrho_{K}\left(  x\right) \\
&  =\left(  \psi\ast\rho_{\delta}\left(  x\right)  +\epsilon-\psi_{n}\left(
x\right)  \right)  \varrho_{K}\left(  x\right)  +\left(  \psi\ast\rho_{\delta
}\left(  x\right)  +\epsilon-\psi\left(  x\right)  \right)  \varrho_{K}\left(
x\right)  .
\end{align*}
Therefore%
\[
\limsup_{n\rightarrow\infty}\int_{K}\left\vert \psi_{n}-\psi\right\vert
dx\leqslant2\int_{X}\left(  \psi\ast\rho_{\delta}+\epsilon-\psi\right)
\varrho_{K}dx.
\]
Sending $\delta\rightarrow0,$ $\epsilon\rightarrow0,$ we obtain $\psi
_{n}\rightarrow\psi$ in $L^{1}\left(  K,dx\right)  .$ Since $K$ is an
arbitrary compact set, it follows that $\psi_{n}\rightarrow\psi$ in
$L_{loc}^{1}\left(  X,dx\right)  .$

\textbf{4}. Since the open set $X$ can be covered by a countable number of
closed balls, on each of these balls we may extract a subsequence of $\psi
_{n}$ which converges almost surely to $\psi.$ Utilizing a diagonal precedure,
we obtain that up to a subsequence $\psi_{n}$ converges to $\psi$ almost
surely. By Fatou lemma,%
\[
\liminf_{n\rightarrow\infty}\int_{X}\left\vert \psi_{n}\right\vert
dx\geqslant\int_{X}\left\vert \psi\right\vert dx,
\]
hence $\psi\in L^{1}\left(  X,dx\right)  .$
\end{proof}

\bibliographystyle{plainnat}
\bibliography{WStochOrderProj}
\end{document}